\pdfoutput=1
\documentclass[11pt, reqno]{amsart}

\usepackage{pdflscape}

\usepackage{natbib}
\usepackage[letterpaper, hmarginratio=1:1]{geometry}
\usepackage[hidelinks]{hyperref}                           
\usepackage{color}
\usepackage{bbm}
\usepackage{comment}
\usepackage{stmaryrd}
\usepackage{xcolor}
\usepackage{enumerate}
\usepackage{amssymb}
\usepackage{amsmath, amsthm}
\usepackage{cleveref}
\usepackage{algorithm}
\usepackage{algpseudocode}
\usepackage{nicematrix}
\usepackage{enumitem}
\usepackage{graphicx}

\usepackage[plain,noend,algo2e]{algorithm2e}
\newtheorem{algo}{Algorithm}

\makeatletter
\renewcommand{\algocf@captiontext}[2]{#1\algocf@typo. \AlCapFnt{}#2} % text of caption
\def\@algocf@capt@plain{top}
\renewcommand{\algocf@makecaption}[2]{%
  \addtolength{\hsize}{\algomargin}%
  \sbox\@tempboxa{\algocf@captiontext{#1}{#2}}%
  \ifdim\wd\@tempboxa >\hsize%     % if caption is longer than a line
    \hskip .5\algomargin%
    \parbox[t]{\hsize}{\algocf@captiontext{#1}{#2}}% then caption is not centered
  \else%
    \global\@minipagefalse%
    \hbox to\hsize{\box\@tempboxa}% else caption is centered
  \fi%
  \addtolength{\hsize}{-\algomargin}%
}
\makeatother

\theoremstyle{definition}
\newtheorem{ccounter}{ccounter}[section]
\newtheorem{cccounter}{cccounter}%[section]
\newtheorem{theorem}[ccounter]{Theorem}

\newtheorem{lemma}[ccounter]{Lemma}

\newtheorem{remark}[ccounter]{Remark}
\newtheorem{assumption}[cccounter]{Assumption}

\newcommand\E{\mathbb{E}}

\renewcommand\P{\mathbb{P}}
\newcommand\R{\mathbb{R}}

\newcommand\be{\begin{equation}}
\newcommand\ee{\end{equation}}
\def\ber{\color{blue}}
\def\eer{\normalcolor}
\def\one{{\mathbbm 1}}

\newcommand\Var{\operatorname{Var}}

\newcommand{\eps}{\varepsilon}
\renewcommand{\epsilon}{\varepsilon}

\renewcommand{\hat}{\widehat}

\newcommand{\bs}{} %previously was \boldsymbol
\newcommand{\onen}{[n]} %{\llbracket 1, n \rrbracket}

\newcommand{\cace}{\operatorname{LATE}}

\renewcommand{\phi}{\varphi}

\newcommand{\reg}{{\text{adj}}}
\newcommand{\Chi}{|z_{1-\alpha}|}
\newcommand{\nmc}{m}
\DeclareMathOperator*{\argmin}{arg\,min}

\title[Randomization-based Confidence Sets for the LATE]{Randomization-based Confidence Sets for the Local Average Treatment Effect}
\author{P.\ M.\ Aronow}
\author{Haoge Chang}
\author{Patrick Lopatto}

\begin{document}

\begin{abstract}
We consider the problem of generating confidence sets in randomized experiments with noncompliance. We show that a refinement of a randomization-based procedure proposed by Imbens and Rosenbaum (2005) has desirable properties. Namely, we show that using a studentized Anderson--Rubin-type statistic as a test statistic yields confidence sets that are finite-sample exact under treatment effect homogeneity, and remain asymptotically valid for the Local Average Treatment Effect when the treatment effect is heterogeneous. We provide a uniform analysis of this procedure and efficient algorithms to construct the confidence set.
%We consider the problem of generating confidence intervals in randomized experiments with noncompliance. We show that a refinement of a randomization-based procedure proposed by Imbens and Rosenbaum (2005) has desirable properties. Namely, we show that use of a studentizied Anderson--Rubin-type statistic as a test statistic yields confidence intervals that (i) are finite-sample exact under treatment effect homogeneity, and (ii) remain asymptotically valid for the Local Average Treatment Effect when the treatment effect is heterogeneous. \textcolor{orange}{Our analysis supports the use of such randomization-based procedures in some cases with low compliance rates. }
%and (iii) have uniform \ber convergence \eer guarantees. % that do not require that the association between treatment encouragement and take-up be strong. 
\end{abstract}
\maketitle

\section{Introduction}

In randomized experiments with noncompliance, the use of an instrumental variable (IV) estimator has become standard practice. When the treatment assignment is randomized, the IV estimator estimates the local average treatment effect (LATE), or the causal effect of the treatment on subjects that receive treatment only as a result of 
the randomized assignment %a randomized encouragement 
\citep{LATE1994}. 

The standard approach for generating confidence sets for the LATE relies on asymptotic approximations, namely estimating Wald-type confidence sets using a normal approximation and linearized variance estimates \citep{angrist2009mostly}. However, this approach is known to have poor finite-sample properties \citep{nelson1990,bound1995problems}. In particular, when the randomized %encouragement
 assignment 
does not strongly predict treatment take-up, an example of the weak instruments problem, these confidence sets may have poor coverage  %properties 
even in relatively large samples. The standard solution in the literature involves test inversion, typically with some variant of the Anderson--Rubin statistic \citep{anderson1949estimation} in the just-identified case. The validity of this procedure is often %typically 
%removing use of "typically" twice in close succession 
justified under a regression model with additive effects and random errors. Certain optimality properties of this test have been demonstrated in the simultaneous-equation model framework \citep{moreira2009tests}.

When analyzing a randomized experiment, an alternative approach relies on explicitly using randomization as a basis for inference. Notably, \cite{imbens2005robust}  proposed a procedure for generating finite-sample exact confidence sets under the assumption that the causal effects are 
constant and additive across subjects. %Under this assumption, 
The Imbens--Rosenbaum procedure proceeds by inverting Fisher’s exact test against 
a corresponding scalar-valued causal model.
%a scalar-valued causal model that allows the received treatment to have a constant effect. 
The primary benefit of this procedure is that it offers strong finite-sample guarantees (under a constant additive-effect assumption), thus avoiding the analytic problems that result from the use of asymptotic approximations. However, it generically offers no guarantees when treatment effects are allowed to be heterogeneous across subjects.

This paper shows that a refinement of the Imbens--Rosenbaum procedure has desirable properties in constructing confidence sets for the LATE. Namely, by choosing as a test statistic a variant of the Anderson--Rubin statistic (the studentized Anderson--Rubin statistic), we can guarantee that the resulting confidence sets are finite-sample valid under the constant additive-effect assumption, and that they remain asymptotically valid for the LATE under treatment effect heterogeneity. Although this test statistic is common (see, e.g., \cite{kang2018inference} for its use in randomization inference), the fact that it provides such a guarantee for the Imbens--Rosenbaum procedure appears to be novel in the literature. Furthermore, we show that this asymptotic guarantee is uniform with respect to instrument strength under mild conditions. In experiments with noncompliance, we can allow the instrument to be arbitrarily weak.
%\ber including designs with two-sided noncompliance and arbitrarily weak instruments\eer.
%, including nondegeneracy of the received treatment. 
In summary, our refinement to the Imbens--Rosenbaum procedure is finite-sample exact under the constant-effect assumption, and can be asymptotically valid for the LATE without assuming either a strong instrument or effect homogeneity. In particular, it is finite-sample exact under the sharp null of no causal effect for each subject.

We also extend our analysis to the regression-adjusted version of the procedure and introduce new exact algorithms for computing (the Monte Carlo versions of) the confidence sets. We note that these algorithms are useful beyond the randomization inference setting and can be applied with other resampling methods, such as the bootstrap, for studentized Anderson--Rubin statistics.

\subsection{Related Work}

Randomization inference for the LATE was studied in \cite{imbens2005robust}, \cite{kang2018inference}, \cite{keele2017randomization}, and \cite{nolen2011randomization}. \cite{kang2018inference} studies the studentized Anderson--Rubin statistic in the design-based framework and provides statistical and computational guarantees based on an asymptotic normal approximation; similar results were also demonstrated in \cite{li2017general}. In contrast, our procedure uses randomization-based critical values and offers a finite-sample guarantee.

%See also \cite{li2017general}. 
%%%I deleted the next part because I think we already said it, and it seems a bit redundant.
%\textcolor{red}{ \ber[I wonder if we can delete the next senIn contrast, we provide a statistical analysis based on the randomization distribution of the studentized Anderson--Rubin statistic, as originally proposed in \cite{imbens2005robust}, and provide efficient algorithms for implementation.   }

%In contrast, we provide a direct statistical analysis of the procedure in \cite{imbens2005robust} with the studentized Anderson--Rubin statistic, and provide efficient algorithms for implementing it. %Our results appear to be new in the literature.

%Effect
Heterogeneity-robust randomization inference was studied in \cite{chung2013exact, wu2021randomization, cohen2022gaussian,zhao2021covariate, tuvaandorj2024robust}.
Wald-type inference on the LATE in randomized experiments with covariate adjustments was considered in, for example, \cite{ansel2018ols,zhong2023inference,bai2023inference,bugni2023inference,ren2024model}.

%\textcolor{orange}{
We refer readers to \cite{stock2002survey}, \cite{andrews2019weak}, and \cite{keane2021practical} for surveys of the weak instruments literature. 
A standard analytical framework in this literature uses a constant-effect linear instrumental variable model \citep{staigerstock97,moreira2003conditional,mikusheva2010robust,tuvaandorj2024robust}. Instead, we adopt the potential outcomes framework and use randomization as the sole basis for inference.
%Compared with the typical setup of a constant effect linear instrumental variable model in the standard weak instrument literature, such as those in \cite{staigerstock97}, \cite{moreira2009tests} and \cite{mikusheva2010robust}, ours is different in that we under randomization model with (nonparametrically specified) potential outcomes. 
% and mention only some of the most relevant works here. 
%The optimality of the Anderson--Rubin statistic in the just-identified case is established in \cite{moreira2009tests}. 
We provide a uniform analysis of our proposed procedures in the spirit of \cite{mikusheva2010robust} and \cite{andrews2011generic}. 

Computational methods for the Anderson--Rubin statistic were considered in \cite{dufour2005projection}, \cite{mikusheva2006tests}, \cite{mikusheva2010robust}, \cite{kang2018inference}, and \cite{li2017general}.
These works use fixed (usually normality-based) critical values, in contrast to our randomization-based critical values.
%Such papers consider the cases with a fixed (usually normality-based) critical values, and our paper studies a nontrivial extension to the randomization inference setting (and possibly with other resampling methods).
%\textcolor{red}{I liked my old explanation.}

Computational methods for randomization inference for the LATE were studied in \cite{kang2018inference}, \cite{keele2017randomization}, and \cite{nolen2011randomization}. \cite{kang2018inference} considers confidence interval construction based on the normal approximation, \cite{keele2017randomization} considers the case of binary outcome variables, and \cite{nolen2011randomization} considers the problem of confidence-interval constructions under the constant-additive-effect assumption.   

\subsection{Organization}
Section \ref{section:basic_setup} outlines the basic setup. Section \ref{section:cs_without_ra} and Section \ref{section:cs_with_ra} provide results for the cases without and with covariate adjustment, respectively. Section \ref{section:implementation} provides efficient algorithms for computing the confidence sets.  
%provides results for algorithms that implement the confidence sets. Section \ref{section:sim_and_dataapp} contains simulation results and data applications. 
Proofs of the theorems in the main text are deferred to the appendices. 

\subsection{Acknowledgments}

P.\ L.\ was supported by NSF postdoctoral fellowship DMS-2202891.

\section{Basic Setup}\label{section:basic_setup}

\subsection{Notation and Identification Assumptions}
We consider an experiment with $n$ units, where each unit is randomized to treatment or control. Let $y_i(d,z)$ denote the potential outcome of unit $i$ under the potential compliance decision $d\in \{0,1\}$ and the treatment assignment $z\in \{0,1\}$. Let $d_i(z)$ denote the compliance decision of the $i$th unit when under treatment status $z$. We write $d_i(1) = 0$ if the $i$th unit is assigned to treatment but does not receive treatment; we write $d_i(1)=1$ if it does. The compliance decision $d_i(0)$ under assignment to control is defined analogously. 

As in \cite{LATE1994}, we assume outcomes are independent of treatment assignments after conditioning on compliance. In other words, potential outcomes only depend on realized treatment status. Let $[n]=\{1,\dots,n\}$.

\begin{assumption}[Exclusion Restriction]\label{a:exclusion}
For all  $i \in [n] $, $y_i(d,z)$ is independent of $z$.\footnote{That is, $y_i(d,1)=y_i(d,0)$ for all $i\in [n]$.} 
\end{assumption}
Given this assumption, we shall write $y_i(d)$ from now on to denote the potential outcome under treatment status $d$.

We further assume, as in \cite{LATE1994}, that $d_i(1) \ge d_i(0)$ (monotonicity), meaning that assignment to treatment does not cause any unit to decline treatment if that unit would have taken the treatment under assignment to control.

\begin{assumption}[Monotonicity]\label{a:monotonicity}
For all  $i \in [ n ] $, we have  $d_i(1) \ge d_i(0)$.
\end{assumption}

Let $\mathcal C$ denote the set of compilers, % such that $d_i(0) = 0$ and $d_i(1) = 1$,
\begin{equation}\label{def:compliers}
    \mathcal{C}=\{i\in[n]: d_i(0)=0, d_i(1)=1\}.
\end{equation}
Our inferential target is the LATE (or Complier Average Treatment Effect) in the experimental sample, defined by 
\begin{equation}
\cace = \frac{1}{|\mathcal C|} \sum_{i \in \mathcal C} \big( y_i(1)  - y_i(0) \big).    
\end{equation}

We shall assume throughout the paper that there is at least one complier ($|\mathcal{C}|\geq 1$) so that the quantity is well-defined. It is shown in \cite{LATE1994} that the LATE is identified under these two assumptions, and is estimable by the ratio of outcome differences between treated and control groups to the compliance rate.

\subsection{Experimental Design and Data Environment}
We assume that the experiments under consideration are completely randomized: among $n$ units, researchers assign $n_1$  units to treatment, chosen uniformly at random. This assumption is not crucial, and our results can be readily generalized to other experimental designs such as  clustered randomization and pair randomization. For each unit $i$, we let $Z_i$ be a random variable such that $Z_i=1$ if unit $i$ is assigned to treatment, and $Z_i=0$ otherwise. We note that $\P(Z_i =1 ) = n_1/n$. 

For each unit $i$, researchers observe compliance decisions and outcomes given by $D_i = d_i(Z_i)$ and  $Y_i = y_i(D_i)$, respectively.  
Each unit $i$ may additionally be associated with some covariate information (e.g., demographic information); if so, this information is collected in a column vector $x_i\in\mathbb{R}^k$. Throughout the paper we assume that $k$ does not change with the sample size $n$. The observed dataset hence consists of a set of outcome--compliance--assignment triples $\{(Y_i,D_i,Z_i)\}_{i=1}^n$, and possibly additional covariates $\{x_i\}_{i=1}^n$.

\subsection{Statistical Framework} 
We adopt the finite-population (design-based) framework, treating the potential outcomes and compliance decisions %$\{(y_i(1),y_i(0),d_i(1),d_i(0))\}_{i=1}^n$ 
%changing this because it's really ugly; reverse if you want
as fixed parameters \citep{imbens2015causal}. %The only source of randomness in our model is the random treatment assignments. 
The sole source of randomness in our model is the vector of random treatment assignments, and statistical uncertainties are evaluated exclusively with respect to this randomness.

\section{Confidence Sets Without Regression Adjustment}\label{section:cs_without_ra}

We begin by describing randomization-based confidence sets for the LATE that do not incorporate covariate information.  We consider a completely randomized experiment where $n_1$ out of $n$ subjects are assigned to treatment. We set $\pi=n_1/n$.
\subsection{Studentized Estimator}

For all $\beta \in \R$, we define the estimator
\be\label{introAR}
\hat\tau(\beta) = \frac{1}{n} \sum_{i=1}^n (Y_i - \beta D_i) ( Z_i - \pi ).
\ee

This statistic is a single-instrument %unsquared    
version of the standard Anderson–Rubin statistic (see, e.g., \cite{andrews2019weak}).\footnote{Our definition of the Anderson--Rubin statistic is a centered version of the statistic  defined in \cite{imbens2005robust}.}  
This statistic is studied in \cite{imbens2005robust} in the finite-population framework. They show that, under a constant-additive-effect assumption\footnote{By a constant-additive-effect assumption, we mean that there exists some $\beta$ such that $y_i(1)  - y_i(0) = \beta$ for all units. 
 
 }  and when $\beta$ is the true treatment effect, $Y_i - \beta D_i$ recovers the control outcomes $y_i(0)$ and is independent of the treatment assignments. This observation leads to the use of statistic (\ref{introAR}) to construct randomization-based, finite-sample exact confidence sets under the constant-additive-effect assumption. We extend these results by showing that, with proper studentization, this procedure has an  %good
asymptotically uniform guarantee even when the treatment effects are heterogeneous. 
 
We remark that a short calculation shows that
\be\label{eq:dim_numerator}
\hat\tau(\beta) = \pi ( 1- \pi)
\left(
\frac{1}{n_1} \sum_{i=1}^n Z_i\left(Y_i-\beta D_i\right)   - \frac{1}{n_0}\sum_{i=1}^n (1 - Z_i)\left(Y_i-\beta D_i\right) 
\right),
\ee
where $n_0 = n - n_1$, so $\hat\tau(\beta)$ is simply a rescaled difference in means statistic.  
We further define
\begin{align}
\hat v_{i}(1,\beta)&= Y_i-\beta D_i - \frac{1}{n_1}\sum_{i\in [n]}Z_i\big(Y_i-\beta D_i\big),\\
\hat v_{i}(0,\beta)&= Y_i-\beta D_i - \frac{1}{n_0}\sum_{i\in [n]}(1-Z_i)\big(Y_i-\beta D_i\big).
\end{align}
Using these notations, we define the variance estimator 
\be\label{eqn:AR_var}
\hat{\sigma}^2(\beta)= \frac{\pi^2(1-\pi)^2}{n_1^2}\sum_{i\in [n]} Z_i \hat v^2_{i}(1,\beta)  +      \frac{\pi^2(1-\pi)^2}{n_0^2}\sum_{i\in [n]} (1-Z_i)\hat v^2_{i}(0,\beta), 
\ee
and the studentized statistic
\be\label{eqn:AR_obs}
 \Delta(\beta) = \frac{ \hat \tau(\beta) }{\hat \sigma(\beta)}.
\ee
We will use the absolute value %of $ \Delta(\beta)$, 
$|\Delta(\beta)|$ as our test statistic and construct confidence sets by test inversion. We refer to the test statistic $|\Delta(\beta)|$ in (\ref{eqn:AR_obs}) as the \textit{AR statistic}. %throughout the paper.

\subsection{Critical Values and Confidence Sets} 
To choose critical values for $|\Delta(\beta)|$, we use its randomization distribution. Note that the assignment vector $Z= (Z_1, \dots, Z_n)$ is a random vector in $\{0,1\}^n$. Let $\bs Z^*$ be a random vector with the same distribution as $\bs Z$, and is independent from $\bs Z$. We use $\P^*$ to denote the probability measure of $Z^*$  that conditions on $\bs Z$. % and $\E^*$ and $\var ^*$ to denote expectation and variance with respect to $\P^*$.

We define $\hat \tau^* \colon \R \rightarrow \R$ by
\be
\tau^*(\beta) = \frac{1}{n} \sum_{i=1}^n (Y_i - \beta D_i) \big(  Z_i^* - \pi \big),
\ee
%where $Y_i$ and $D_i$ are held constant across different $Z_i^*$ draws. 
where we recall that by definition $Y_i$ and $D_i$ depend only on $Z_i$, and not $Z_i^*$.
We also define the variance estimator 
\begin{equation}\label{eqn:AR_var_rand}
    \begin{split}
  \hat \sigma^{*2}(\beta) =&  \frac{\pi^2(1-\pi)^2}{n_1^2}\sum_{i\in [n]}Z^*_i \left(Y_i - \beta D_i -  \frac{1}{n_1}\sum_{i\in [n]}Z^*_i (Y_i - \beta D_i)\right)^2\\
 &+ \frac{\pi^2(1-\pi)^2}{n_0^2}\sum_{i\in [n]}\left(1-Z^*_i\right) \left(Y_i - \beta D_i -  \frac{1}{n_0}\sum_{i\in [n]}\left(1-Z^*_i\right) (Y_i - \beta D_i) \right)^2,
\end{split}
\end{equation}
and the studentized statistic $\Delta^*(\beta)$ by
\be\label{eqn:AR_rand}
\Delta^*(\beta)=\frac{\widehat{\tau}^*(\beta)}{\widehat{\sigma}^*(\beta)}.
\ee

The distribution of $|\Delta^*(\beta)|$ under $\P^*$ is discrete and supported on $[0, \infty]$. For every $\alpha \in (0, 1)$, there exists a smallest (possibly infinite) value $\eta^*_{1-\alpha} (\beta, Z)$ such that 
\be\label{eqn:rand_quantile}
\P^*\left( \big|\Delta^*(\beta)\big| \le \eta^*_{1-\alpha} (\beta, Z) \right) \ge 1 - \alpha.
\ee
Our confidence set is defined as 
\be\label{eqn:inversion_dim}
 I_{1-\alpha}^*(Z) = \big\{ \beta \in \R :  |\Delta(\beta)| \le \eta^*_{1-\alpha} (\beta, Z)  \big\}.
\ee
\begin{remark}
  The Wald estimator, 
\begin{equation}
    \widehat{\beta}_{\textnormal{Wald}} = \frac{ \frac{1}{n_1}\sum_{i\in [n]}Z_iY_i-\frac{1}{n_1}\sum_{i\in [n]}\left(1-Z_i\right)Y_i }{ \frac{1}{n_1}\sum_{i\in [n]}Z_iD_i-\frac{1}{n_0}\sum_{i\in [n]}\left(1-Z_i\right)D_i },
\end{equation}
always lies within the set $I_{1-\alpha}^*(Z)$ (because $\Delta(  \widehat{\beta}_{\textnormal{Wald}})=0$), provided it is well defined. We also note that $I_{1-\alpha}^*(Z)$ may not necessarily form a single interval but is generally a finite union of intervals when it is nonempty. 
\end{remark}

\subsection{Theoretical Guarantees} As noted in \cite{imbens2005robust}, the confidence set (\ref{eqn:inversion_dim}) is a finite-sample valid $(1-\alpha)\times 100\%$ confidence set under a constant-additive-effect assumption.
In practice, assuming a constant effect may be restrictive. We show below that the confidence set $I^*_{1-\alpha}(Z)$ is valid asymptotically, even when the treatment effects are heterogeneous.
Following \cite{mikusheva2010robust} and \cite{andrews2011generic}, we provide an asymptotic uniform-coverage guarantee of the confidence set under the following assumptions.

Recall the definition of complier group $\mathcal{C}$ in (\ref{def:compliers}).

\begin{assumption}[Parameter Space]\label{d:theta}
Let $n\in \mathbb{N}$ and constants $A, \delta>0$ be given.
We denote
\begin{equation}
   \bar{\beta}=\frac{1}{|\mathcal{C}|}\sum_{i\in \mathcal{C}}\left(y_i(1)-y_i(0)\right),
\end{equation}
and set
\begin{equation}
    w_i(1)=y_i(d_i(1))-\bar{\beta} d_i(1), \qquad w_i(0)=y_i(d_i(0))-\bar{\beta} d_i(0).
\end{equation}
The following conditions on the potential outcomes and compliance decisions hold. 

%$,
%consider the set of potential outcomes and compliance statuses  %$\${\left(y_i(1),y_i(0),d_i(1),d_i(0)\right)\}_{i=1}^n%$. 
%Denote the LATE as
%\begin{equation}
 %   \bar{\beta}=\frac{1}{|\mathcal{C}|}\sum_{i\in %%\mathcal{C}}\left(y_i(1)-y_i(0)\right),
%\end{equation}
%and define
%\begin{equation}
%    w_i(1)=y_i(d_i(1))-\bar{\beta} d_i(1), \qquad w_i(0)=y_i(d_i(0))-\bar{\beta} d_i(0).
%\end{equation}

%The following conditions hold.

\begin{enumerate}[label=(\roman*)]
% \item $|\mathcal  C | \ge 1$. 
\item

We have
\be\label{ii}
{\sigma_{\bs w(1), \bs w(0), \mathcal C}} \geq - \delta\cdot {\sigma_{\bs w (1),\mathcal C} \cdot \sigma_{\bs w(0),\mathcal C}} ,
\ee
where $\sigma_{\bs w(1), \bs w(0),\mathcal{C}}$ is the complier-group correlation of adjusted potential outcomes, 
\begin{equation}
    \sigma_{w(1),w(0),\mathcal{C}}=\frac{1}{|\mathcal{C}|-1}\sum_{i\in \mathcal{C}}\left(\left(w_i(1)-\mu_{w(1),\mathcal{C}}\right)\left(w_i(0)-\mu_{w(0),\mathcal{C}}\right)\right),
\end{equation}
where  $\mu_{w(a),\mathcal{C}}=\frac{1}{|\mathcal{C}|}\sum_{i\in\mathcal{C}}w_i(a)$ for $a\in\{0,1\}$,  and $\sigma^2_{\bs w(1),\mathcal{C}}$ and $\sigma^2_{\bs w(0),\mathcal{C}}$ are the complier-group adjusted potential outcome variances,
\begin{equation}
  \sigma^2_{w(a),\mathcal{C}}=\frac{1}{|\mathcal{C}|-1}\sum_{i\in\mathcal{C}}\left(w_i(a)-\mu_{w(a),\mathcal{C}}\right)^2,
\end{equation}
for $a\in\{0,1\}$. If $|\mathcal{C}|=1$, all variances and covariances are defined to be zero.
\item For each $a \in \{ 0, 1 \}$, 
\be \label{fourthmoment}
\left(\frac{1}{n}\sum_{i=1}^n |w_i(a)- \mu_{\bs w(a)}|^{4}\right)^{1/4}\leq  A  \sigma_{\bs w(a)},
\ee
where $\mu_{w(a)}=\frac{1}{n}\sum_{i=1}^n w_i(a)$ and $\sigma_{w(a)}^2=\frac{1}{n-1}\sum_{i=1}^n\left(w_i(a)-\mu_{w(a)}\right)^2$.
\end{enumerate}
\end{assumption}
\begin{remark}
Condition \eqref{ii} rules out extreme correlations between the potential treatment and control outcomes of the compliers, a common assumption in the finite-population setting. 
Condition \eqref{fourthmoment} is a standard technical condition that arises when applying the central limit theorem for triangular arrays, and may be thought of as excluding heavy-tailed or sparse data. When the potential outcomes  obey the constant-effect assumption (i.e.,  $y_i(1)=y_i(0)+\beta$ for all unit $i$), this assumption boils down to an assumption on the control outcomes: 
\begin{equation}
    \left(\frac{1}{n}\sum_{i=1}^n |y_i(0)-\mu_{y(0)}|^4\right)^{\frac{1}{4}}\leq A\sigma_{y(0)}.
\end{equation} 
%In practice, researchers can plot ratios between fourth-moment and variance estimators for a set of plausible LATEs as a diagnostic check.
\end{remark}
\begin{assumption}[Completely Randomized Experiment]\label{a:cr}    Given $n\in \mathbb{N}$ and a constant $0<r\leq1/2$, we have $n_1\in [rn,(1-r)n]$.
\end{assumption}

We denote the parameter space of potential outcomes, compliance decisions, and experimental designs satisfying Assumptions 1 through 4 with $n$ units and constants $\delta$, $r$, and $A$ by 
\begin{equation}\label{eqn:parameter space}
 \Theta_n(\delta,r,A)=\Big\{\Big( 
 \big\{\big(y_i(1),y_i(0),d_i(1),d_i(0)\big) \big\}_{i=1}^n, \Pi_n 
 \Big)\Big\}  ,
\end{equation}
where $\{(y_i(1),y_i(0),d_i(1),d_i(0))\}_{i=1}^n$ satisfies Assumption \ref{d:theta} with constants $\delta$ and $A$, and
$\Pi_n$ is a distribution of random assignment variables $\{Z_i\}_{i=1}^n$ implementing complete randomization that satisfies Assumption \ref{a:cr} with constant $r$. 
Each combination of potential outcomes, compliance decisions, and an experimental design completely determines the %randomization 
distribution of our test statistic.
%Given Assumption $\ref{d:theta}$ and Assumption $\ref{a:cr}$, 
%The parameter space of potential outcomes, compliance decisions, and experimental designs satisfying Assumption~\ref{d:theta} and Assumption~\ref{a:cr}
%is fully characterized by the sample size $n$ and the constants $\delta$, $A$, and $r$. Each combination of a parameter (i.e. potential outcomes and compliance decisions) and an experimental design uniquely determines the randomization distribution of our statistics.
%We denote this space of parameters and experimental designs as
%\begin{equation}\label{eqn:parameter space}
 %\Theta_n(\delta,r,A)=\{\hspace{3pt} \left( \%{\left(y_i(1),y_i(0),d_i(1),d_i(0)\right)\}_{i=1}^n, \Pi_n \right) \hspace{3pt}\}   
%\end{equation}
%where $\{\left(y_i(1),y_i(0),d_i(1),d_i(0)\right)\}_{i=1}^n$ satisfies Assumption \ref{d:theta} with constants $\delta$ and $A$, and $\Pi_n$ is the distribution of random assignment variables $\{Z_i\}_{i=1}^n$ that satisfies Assumption \ref{a:cr} with constant $r$. 
We prove an asymptotic coverage guarantee uniformly over sequences in $\Theta_n(\delta,r,A)$.
%of such parameters.
\begin{theorem}\label{t:main}
Fix $\alpha, \delta \in (0,1)$, $r \in (0,1/2]$, and $A > 0$. 
Let $\Theta_n(\delta,r,A)$ be the parameter space of models that satisfy Assumptions \ref{a:exclusion}, \ref{a:monotonicity}, \ref{d:theta}, and \ref{a:cr} with these constants. %$\delta, r, A$. 
We have
\begin{align}
   \lim_{n\rightarrow \infty}   
   \inf_{\beta \in \R} 
   \inf_{ \{\theta \in \Theta_n(\delta,r,A)  : \cace=\beta\}} \P_{\theta}\big (\beta  \in I^*_{1-\alpha}(Z) \big) \geq 1-\alpha . 
\end{align}
\end{theorem}

\section{Confidence Sets With Regression Adjustment}\label{section:cs_with_ra}

\subsection{Notation and Assumptions}

We now show how to modify the construction in the previous section to incorporate pre-treatment covariates via regression adjustment, a common method for improving the efficiency of inferential procedures.
%. \textcolor{orange}{Regression adjustment is commonly used to improve the efficiency of testing procedure.} 
We retain the previous notations and additionally suppose that each unit $i$ comes equipped with a column vector $x_i \in \R^k$ of covariates for some $k \in \mathbb{N}$. We suppose that each $x_i$ is demeaned, meaning that 
\be\label{demean}
\sum_{i=1}^n x_i = 0_k,
\ee
where $0_k$ is the zero vector in $\mathbb R^k$. 
We make this assumption without loss of generality, since given an arbitrary set of covariate vectors, one can enforce \eqref{demean} by subtracting the mean of the vectors from each $x_i$. We denote the $s$th element of $x_i$ by $x_{si}$. %Treating the $x_i$ as column vectors, we collect them into a matrix $X \in \mathbb R^{k \times n}$. 
\begin{comment}Let $M^{(1)} \in \mathbb{R}^{n_1 \times k }$ be the random submatrix of $X'$ that keeps only rows corresponding to subjects $i$ such that $Z_i=1$, and define $M^{(0)} \in \mathbb{R}^{n_0 \times k }$ analogously. 

Given a $\beta \in \R$, and a treatment assignment vector $Z$, consider the the linear regression model
\be\label{olsme}
Y^{(1)} = D^{(1)} \beta  + 1_{n_1} \phi  + M^{(1)} \gamma + \upsilon,
\ee 
where $1_{n_1} \in \R^{n_1}$ is a vector of ones, the column vectors $Y^{(1)}, D^{(1)} \in \R^{n_1}$ are defined analogously to $M^{(1)}$ by keeping only the indices with $Z_i = 1$, and $\upsilon$ is a mean zero disturbance term. Let $\hat \phi(1,\beta) \in \R$ and $\hat \gamma(1, \beta) \in \R^{k}$ denote the parameters given by the ordinary least squares estimator applied to \eqref{olsme}. Define $\hat \phi(0,\beta)$ and $\hat \gamma(0,\beta)$ analogously. We define $\hat \phi^* (a, \beta)$ and $\hat \gamma^*(a,\beta)$ similarly by using the $Z_i^*$ instead of the $Z_i$ to decide which rows to keep in the vectors $M^{(1)}, Y^{(1)}, D^{(1)}$ in \eqref{olsme}.
\end{comment}

%\textcolor{orange}{
Given a $\beta \in \R$ and a treatment assignment vector $Z$, consider the linear regression %estimator
\be\label{olsme}
Y_i-\beta D_i \sim (1-Z_i)\phi_0 + Z_i\phi_1 +  (1-Z_i)x_i'\gamma_0 + Z_i x_i'\gamma_1.
\ee %}
This is an interacted linear regression that regresses $Y_i-\beta D_i$ on 
\begin{equation}
    \big((1-Z_i),Z_i,(1-Z_i)x_i,Z_ix_i\big)
\end{equation}
without an intercept. For each $\beta$, denote the ordinary least squares (OLS) estimators resulting from \eqref{olsme} by $\hat \phi(0,\beta)$, $\hat \phi(1,\beta)$, $\hat \gamma(0,\beta)$, and $\hat \gamma(1,\beta)$, corresponding to  the order the associated coefficients appear in \eqref{olsme}. We define the randomization statistics $\hat \phi^* (0, \beta)$,  $\hat \phi^* (1, \beta)$, $\hat \gamma^*(0,\beta)$, and $\hat \gamma^*(1,\beta)$ similarly by using $Z^*$ instead of  $Z$.\footnote{If the design matrix is not invertible, which will happen with probability approaching 0 under our assumptions,  %to appear later,  
one can use the least-square coefficients with the pseudo-inverse matrix instead. }

\subsection{Studentized Estimators} 

%Denote $Y_\mathcal{I}(\beta)=Y_i-\beta D_i$. 
We define the covariate-adjusted difference in means statistic as
\begin{align}\label{formula:adj_num}
\begin{split}
     \hat \tau_\reg(\beta) = & \frac{1}{n_1}\sum_{i\in [n]}Z_i\big(Y_i-\beta D_i  -  x_i'\hat \gamma(1, \beta)\big)  \\  & - \frac{1}{n_0}\sum_{i\in [n]}(1-Z_i)\big(Y_i-\beta D_i  - x_i'\hat\gamma(0, \beta)\big)
\end{split}
\end{align}
and we define the covariate-adjusted randomization estimator $   \hat \tau^*_\reg(\beta)$ analogously, again replacing $Z_i$ by $Z_i^*$,  and  $\hat{\gamma}(0,\beta)$ and  $\hat{\gamma}(1,\beta)$ by their randomization analogues. 
The addition of the regression terms aims to improve the precision of $\hat \tau_\reg(\beta)$ relative to $\hat \tau(\beta)$ and consequently decrease the size of the resulting confidence sets \citep{zhao2021covariate}.\footnote{
We note that the infeasible covariate-adjusted estimator, where $\hat\gamma(1,\beta)$ and $\hat \gamma(0,\beta)$ are replaced by their infeasible counterparts defined in (\ref{covreg}), has mean $0$ when $\beta$ is equal to the true LATE by a straightforward calculation.}

\begin{comment}
    We note that the infeasible covariate-adjusted estimator, where $\hat\gamma(1,\beta)$ and $\hat \gamma(0,\beta)$ are replaced by their infeasible counterparts defined in (\ref{covreg}):
\begin{equation}
    \hat \tau_{\reg}^{\textnormal{inf}}(\beta)=  \frac{1}{n_1}\sum_{i\in [n]}Z_i\big(Y_i-\beta D_i  -  x_i' \gamma(1, \beta)\big) - \frac{1}{n_0}\sum_{i\in [n]}(1-Z_i)\big(Y_i-\beta D_i  - x_i'\gamma(0, \beta)\big),
\end{equation}
has mean $0$ when $\beta$ is equal to the true LATE because
\begin{align}
    \E[ \hat \tau_{\reg}^{\textnormal{inf}}(\beta)]=&\E[ \frac{1}{n_1}\sum_{i\in [n]}Z_i\big(Y_i-\beta D_i\big)-\frac{1}{n_0}\sum_{i\in [n]}(1-Z_i)\big(Y_i-\beta D_i \big)]\\
    & -\left(\E[\frac{1}{n_1}\sum_{i\in [n]}Z_ix_i'\gamma(1, \beta)]-\E[\frac{1}{n_0}\sum_{i\in [n]}(1-Z_i)x_i'\gamma(0, \beta)]\right) =\frac{1}{\pi(1-\pi)}E[\widehat{\tau}(\beta)]+0,
\end{align}
where $\hat \tau(\beta)$ is defined in (\ref{eq:dim_numerator}).
\end{comment}

The residual estimators associated with \eqref{olsme} and their randomization-based counterparts are %defined as 
\begin{align*}
    \hat \epsilon_i (\beta) &=Y_i - \beta D_i - \hat \phi (Z_i, \beta)  - x_i'\hat {\gamma}(Z_i, \beta) ,\\ 
 \hat \epsilon^*_i ( \beta)& =Y_i - \beta D_i - \hat \phi^* (Z_i^*, \beta) - x_i'\hat {\gamma}^*(Z_i^*, \beta).
\end{align*}
We use these to define the variance estimators
\begin{align*}
\hat{\sigma}^2_\reg (\beta) &= \frac{1}{n_1^2} \sum_{i\in [n]}Z_i \hat{\epsilon}_i(\beta)^2 +  \frac{1}{n_0^2} \sum_{i\in [n]}\left(1-Z_i\right) \hat{\epsilon}_i(\beta)^2, \\
     \hat{\sigma}^2_{\reg,*}(\beta) &= \frac{1}{n_1^2} \sum_{i\in [n]}Z_i^* \hat{\epsilon}^*_i(\beta)^2 +  \frac{1}{n_0^2} \sum_{i\in [n]}\left(1-Z_i^*\right) \hat{\epsilon}^*_i(\beta)^2.
\end{align*}
The adjusted studentized difference-in-means estimators are 
\be\label{adjdelta}
\Delta_{\reg}(\beta) = \frac{  \hat \tau_\reg(\beta) }{\hat{\sigma}_{\reg}(\beta) }, \qquad  \Delta^*_\reg(\beta) = \frac{  \hat \tau^*_\reg(\beta) }{\hat{\sigma}_{\reg,*} (\beta)}.
\ee

\subsection{Critical Values} 
Given the definitions in \eqref{adjdelta}, our confidence set  construction is similar to the previous case. Given $\alpha \in (0, 1)$, there exists a smallest (possibly infinite) value $\eta^{\reg,*}_{1-\alpha} (\beta, Z)$ 
%as the smallest real number in $[0,\infty]$ 
such that 
\[
\P^*\left( \big|\Delta_\reg^*(\beta)\big| \le \eta^{\reg,*}_{1-\alpha} (\beta, Z) \right) \ge 1 - \alpha.
\]
Our proposed confidence set is
\be\label{eqn:inversion_adj}
 I^{\reg,*}_{1-\alpha}(Z) = \big\{ \beta \in \R :  |\Delta_{\reg}(\beta)| \le \eta^{\reg,*}_{1-\alpha} (\beta, Z)  \big\}.
\ee
The implementation is discussed in Section \ref{section:implementation}.

\begin{remark}\label{remark2}
We note that $I_{1-\alpha}^{\reg,*}(Z)$ may not necessarily form a single interval but is generally a finite union of intervals when it is nonempty. The regression-adjusted Wald estimator, defined as 
\begin{equation}
\widehat{\beta}^{\reg}_{\textnormal{Wald}}=\frac{ \frac{1}{n_1}\sum_{i\in [n]}Z_i\left(Y_i-x_i'\widehat{\gamma}_{y,1}\right) -\frac{1}{n_0}\sum_{i\in [n]}\left(1-Z_i\right)\left(Y_i-x_i'\widehat{\gamma}_{y,0}\right) }{ \frac{1}{n_1}\sum_{i\in [n]}Z_i\left(Y_i-x_i'\widehat{\gamma}_{d,1}\right) -\frac{1}{n_0}\sum_{i\in [n]}\left(1-Z_i\right)\left(Y_i-x_i'\widehat{\gamma}_{d,0}\right)},
\end{equation}
always lies within this set (because $\Delta_{\reg}(  \widehat{\beta}^{\reg}_{\textnormal{Wald}})=0$), provided it is well-defined (i.e., the denominator is not zero). The coefficients $\widehat{\gamma}_{y,1}$ and  $\widehat{\gamma}_{y,0}$ are the OLS estimators obtained from the interacted regression
\begin{equation}\label{beta_wald}
    Y_i \sim \alpha_{y,0} + \alpha_{y,0} Z_i + Z_ix_i'\gamma_{y,1} + \left(1-Z_i\right)x_i'\gamma_{y,0},
\end{equation}
while  $\widehat{\gamma}_{d,1}$ and $\widehat{\gamma}_{d,0}$ are the OLS estimators from the regression
\begin{equation}
    D_i \sim \alpha_{d,0} + \alpha_{d,1} Z_i + Z_ix_i'\gamma_{d,1} + \left(1-Z_i\right)x_i'\gamma_{d,0}.
\end{equation}
The estimator $\widehat{\beta}^{\reg}_{\textnormal{Wald}}$ is also algebraically equivalent to the coefficient for $D_i$ in the instrumental variable regression \begin{equation}
    Y_i\sim \alpha + \alpha_1 D_i + Z_ix_i'\gamma_{1}  + \left(1-Z_i\right)x_i'\gamma_{0},  
\end{equation}
where the instruments are $1, Z_i, Z_ix_i$, and $ \left(1-Z_i\right)x_i$, when both estimators are well-defined (i.e., the denominator is not zero and the design matrix is invertible).  
These algebraic relationships are documented in Appendix \ref{appendix:remark}. 
\end{remark}

\subsection{Theoretical Guarantees} 

A straightforward extension of the reasoning in  \cite{imbens2005robust} shows that the confidence sets $ I^{\reg,*}_{1-\alpha}(Z)$ are valid  if the constant additive effects condition holds, regardless of the value of $n$. We show below that the confidence sets $ I^{\reg,*}_{1-\alpha}(Z)$ are asymptotically valid with effect heterogeneity under mild regularity conditions.

\begin{assumption}\label{d:expwithcovariates}
Let $k,n\in \mathbb{N}$ and constants $\delta, \tilde{\epsilon},A,B>0$ be given. 
We denote
\begin{equation}
   \bar{\beta}=\frac{1}{|\mathcal{C}|}\sum_{i\in \mathcal{C}}\left(y_i(1)-y_i(0)\right).
\end{equation}
For each $a\in\{0,1\}$, we define
\begin{equation}
    \epsilon_{i}(a,\bar{\beta})=y_i(d_i(a))-\bar{\beta}d_i(a)- \phi(a,\bar{\beta})-x_i' \gamma(a,\bar{\beta}),
\end{equation}
where we recall $\phi(a,\bar{\beta})$ and $\gamma(a,\bar{\beta})$ were defined in \eqref{covreg}.
The following conditions on the potential outcomes, compliance decisions, and covariates hold. 
%Given $k,n>0$, and positive constants $\delta$, $\tilde{\epsilon}$, $A$ and $B$, consider the set of potential outcomes and compliance statuses  $\{\left(y_i(1),y_i(0),d_i(1),d_i(0)\right)\}_{i=1}^n$. Denote the LATE as
%\begin{equation}
%    \bar{\beta}=\frac{1}{|\mathcal{C}|}\sum_{i\in \mathcal{C}}\left(y_i(1)-y_i(0)\right),
%\end{equation}
%For each $a\in\{0,1\}$, define
%\begin{equation}
%    \epsilon_{i}(a,\bar{\beta})=y_i(d_i(a))-\bar{\beta}d_i(a)- \phi(a,\bar{\beta})-x_i' \gamma(a,\bar{\beta}),
%\end{equation}
%where $\phi(a,\bar{\beta})$ and $\gamma(a,\bar{\beta})$ are defined in (\ref{covreg}).
\begin{enumerate}[label=(\roman*)]
\item We have 
\be \label{ii-covariate}
\frac{\sigma_{\epsilon(1, \bar{\beta}) ,\epsilon(0, \bar{\beta})}}{\sigma_{\epsilon(1, \bar{\beta})} \sigma_{\epsilon(0, \bar{\beta})}} \geq -\delta.
\ee

\item
For each $a \in \{0,1\}$,
\be\label{iv-covariate}
\left(\frac{1}{n}\sum_{i\in[n]} \epsilon_i^{4}(a,\bar{\beta})\right)^{1/4}\leq  B\sigma_{\epsilon(a,\bar{\beta})}.
\ee
\begin{comment}
\item
For each $a \in \{0,1\}$,
\be\label{iv-covariate2}
\left(\frac{1}{n}\sum_{i\in[n]} v_i(a)^{4}\right)^{1/4}\leq  A \sigma_{v(a)}.
\ee
\end{comment}
\item% The dimension of the covariate  is $k$, so that $x_i\in\mathbb{R}^k$, and %we have
We have 
\be\label{vi-covariate}
\sum_{i=1}^n x_i = 0_k,
\ee
where $0_k$ denotes the zero vector in $\R^k$. 
\item For all $s \in [ k]$, 
\be\label{vii-covariate}
\frac{1}{n}\sum_{i=1}^n x_{si}^4 \leq A  . 
\ee
\item Let $\lambda_{\min}$ denote the smallest eigenvalue of $n^{-1} \sum_{i=1}^n x_ix_i'$.  Then 
\be\label{viii-covariate}
\lambda_{\min} \ge \tilde \eps. 
\ee 
\end{enumerate}
\end{assumption}
We denote the parameter space of potential outcomes, compliance decisions, covariates, and experimental designs satisfying Assumptions 1, 2, 3, and 5 with $n$ units, covariates of dimension $k$, and constants $\delta$, $\tilde{\epsilon}$,  $A$, $B$, and $r$ by
\begin{equation}\label{eqn:parameter space2}
 \Xi_n(k,\delta,\tilde{\epsilon},A,B,r)=
 \Big\{
 \Big(
 \big\{
\big(y_i(1),y_i(0),d_i(1),d_i(0),x_i\big)
\big\}_{i=1}^n,
\Pi_n
 \Big)
\Big\}   .
\end{equation}
Here $\{(y_i(1),y_i(0),d_i(1),d_i(0), x_i)\}_{i=1}^n$ satisfies Assumption \ref{d:expwithcovariates} with constants $\delta$, $\tilde{\epsilon}$, $A$, and $B$, and 
$\Pi_n$ is a distribution of random assignment variables $\{Z_i\}_{i=1}^n$ implementing complete randomization that satisfies Assumption \ref{a:cr} with constant $r$. 
Each combination of potential outcomes, compliance decisions, covariate vectors, and an experimental design completely determines the %randomization 
distribution of our test statistic. 
\begin{comment}
Given Assumption $\ref{d:expwithcovariates}$ and Assumption $\ref{a:cr}$, the parameter space and experimental designs are fully characterized by the sample size $n$, the covariate dimension $k$, and the constants $\tilde{\epsilon}$, $\delta$, $B$, $A$, and $r$. Each combination of a parameter and an experimental design uniquely determines the randomization distribution of our statistics.
We denote this space of parameters and experimental designs as
\begin{equation}\label{eqn:parameter space2}
 \Xi_n(k,\delta,\tilde{\epsilon},A,B,r)=
 \Big\{
 \Big(
 \big\{
\big(y_i(1),y_i(0),d_i(1),d_i(0),x_i\big)
\big\}_{i=1}^n,
\Pi_n
 \Big)
\Big\}   
\end{equation}
where $\{\left(y_i(1),y_i(0),d_i(1),d_i(0),x_i\right)\}_{i=1}^n$ satisfies Assumption \ref{d:expwithcovariates} with constants $\delta$, $A$, $\tilde{\epsilon}$, $B$, and $\Pi_n$ satisfies Assumption \ref{a:cr} with constant $r$. 
\end{comment}
We prove an asymptotic coverage guarantee uniformly over sequences in $\Xi_n(k,\delta,\tilde{\epsilon},A,B,r)$. 
\begin{theorem}\label{thm:coverage2}
Fix $k\in \mathbb{N}$, $\alpha \in (0,1)$, $r \in (0,1/2]$, and %positive constants 
$\delta,\tilde{\epsilon},A,B>0$.
Let $\Xi_n$ be the parameter space of models that satisfy Assumptions \ref{a:exclusion}, \ref{a:monotonicity}, \ref{a:cr}, and \ref{d:expwithcovariates} with these  constants. We have
\begin{align}
   \lim_{n\rightarrow \infty}   
   \inf_{\beta \in \R} 
   \inf_{ \{\theta \in \Xi_n : \cace=\beta\}} \P_{\theta}\big (\beta  \in I^{\reg,*}_{1-\alpha}(Z) \big) \geq 1-\alpha . 
\end{align}
\end{theorem}

\section{Implementation}\label{section:implementation}

To compute the confidence sets in (\ref{eqn:inversion_dim}) and (\ref{eqn:inversion_adj}) algorithmically, one %theoretically 
seemingly
needs to enumerate over the uncountable set of all possible LATEs (the real line). In practice, one may adopt a naive approach to approximate the confidence sets by selecting a sufficiently fine grid and applying the test inversion procedure to each point on the grid. This section presents new algorithms that improve upon this naive approach. These algorithms can recover (a Monte Carlo version of) the confidence sets \textit{exactly}.

%The new algorithms utilize 
Our algorithms use a characterization of the confidence sets stated below in \Cref{lemma:cs_as_intervals}. The lemma shows that each confidence set is a collection of intervals whose endpoints are the intersections of two polynomial functions. Each polynomial function is associated with a simulated assignment. The algorithm then solves $O(\nmc^2)$ subproblems, where $\nmc$ is the number of simulated assignments used in the randomization tests.\footnote{For AR statistics with fixed (normal) critical values, \cite{dufour2005projection}, \cite{mikusheva2010robust}, \cite{mikusheva2006tests}, and \cite{li2017general} describe all possible
geometric shapes of the inverted confidence set. In our setting, it is clear that the set is not empty when the Wald estimator (i.e., the LATE estimate) is well defined. Beyond that, to our knowledge, there is %little
no known simple geometric description of the confidence set constructed using randomization tests. Generically, the inverted set can consist of multiple intervals. 
}
\begin{comment}
However, it is possible to determine whether the inverted set is bounded or not: that is if "$-\infty$" or "$\infty$" belong to the set. Let $k=1,...,N_{perm}$ be the number of simulations. Denote the realized AR statistics in the $k$th randomization as $AR_k(\beta)\equiv\left|\frac{  \frac{1}{n}\sum_{i\in [n]} Y_i^{\beta}(Z_i^*-\pi)}{\sqrt{\hat{\sigma}^2_\beta(\mathcal{O}, \mathcal{Z}_k)}}\right|$.     
\end{comment}
For brevity, we focus our discussion here on the unadjusted AR statistic (\ref{eqn:AR_obs}). The same logic applies to the regression-adjusted AR statistic (\ref{adjdelta}).
Recall the definition of the AR statistic with randomization (\ref{eqn:AR_rand}) and the quantile of its randomization distribution (\ref{eqn:rand_quantile}). 

In practice, an exact calculation of $\eta^*_{1-\alpha}(\beta)$ for a fixed $\beta$ is infeasible even for a dataset of moderate size (for example, an experiment with $n=30$ and $n_1=15$ contains over $1.5\times 10^8$ possible assignments). The critical value $\eta^*_{1-\alpha}(\beta)$ is usually estimated by a Monte Carlo simulation with a suitably large number of independent random assignments to approximate the distribution of $\Delta^*(\beta)$. Let $\nmc$ be the number of simulated assignments, and $\Delta^*_{k}(\beta)$ the AR statistic associated with $k$th assignment. We define
\begin{equation}\label{eqn:rand_quantile_feasible}
\widehat{\eta}^*_{1-\alpha}(\beta,Z) = \min\left\{ \left|\Delta^*_k(\beta)\right|: \frac{1}{\nmc}\sum_{s=1}^{\nmc} \one\{ \left|\Delta^*_s(\beta)\right|\leq \left|\Delta^*_k(\beta)\right| \}\geq 1-\alpha\right\},
\end{equation}
and the confidence set of interest is defined as
\begin{equation}\label{eqn:perm_cs}
    \hat I^*_{1-\alpha}(Z) = \big\{ \beta \in \R :  |\Delta(\beta)| \le \hat{\eta}^*_{1-\alpha} (\beta, Z)  \big\}.
\end{equation}
\subsection{AR functions}
To numerically construct the set \eqref{eqn:perm_cs}, it is instructive to look at the statistics $\Delta^2(\beta)$. We note that each AR statistic should be considered as a function of $\beta$. Conceptually, we are working with random functions of $\beta$ defined over the real line. %, as a function of $\beta$. 
We therefore refer to the AR statistics as the AR functions below. For the observed AR function, we have
\be\label{eqn:AR_sq_ratio}
    \Delta^2(\beta) = \frac{a \beta^2 + b \beta +c }{d \beta^2 + e\beta + f},
\ee
where
\newpage 
\begin{align}
     a=& \left( \frac{1}{n}\sum_{i=1}^n D_i (Z_i-\pi )  \right)^2,\\
     b=& -2 \left(\frac{1}{n}\sum_{i=1}^n Y_i (Z_i-\pi )\right)\left( \frac{1}{n}\sum_{i=1}^n D_i (Z_i-\pi )  \right),\\
     c= &\left( \frac{1}{n}\sum_{i=1}^n Y_i (Z_i-\pi )  \right)^2,\\
     d= & \frac{\pi^2(1-\pi)^2}{n_1^2}\sum_{i\in [n]}Z_i\left(D_i -\frac{1}{n_1}\sum_{i\in[n]}Z_i D_i\right)^2 \\
     & +\frac{\pi^2(1-\pi)^2}{n_0^2}\sum_{i\in [n]}\left(1-Z_i\right)\left(D_i -\frac{1}{n_0}\sum_{i\in[n]}\left(1-Z_i\right)D_i\right)^2,\\
     e= &-2\frac{\pi^2(1-\pi)^2}{n_1^2}\sum_{i\in [n]}Z_i\left(Y_i -\frac{1}{n_1}\sum_{i\in[n]}Z_i Y_i\right)\left(D_i -\frac{1}{n_1}\sum_{i\in [n]}Z_iD_i\right) \\
     &-2\frac{\pi^2(1-\pi)^2}{n_0^2}\sum_{i\in [n]}(1-Z_i)\left(Y_i -\frac{1}{n_0}\sum_{i\in[n]}(1-Z_i) Y_i \right)\left(D_i -\frac{1}{n_0}\sum_{i\in [n]}(1-Z_i)D_i\right),  \\
     f= & \frac{\pi^2(1-\pi)^2}{n_1^2}\sum_{i\in [n]}Z_i\left(Y_i -\frac{1}{n_1}\sum_{i\in[n]}Z_i Y_i\right)^2 \\
      &+\frac{\pi^2(1-\pi)^2}{n_0^2}\sum_{i\in [n]}\left(1-Z_i\right)\left(Y_i -\frac{1}{n_0}\sum_{i\in[n]}\left(1-Z_i\right)Y_i\right)^2.
\end{align}

Note that the assignment vector $Z= (Z_1, \dots, Z_n)$ is a random vector in $\{0,1\}^n$. Let $\bs Z^*$ be a random vector with the same distribution as $\bs Z$, and independent from $\bs Z$. We denote the $k$th simulated assignment by $Z^{*}_k$, and its $i$th entry by $Z^{*}_{ik}$.

For the AR function calculated from the $k$th simulated assignment, we have %the similar expression:
\begin{equation}\label{eqn:AR_sq_ratio_rand}
   \Delta^{*2}_k(\beta) = \frac{a_k \beta^2 + b_k \beta +c_k }{d_k \beta^2 + e_k\beta + f_k},
\end{equation}
where the coefficients are calculated using the same formulas as above, but with the $Z_i$'s replaced by the simulated assignments $Z_{ik}^{*}$'s. 
For regression-adjusted AR statistics, similar expressions can be defined. We include them in Appendix \ref{app:reg_AR_expression}. 

We make two comments.
\begin{remark}
We note that the denominators $d\beta^2+e\beta+f$ and $d_k\beta^2+e_k\beta+f_k$ are non-negative variance estimators and will be positive over the entire real line unless the $Y_i-\beta D_i$ are perfectly correlated. For example, by \eqref{eqn:AR_var_rand}, we have that $d_k\beta^2+e_k\beta +f_k$ is equal to 
\begin{equation}
         \begin{split}
  \hat \sigma^{*2}_k(\beta) =&  \frac{\pi^2(1-\pi)^2}{n_1^2}\sum_{i\in [n]}Z^{*}_{ik} \left(Y_i - \beta D_i -  \frac{1}{n_1}\sum_{i\in [n]}Z^{*}_{ik}(Y_i - \beta D_i)\right)^2\\
 &+ \frac{\pi^2(1-\pi)^2}{n_0^2}\sum_{i\in [n]}\left(1-Z^{*}_{ik}\right) \left(Y_i - \beta D_i -  \frac{1}{n_0}\sum_{i\in [n]}\left(1-Z^{*}_{ik}\right) (Y_i - \beta D_i) \right)^2.  
    \end{split}
\end{equation}
This quantity will be zero only if the $Y_i-\beta D_i$ are perfectly correlated within both the treatment and control groups. That is,
\be
Y_i - Y_j = \beta (D_i-D_j)
\ee for some $\beta$ and for any pair $i$ and $j$ in the treatment or control group. This implies that the $Y_i$'s are completely parallel to $D_i$'s. 
Such datasets are relatively rare, especially for those with continuous outcomes. We shall assume from now on that $d\beta^2+e\beta+f$ and $d_k\beta^2+e_k\beta+f_k$ are positive on $\mathbb{R}$. This assumption implies that  $\Delta^2(\beta)$ and each $\Delta^{*2}_k(\beta)$ are continuous functions on $\mathbb{R}$.
%We note that perfect correlations can be efficiently detected in practice by checking the correlation coefficient.
\end{remark}
\begin{assumption}\label{assn:positive_variance}
   Let $d\beta^2+e\beta+f$ and $d_k\beta^2+e_k\beta+f_k$ be defined as in (\ref{eqn:AR_sq_ratio}) and (\ref{eqn:AR_sq_ratio_rand}), respectively. As functions of $\beta$, they are positive everywhere on $\mathbb{R}$.
\end{assumption}
%We note that this assumption implies that  $\Delta(\beta)^2$ and $\Delta^{*2}_k(\beta)$ are continuous functions on $\mathbb{R}$.
\begin{remark}
Our algorithm makes use of the intersection points of two AR functions. The intersections of two AR functions are the solutions to the equality $\Delta^{2*}_k(\beta)=\Delta^{2*}_s(\beta)$, where $k$ and $s$ are two indices for simulated assignments. The solutions solve the equation
\begin{align}
        \frac{a_k \beta^2 + b_k \beta +c_k}{d_k\beta^2+e_k\beta+f_k}= 
        \frac{a_s\beta^2+b_s\beta+c_s}{d_s\beta^2+e_s\beta+f_s},
\end{align}
which, by Assumption \ref{assn:positive_variance}, is equivalent to
\begin{align}
    \left(a_k \beta^2 + b_k \beta +c_k\right)\left( d_s\beta^2+e_s\beta+f_s\right) = \left(a_s\beta^2+b_s\beta+c_s \right)\left(d_k\beta^2+e_k\beta+f_k\right).
\end{align}
Hence the solution set is the solution set of a polynomial function in $\beta$ with the degree at most 4. The solution set may be empty, a discrete set of size at most 4, or the real line. When the solution set is the real line, the two functions agree over the real line and they are the same function. From now on, when we discuss the intersections of two AR functions, we mean the intersections of two \textit{distinct} AR functions.
\end{remark}

\subsection{Algorithm Details and Guarantees} For every $\beta\in\mathbb{R}$, let $\mathcal{I}(\beta)$ be the set of indices of randomizations that attain equality in the definition of $\widehat{\eta}^*_{1-\alpha}(\beta,Z)$. That is, 
\begin{equation}\label{eqn:quantile_set}
    \mathcal{I}(\beta) = \bigr\{k: \big|\Delta^*_k(\beta)\big|=\widehat{\eta}^*_{1-\alpha}(\beta,Z) \bigl\}. 
\end{equation}
The set is always nonempty, by %the definition of $\widehat{\eta}^*_{1-\alpha}(\beta,Z)$ in 
\eqref{eqn:rand_quantile_feasible}. The following lemma simplifies the test inversion task by partitioning the uncountably infinite parameter space into a finite number of intervals. Points in each interval share a common element in their index set $\mathcal{I}(\beta)$.
The lemma states that the index set $\mathcal{I}(\beta)$ can only change when two distinct AR functions intersect. Conversely, if there is no intersection on an open interval, then the index set remains the same. Moreover, the index set on this open set is a subset of the index set at its boundary points.

%The lemma states that if the index sets at two points are different, then there must be two AR-functions intersect with each other.
%\begin{lemma}\label{lemma:intersection}
%For any $\beta_1,\beta_2\in\mathbb{R}$ , if $ \mathcal{I}(\beta_1)\not = \mathcal{I}(\beta_2)$, then there must exist one $\beta\in [\beta_1,\beta_2]$ and a pair $s$ and $t$ such that, $\Delta^*_{s}(\beta)=\Delta^*_{t}(\beta)$.
%\end{lemma}
%\begin{proof}
%If $\mathcal{I}(\beta_1 )$ or $\mathcal{I}(\beta_2)$ has strictly more than 1 element, the statement is clearly true. 

%The remaining case is where both $\mathcal{I}(\beta_1)$ and $\mathcal{I}(\beta_2)$ have only one element (the sets cannot be empty by definition) and $\mathcal{I}(\beta_1)\not =\mathcal{I}(\beta_2)$. The conclusion then follows by a contradiction: if there is no intersection on the interval $[\beta_1,\beta_2]$, the ranking of the curves should remain invariant, which implies fact $\mathcal{I}(\beta_1)=\mathcal{I}(\beta_2)$.
%\end{proof}
%The lemma is useful due to its converse.
\begin{lemma}\label{lemma:cs_as_intervals}
Consider an open interval $(\beta_1,\beta_2)$, where it is possible that $\beta_1=-\infty$ or $\beta_2=\infty$. Under Assumption \ref{assn:positive_variance}, and supposing there is no intersection among the AR functions $\{\Delta^{*2}_{k}(\beta)\}_{k=1}^{\nmc}$ in $(\beta_1,\beta_2)$, then
\begin{enumerate}[label=(\roman*)]
\item the index set $\mathcal{I}(\beta)$ remains the same on $(\beta_1,\beta_2)$,
\item $\mathcal{I}(\beta)\subset \mathcal{I}(\beta_1) \cap \mathcal{I}(\beta_2)$ for all $\beta\in (\beta_1,\beta_2)$, if both $\beta_1\in\mathbb{R}$ and $\beta_2\in\mathbb{R}$,
\item $\mathcal{I}(\beta)\subset \mathcal{I}(\beta_2)$ for all $\beta\in (\beta_1,\beta_2)$, if $\beta_1=-\infty$ and $\beta_2\in\mathbb{R}$,
\item $\mathcal{I}(\beta)\subset \mathcal{I}(\beta_1)$ for all $\beta\in (\beta_1,\beta_2)$, if $\beta_2=\infty$ and $\beta_1\in\mathbb{R}$.
\end{enumerate}
\end{lemma}
\begin{proof}
%The first statement is clear by Lemma (\ref{lemma:intersection}). 
If there is no intersection among the distinct AR functions on the interval  $(\beta_1,\beta_2)$, the ranking of the values of the curves at each $\beta$ remains the same, and hence the index set $\mathcal{I}(\beta)$ remains the same. This shows the first claim. 
The second statement follows by a contradiction: for $\beta_1\in\mathbb{R}$ and $\beta_2\in\mathbb{R}$, suppose there is a $t\in\mathcal{I}(\beta)$ and $t\not\in \mathcal{I}(\beta_1) \cap \mathcal{I}(\beta_2)$. The collection of squared AR statistics that are less than or equal to $\Delta^{*2}_t(\beta)$ over $(\beta_1,\beta_2)$ will also be weakly dominated on the boundary by continuity. Thus, in order for $t\not\in \mathcal{I}(\beta_1) \cap \mathcal{I}(\beta_2)$, there must exist an index $s$ such that $\Delta^{*2}_{s}(\beta)\geq \Delta^{*2}_{t}(\beta)$ on $(\beta_1,\beta_2)$ but $\Delta^{*2}_{s}(\beta)<\Delta^{*2}_{t}(\beta)$ at points $\beta_1$ or $\beta_2$. This is impossible by continuity. The arguments are similar if $\beta_1=-\infty$ or $\beta_2=\infty$.
\end{proof}

The lemma implies that we can partition the real line into multiple intervals according to the intersections of distinct AR functions. Note that the intersections are the solutions of polynomials with degree at most 4. For such polynomials, closed-form solutions are available and can be computed in constant time. 

There are at most $\nmc$ draws of assignments, and hence we have at most $4\nmc^2$ possible intersections. We can sort the intersections on the real line and divide the real line into $K$ intervals $\{[\beta_i, \beta_{i+1}]\}_{i=0}^{K}$, where $\beta_0\equiv -\infty$ and $\beta_{K+1}\equiv \infty$.\footnote{With an abuse of notation, we mean $(-\infty,\beta_1]$ when we write $[\beta_0,\beta_1]$ with $\beta_0=-\infty$, and similarly $[\beta_K,\infty)$ when we write $[\beta_K,\beta_{K+1}]$ with $\beta_{K+1}=\infty$. } On any such interval $[\beta_i,\beta_{i+1}]$, by Lemma \ref{lemma:cs_as_intervals}, there exists an element $t_i$ such that $t_i\in \mathcal{I}(\beta)$ for any $\beta\in [\beta_i,\beta_{i+1}]$ and we have 
\begin{equation}
    \bigl\{\beta: \Delta^2(\beta) \leq \widehat{\eta}^*_{1-\alpha}(\beta,Z),\beta\in [\beta_i,\beta_{i+1}] \bigr\}=       \bigl\{\beta: \Delta^2(\beta)\leq \Delta_{t_i}^{*2}(\beta)\bigr\}\cap [\beta_i,\beta_{i+1}].
\end{equation}
Note that to find $\bigr\{\beta:\Delta^2(\beta)\leq \Delta_{t_i}^{*2}(\beta)\bigl\}$, one only needs to find the zeros of a polynomial function with the degree at most 4. The zeros of the polynomials, and hence the region where the polynomial is nonpositive, can be computed in constant time. 
Thus the confidence set \eqref{eqn:perm_cs} can be constructed efficiently by taking the intersection of the interval $[\beta_i,\beta_{i+1}]$ with the set $\bigr\{\beta:\Delta^2(\beta)\leq \Delta_{t_i}^{*2}(\beta)\bigl\}$.
%positive set (or negative set, depending on the choice of the sign) of the corresponding polynomials.
For regression-adjusted AR statistics, the procedure is the same, but with different AR functions $\Delta^2_{\reg}(\beta)$ and $\Delta^{*2}_{\reg,s}(\beta)$. We document them in Appendix \ref{app:reg_AR_expression}.     
We summarize the algorithm for confidence sets without regression adjustment in Algorithm \ref{alg:1.1} below.
\begin{algo}\label{alg:1.1}
Constructing Confidence Set (\ref{eqn:perm_cs})
%\vspace*{-12pt}
\begin{tabbing}
   \qquad \enspace \textbf{Require:} Observed data $\{Y_i,D_i,Z_i\}_{i\in [n]}$, simulated assignments $\{Z^*_{ik}\}_{i\in[n],k\in [\nmc]}$ \\
   \qquad \enspace \textbf{Step 1:} For every $s,t\in [\nmc]$ that define distinct AR functions $\Delta^{*2}_s(\beta)$ and $\Delta^{*2}_t(\beta)$,\\
   \qquad \qquad calculate their intersections\\
   \qquad \enspace \textbf{Step 2:}   Sort the distinct intersections on the real line and denote the intervals that result by $\bigl\{[\beta_i, \beta_{i+1}]\bigr\}_{i=0}^{K}$\\
   \qquad \enspace \textbf{Step 3:} For each interval $[\beta_i, \beta_{i+1}]$: \\
   \qquad \qquad If $\beta_i\neq -\infty$ and $\beta_{i+1}\neq +\infty$:\\
   \qquad \qquad \qquad Calculate $\mathcal{I}( \left(\beta_i+\beta_{i+1}\right)/2)$, defined  in (\ref{eqn:quantile_set})\\
   \qquad \qquad \qquad Select any $t_i$ such that $t_i\in\mathcal{I}( \left(\beta_i+\beta_{i+1}\right)/2)$\\
   \qquad \qquad If $\beta_i=\beta_0=-\infty$:\\
   \qquad \qquad \qquad Select any $t_i\in \mathcal{I}(\beta_1-1)$\\
   \qquad \qquad If $\beta_{i+1}=\beta_{K+1}=+\infty$:\\
   \qquad \qquad \qquad Select any $t_i\in\mathcal{I}(\beta_{K}+1)$ \\
  \qquad \enspace \textbf{Step 4:} For each interval $\bigl\{[\beta_i, \beta_{i+1}]\bigr\}_{i=0}^{K}$:\\
%  \qquad \qquad Find $\mathrm{CS}_i=\bigl\{\beta: \big|\Delta^2(\beta)\big| \leq \big|\Delta^{*2}_{t_i}(\beta)\big|,\beta\in [\beta_i,\beta_{i+1}] \bigr\}$\\
 \qquad \qquad Find $\mathrm{CS}_i=\bigl\{\beta: \big|\Delta^2(\beta)\big| \leq \big|\Delta^{*2}_{t_i}(\beta)\big|,\beta\in [\beta_i,\beta_{i+1}] \bigr\}$\\
  \qquad \enspace \textbf{Step 5:} Return $\mathrm{CS}=\bigcup_{i=1}^K \mathrm{CS}_i$
\end{tabbing}
\end{algo}

% \renewcommand{\thealgorithm}{1.1} 
% \begin{algorithm}
% \caption{Constructing Confidence Set (\ref{eqn:perm_cs})}
% \label{alg:1.1}
% \begin{algorithmic}
% \Require Observed data $\{Y_i,D_i,Z_i\}_{i\in [n]}$, simulated assignments $\{Z_{ik}\}_{i\in[n],k\in [\nmc]}$. 
% \State \textbf{Step 1:} For every $s,t\in [\nmc]$ that define distinct AR functions $\Delta^{*2}_s(\beta)$ and $\Delta^{*2}_t(\beta)$, calculate their intersections.
% \State \textbf{Step 2:} Sort the intersections on the real line and denote the intervals that result by $\bigl\{[\beta_i, \beta_{i+1}]\bigr\}_{i=0}^{K}$.
% \State \textbf{Step 3:} For each interval $[\beta_i, \beta_{i+1}]$, calculate $\mathcal{I}( \left(\beta_i+\beta_{i+1}\right)/2)$, defined  in (\ref{eqn:quantile_set}), and select one $t_i$ such that $t_i\in\mathcal{I}( \left(\beta_i+\beta_{i+1}\right)/2)$. For $\beta_0=-\infty$, we take any $t_i\in \mathcal{I}(\beta_1-0.1)$. For $\beta_{K+1}=\infty$, we take any $t_i\in\mathcal{I}(\beta_{K}+0.1)$.
% \State \textbf{Step 4:} For each interval $\bigl\{[\beta_i, \beta_{i+1}]\bigr\}_{i=0}^{K}$, find
% \begin{equation}
%     \mathrm{CS}_i=\bigl\{\beta: \left|\Delta^2(\beta)\right| \leq \left|\Delta^{*2}_{t_i}(\beta)\right|,\beta\in [\beta_i,\beta_{i+1}] \bigr\}.
% \end{equation}
% \State \textbf{Step 5:} Return $\mathrm{CS}=\bigcup_{i=1}^K \mathrm{CS}_i$.  
% \end{algorithmic}

% \end{algorithm}

The construction of confidence set with covariates are documented in Algorithm \ref{alg:1.2} in  Appendix \ref{app:reg_AR_expression}. In Appendix \ref{appendix:algo}, we outline alternative algorithms (Algorithm \ref{alg:2.1} and Algorithm \ref{alg:2.2}). They are variants of Algorithm \ref{alg:1.1} and \ref{alg:1.2} and can also construct the confidence set exactly. They are in practice much faster than Algorithm \ref{alg:1.1} and \ref{alg:1.2}.

\begin{theorem}\label{thm:calculation}
  Without covariates and under Assumption \ref{assn:positive_variance}, the confidence set $\mathrm{CS}$ returned by Algorithm \ref{alg:1.1} or Algorithm \ref{alg:2.1} is equal to the one in (\ref{eqn:perm_cs}).
    With covariates and under Assumption \ref{assn:positive_variance_adj}, the confidence set $\mathrm{CS}$ returned by Algorithm \ref{alg:1.2} or Algorithm \ref{alg:2.2} is equal to the one in (\ref{eqn:perm_cs_adj}).\footnote{Assumption \ref{assn:positive_variance_adj} is stated in Appendix \ref{app:reg_AR_expression}. }
\end{theorem}

The proof is included in Appendix \ref{section:alg_proof}.

\section{Simulations and Data Applications}\label{section:sim_and_dataapp}
\subsection{Simulations}
We conduct a simulation study to compare the properties of our method with the standard two-stage least squares (2SLS) method. The simulation is a completely randomized experiment with $n=100$ units, with half assigned to treatment ($n_1=50$).

For each unit, the potential outcomes under assignment to control are imputed as
\begin{equation}\label{sim:y0}
   y_i(0) = \tau_g + 0.1\mathcal{Z}_{1i}, \hspace{5pt} g\in \{\mathcal{A},\mathcal{N},\mathcal{C}\},
\end{equation}
where the $\mathcal{Z}_{1i}$ are independent standard normal variables, $g$ denotes the compliance status (principal stratum) of unit $i$, and $\mathcal{A},\mathcal{N}$, and $\mathcal{C}$ stand for always-takers, never-takers, and compliers, respectively. We choose $\tau_{\mathcal{A}}=-2$, $\tau_{\mathcal{C}}=-1$, and $\tau_{\mathcal{N}}=0$.  
The potential outcomes under treatment are imputed as
\begin{equation}\label{sim:y1}
    y_i(1) = y_i(0) + 0.1\mathcal{Z}_{2i},
\end{equation}
where the $\mathcal{Z}_{2i}$ are standard normal variables. For each group, the treatment effect is recentered to ensure that it is exactly zero. For each unit, we also generate three covariates, each independently following a standard normal distribution.

Thus, our simulation design features heterogeneous treatment effects and a zero local average treatment effect. Additionally, it introduces a strong negative correlation between potential outcomes and compliance status, which may arise in practice and exacerbates the weak IV problem \citep{nelson1988distribution, nelson1990, maddala1992exact, angrist2024one}.

In the simulations, we vary the size of the complier group, considering sizes of 5, 10, 20, 50, and 90, corresponding to compliance rates of 5\%, 10\%, 20\%, 50\%, and 90\%. The remaining population is equally divided between always-takers and never-takers.\footnote{When the remaining population is odd, the extra unit is assigned as an always-taker.} For each simulation scenario, we fix the potential outcomes and run 2,000 simulations.\footnote{For randomization-based confidence sets, we use $10^4$ randomly generated assignments and Algorithm \ref{alg:2.2} for faster computation.} One would expect the 2SLS confidence intervals to exhibit severe undercoverage when the compliance rate is low.

Table \ref{table1} presents the coverage rates of the 2SLS confidence intervals and randomization-based confidence intervals, along with the time required to construct the latter. The first row reports the coverage rates of the nominal 95\% 2SLS intervals. As the compliance rate decreases, the 2SLS intervals exhibit significant undercoverage. In contrast, the randomization-based confidence sets maintain nominal coverage rates even under low compliance rates.

For computational times, with $10^4$ random assignments, Algorithm \ref{alg:2.2} completes in under an hour, while Algorithm \ref{alg:1.2} takes significantly longer. With $10^3$ random assignments, both algorithms finish in under a minute. For practical use, we recommend using $10^3$ random assignments for exploratory analysis and $10^4$ for formal reporting.
These results demonstrate the practicality of our method. 

\begin{landscape}

\begin{table}[ht]
    \centering
\begin{tabular}{|r|c|c|c|c|c|}
  \hline
Compliance Rate & 5\% & 10\% & 20\% & 50\% & 90\% \\ 
  \hline
Coverage Rate (2SLS) &  0.73 & 0.84 & 0.90 & 0.96 & 0.96 \\ 
Coverage Rate (randomization) & 0.95 & 0.95 & 0.95 & 0.96 & 0.95 \\ 
Avg. Time (seconds, Algo \ref{alg:2.2}, $\nmc=10^4$) & 3298.71 & 3445.44 & 3254.20 & 3575.64 & 3422.52 \\  
Avg. Time (seconds, Algo \ref{alg:1.2}, $\nmc=10^4$) &  32764.91 & 34604.22 & 35347.09 & 36065.46 & 31410.07 \\ 
Avg. Time (seconds, Algo \ref{alg:2.2}, $\nmc=10^3$) & 8.86 & 8.58 & 8.39 & 8.22 & 7.84 \\ 
Avg. Time (seconds, Algo \ref{alg:1.2}, $\nmc=10^3$) & 57.81 & 55.28 & 51.99 & 49.73 & 48.16 \\ 
   \hline
\end{tabular}
    \caption{This table reports simulation results where potential outcomes are generated according to (\ref{sim:y0}) and (\ref{sim:y1}), with the LATE set to 0 in all cases. The number of simulations is 2000 for all cases. The first row shows the coverage rates of the nominal 95\% 2SLS intervals from the IV regression $Y\sim~1+D+X \vert 1,Z,X$. The second row shows the coverage rates of the nominal 95\% randomization-based confidence sets from (\ref{eqn:perm_cs_adj}). The third and fifth rows display the average computation time (in seconds) for constructing randomization-based confidence sets using Algorithm \ref{alg:2.2} with $10^4$ and $10^3$ simulated assignments, respectively. The fourth and sixth rows report the average computation time for Algorithm \ref{alg:1.2} with $10^4$ and $10^3$ simulated assignments, respectively. The values in the fourth, fifth, and sixth rows are calculated based on 50 separate simulations.  }
    \label{table1}
\end{table}
\end{landscape}
\subsection{Data Applications}
For an application, we apply our method to the six GOTV (Getting Out the Vote) experiments in \cite{green2003getting}.\footnote{Some experiments in the paper were not completely randomized, but we analyze the data as if they were.} This paper examines the effectiveness of door-to-door canvassing in increasing voter turnout in local elections in 2001. The authors conducted six field experiments to assess the impact of personal contact on voter participation. Their findings indicate that door-to-door canvassing can significantly boost voter turnout.

One complication in the experiments is one-sided noncompliance: some voters in the treatment group cannot be reached for personal contact. Consequently, the authors reported the local average treatment effect to assess the actual impact of personal contact. In the application below, we conduct a similar analysis. We report both the randomization confidence set without covariate adjustment and the randomization confidence set adjusted by the variable indicating whether the voter voted in the 2000 presidential election. We also present the results from the 2SLS regressions, one without adjustment and one adjusted using the same (interacted) pretreatment variable.

\begin{table}[ht]
    \centering
\begin{tabular}{|c|c|c|c|}
  \hline
City & Bridgeport & Raleigh & Minneapolis \\ 
\hline 
  Sample Size & 1307 & 4660 & 2827  \\ 
  \hline
  \multicolumn{4}{|c|}{Without Covariate}\\
\hline 
  Wald Est. & 0.163 & $-0.020$ & 0.102  \\ 
  TSLS CI & [0.035,0.292] & $[-0.080,0.039]$ & $[-0.069,0.273]$  \\ 
  Perm. CI & [0.033,0.294] & $[-0.081,0.040]$ & $[-0.068,0.275]$  \\ 
\hline 
  \multicolumn{4}{|c|}{With Covariate}\\
  \hline 
  Wald Est. & 0.173 & $-0.016$ & 0.108 \\ 
  TSLS CI & [0.054,0.293] & $[-0.072,0.040]$ & $[-0.046,0.263]$ \\ 
  Perm. CI & [0.053,0.296] & $[-0.073,0.039]$ & $[-0.046,0.265]$ \\ 
   \hline
\end{tabular}
    \caption{Applications to the \cite{green2003getting} dataset for cities Bridgeport, Raleigh, and Minneapolis. Row 1 reports the sample size for each city. Row 2 reports the Wald estimate without covariate, Row 3 reports the nominal 95\% confidence interval resulted from the instrumental variable regression $Y\sim 1+ D|1,Z$, and Row 4 reports the nomial 95\% randomization-based confidence interval. Rows 4–6 present similar results with covariate adjustments. Row 5 reports the nominal 95\% confidence interval resulted from the instrumental variable regression $Y\sim 1+ D+Z\cdot X+(1-Z)\cdot X|1,Z,Z\cdot X,(1-Z)\cdot X$.    }
    \label{tab2}
\end{table}

\begin{table}[ht]
    \centering
\begin{tabular}{|c|c|c|c|}
  \hline
City &  Detroit & Columbus & St.\ Paul \\ 
\hline 
  Sample Size &  4954 & 2424 & 2097 \\ 
  \hline
  \multicolumn{4}{|c|}{Without Covariate}\\
\hline 
  Wald Est. &  0.083 & 0.104 & 0.150 \\ 
  TSLS CI & $[-0.007,0.173]$ & $[-0.059,0.267]$ & [0.018,0.282] \\ 
  Perm. CI &  $[-0.009,0.176]$ & $[-0.057,0.266]$ & [0.016,0.284] \\ 
\hline 
  \multicolumn{4}{|c|}{With Covariate}\\
  \hline 
  Wald Est. & 0.078 & 0.111 & 0.103 \\ 
  TSLS CI &  [0.004,0.152] & $[-0.043,0.265]$ & $[-0.024,0.229]$ \\ 
  Perm. CI &  [0.003,0.153] & $[-0.042,0.267]$ & $[-0.025,0.231]$ \\ 
   \hline
\end{tabular}
    \caption{Applications to the \cite{green2003getting} dataset for cities Detroit, Columbus, and St.\ Paul. Row 1 reports the sample size for each city. Row 2 reports the Wald estimate without covariate, Row 3 reports the nominal 95\% confidence set resulted from the instrumental variable regression $Y\sim 1+ D|1,Z$, and Row 4 reports the nominal 95\% randomization-based confidence set. Rows 4–6 present similar results with covariate adjustments. Row 5 reports the nominal 95\% confidence interval resulted from the instrumental variable regression $Y\sim 1+ D+Z\cdot X+(1-Z)\cdot X|1,Z,Z\cdot X,(1-Z)\cdot X$.    }
    \label{tab3}
\end{table}

Tables \ref{tab2} and \ref{tab3} present results from six cities in \cite{green2003getting}. The tables include point estimates for the LATEs and nominal 95\% confidence intervals. Overall, the confidence intervals based on the 2SLS estimator are similar to the randomization-based intervals. Additionally, incorporating voting records from the previous election as a covariate %improves the length of 
shortens the confidence intervals, enabling better inference.

\appendix

\section{Notations and Definitions}\label{s:notation}

We let $\mathbb{N}$ denote the positive integers and recall that $\onen$ denotes the set $\{1,\dots ,n\}$. We let $C>0$ denote a constant (not depending on $n$), which may change line to line and depend on various other parameters. When $X_n$ and $Y_n$ are deterministic, we use the standard asymptotic notation $X_n = O (Y_n)$ for $n$-dependent quantities $X_n$ and $Y_n$, with $Y_n>0$, to mean that 
$\limsup_{n \rightarrow \infty} |X_n|/ {Y_n} < \infty$.
We also write $X_n \lesssim Y_n$ for $X_n = O(Y_n)$. We similarly write $ X = o(Y_n)$ if the limit is $0$. 

For stochastic $o$ and $O$ symbols, we follow the conventions in \cite{van2000asymptotic}. 
We write $X_n = O_p(1)$ to mean that for every $\epsilon >0$, there exists $M_\epsilon, N_\epsilon>0$  such that $\P( |X_n| > M_\epsilon ) < \epsilon$ for all $n > N_{\epsilon}$, where $N_{\epsilon}$ is an integer. 
We write $X_n = O_{p}(Y_n)$ when we mean that $X_n=Y_nR_n$ for some $R_n=O_p(1)$ and $X_n=o_p(Y_n)$ when $X_n/Y_n$ converges to zero in probability. When the statements hold almost surely, we replace them with $o_{a.s.}$ and $O_{a.s.}$.

We also define a set $\mathcal A$ of always-takers and a set $\mathcal N$ of never-takers by 
\be
\mathcal A = \{ i \in \onen : d_i(0) = 1,   d_i(1) =1 \},
\qquad
\mathcal N = \{ i \in \onen : d_i(0) = 0,   d_i(1) =0  \}.
\ee
Let $\mathcal T(1),\mathcal T(0) \subset [n]$ denote the sets of subjects in the treatment and control groups defined by $Z$, respectively, so that 
\be
\mathcal T(1) = \{ i \in \onen : Z_i =1\},
\qquad
\mathcal T(0) = \{ i \in \onen : Z_i=0 \}.
\ee

\subsection{Variances and Covariances} For any vector $ \bs q = (q_i)_{i=1}^n \in \R^n$, we define the population mean and variance by
\be
\mu_{\bs q} = \frac{1}{n-1} \sum_{i=1}^n q_i, \qquad \sigma^2_{\bs q} = \frac{1}{n-1} \sum_{i=1}^n ( q_i - \mu_{\bs q})^2. 
\ee 
We also define within-group variances for each group $G \in \{\mathcal A, \mathcal C, \mathcal N \}$ by 
\be
\sigma^2_{\bs q, G} = \frac{1}{|G|-1} \sum_{i \in G} ( q_i - \mu_{\bs q})^2,
\ee
if $|G| \ge 2$ and zero otherwise.

Given an additional vector  $ \bs \tilde q = (\tilde q_i)_{i=1}^n \in \R^n$, we define the population covariance by
\be
\sigma_{\bs q \bs \tilde q} = \frac{1}{n} \sum_{i=1}^n (q_i - \mu_{\bs q})(\tilde q_i - \mu_{\bs \tilde q}),
\ee
and the within-group covariances by 
\be
\sigma_{\bs q \bs \tilde q, G} = \frac{1}{|G|-1} \sum_{i\in G} (q_i - \mu_{\bs q})(\tilde q_i - \mu_{\bs \tilde q}),
\ee
if $|G| \ge 2$ and zero otherwise.\\

\subsection{Normal Quantiles} Let $\mathcal{Z}$ be a mean zero, variance one normal random variable. We frequently denote its distribution by $N(0,1)$. For an $\alpha \in (0,1)$, define $z_{1-\alpha}\in \R$ as the value such that  
$
\P ( \mathcal{Z} > z_{1 - \alpha} ) = \alpha. 
$
We note that for $\alpha\in (0,1)$, $z_{1-\alpha/2}$ is the $1-\alpha$ quantile of the absolute value of a standard normal random variable because $
\P ( |\mathcal{Z}| > z_{1 - \alpha/2} ) = \alpha. 
$
%\begin{equation}
%    \P (|\mathcal{Z}|\geq z_{1-\alpha/2}) = P(\mathcal{Z}\geq z_{1-\alpha/2}) +  P(\mathcal{Z}\leq -z_{1-\alpha/2})=\alpha.
%\end{equation}
\section{Useful Theorems}

\subsection{Variance Representation}
Recall the characterization of $\widehat{\tau}(\beta)$ in (\ref{eq:dim_numerator}). 
The following lemma is an immediate consequence of \cite[Theorem 6.2]{imbens2015causal}. 
\begin{lemma}\label{l:imbensrubin}
For all $\beta \in \R$, the variance of $\hat \tau(\beta)$ is
\begin{equation}\label{eqn:ARvariance1}
   \Var\big(\hat \tau (\beta)\big) = \frac{S^2_1(\beta)}{n_1} + \frac{S^2_0(\beta)}{n_0} - \frac{S^2_{10}(\beta)}{n}, 
\end{equation}
where 
\begin{align}
    S^2_1(\beta) &=\frac{\pi^2(1-\pi)^2}{n-1}\sum_{i\in [n]} v_i(1,\beta)^2,
    \label{eqn:variances1}\\
    S^2_0(\beta) &=\frac{\pi^2(1-\pi)^2}{n-1}\sum_{i\in [n]} v_i(0,\beta)^2 \label{eqn:variances2}, \\
    S^2_{10}(\beta) &=\frac{\pi^2(1-\pi)^2}{n-1}\sum_{i\in [n]} 
    \big( v_i(1,\beta) - v_i(0,\beta) \big)^2\label{eqn:variances3} \\
    v_{i}(1,\beta)&= y_i(d_i(1))-\beta d_i(1) - \frac{1}{n}\sum_{i\in [n]}\big(y_i(d_i(1))-\beta d_i(1)\big), \label{v11}\\
 v_{i}(0,\beta)&= y_i(d_i(0))-\beta d_i(0) - \frac{1}{n}\sum_{i\in [n]}\big(y_i(d_i(0))-\beta d_i(0)\big), \label{v00}. 
\end{align}
\end{lemma}

\subsection{Combinatorial Central Limit Theorem}
Consider a completely randomized experiment where $n_1$ out of $n$ subjects are assigned to the treatment. Let $Z_i$ be the binary variable indicating the assignment status of the $i$th subject. Given a collection of ordered pairs, $\{(A_i(1),A_i(0))\}_{i=1}^n$, we define the difference in means estimator
\begin{equation*}
    \hat \tau_A =\frac{1}{n_1}\sum_{i\in [n]} Z_i A_i(1) - \frac{1}{n_0}\sum_{i\in [n]} (1-Z_i)A_i(0).
\end{equation*}

For $a\in \{0,1\}$, define $A(a) = \{ A_1(a), A_2(a), \dots, A_n(a) \}$. 
%Define 
%$\mu_{A(a)}=\frac{1}{n}\sum_{i=1}^nA_i(a)$ where $a\in\{0,1\}$.  
The following theorem is adapted from \cite[Theorem 4]{li2017general}. 
\begin{theorem}\label{theorem:clt}
Consider a sequence of completely randomized experiments where in the $n$th experiment, $n_{1n}$  out of $n$ units are assigned to treatment, and define $n_{0n}=n-n_{1n}$. Suppose that
 \begin{equation}\label{noteme}
    \lim_{n \rightarrow \infty} \max_{a\in\{0,1\}}\frac{1}{n_{an}^2}\frac{\max_{i\in [ n]} \big| A_i(a)-\mu_{\bs A(a)}  \big|^2  }{\Var(\hat{\tau}_A)} = 0. 
\end{equation}
Then 
\begin{equation*}
    \frac{\hat{\tau}_A-\E[\hat{\tau}_A] }{\sqrt{\Var(\hat{\tau}_A )}} \xrightarrow{d} N(0,1).
\end{equation*}
\end{theorem}
%\begin{remark}\label{r:seqnote}
%In \eqref{noteme}, the $y_i$ in the $n$th term of the limit are the potential outcome functions from $\theta_n$. However,
%we have omitted the $n$-dependence in order to streamline the expressions.  We  adopt this convention for the $y_i$, $d_i$, $a_i$, and $Y_i$ in all limits below.
%\end{remark}

\subsection{Difference in Means Estimator}

%Recall the definition of $\widehat{\tau}(\beta)$ in (\ref{eq:dim_numerator}).  
The lemma below follows from a standard computation, so we omit the proof.
\begin{lemma}\label{l:unbiased}
We have $\E[ \hat \tau(\beta) ] = 0$ when $\beta = \frac{1}{|\mathcal{C}|}\sum_{i\in \mathcal{C}}\left(y_i(1)-y_i(0)\right)$. 
\end{lemma}

\section{Proof of \texorpdfstring{\Cref{t:main}}{Proof for Confidence Sets Without Covariate Adjustment}}\label{s:mainproof}

Throughout this section, we let $\delta, r, A$ be the constants fixed in the statement of \Cref{t:main}, and set $\Theta_n =\Theta_n (\delta, r , A)$. We note that $\Theta_n$ encodes Assumptions \ref{a:exclusion}-\ref{a:cr}. 
We typically denote elements of $\Theta_n$ by $\theta_n$.
%When we say "for all $\theta_n\in\Theta_n$, (statement)", we mean that the statement is true for any parameter and experimental design in the set $\Theta_n$. 
For each model $\theta_n$, we denote its local average treatment effect by $\cace(\theta_n)$. 

%In \Cref{lemma:lyapunov}, and 
In various lemmas below, we will consider subsequences $( \theta_{n_k})_{k=1}^\infty$ of a given sequence $(\theta_n)_{n=1}^\infty$ of models. We let $n_{1k}$ denote the number of elements in the treatment group of $\theta_{n_k}$, and $n_{0k}$ the number of elements in the control group of $\theta_{n_k}$, so that $n_{0k} + n_{1k} = n_k$.

\subsection{Preliminary Results}
Recall the definitions of $w_i(1)$ and $w_i(0)$ in Assumption \ref{d:theta}. 
In results on sequences $(\theta_n)_{n=1}^\infty$, we will generally omit the $n$-dependence on $w_i(1)$, $w_i(0)$, and other model-dependent quantities for brevity. 
%\Cref{lemma:max} derives a bound on the ratio between the maximum deviation of potential outcomes from their means and the variances of those outcomes.

\begin{lemma}\label{lemma:max}
For all $\theta_n \in \Theta_n $ and $a\in\{0,1\}$,
\begin{equation*}
   \frac{ \max_{i\in [n]} |w_i(a)-\mu_{\bs w(a)}|}{\sigma_{\bs w(a)}}   \leq A  n^{1/4}.
\end{equation*}  
\begin{comment}
and 
\begin{equation*}
     \max_{i\in [n]} 
     \frac
     {|y_i(a)-\mu_{\bs y(a)}|^2}
     {\sum_{i=1}^n   \left(y_i(a)-\mu_{\bs y(a) }\right)^2}   \leq A  N^{-1/2}.  
\end{equation*}
\end{comment}
\end{lemma}
\begin{proof}
For all $a\in\{0,1\}$, we have 
\begin{align}
 \max_{i\in [n]}  |w_i(a)-\mu_{\bs w(a)}|  
 &= \max_{i\in [n]}  \left(|w_i(a)-\mu_{\bs w(a)}|^{4}  \right)^{1/4}\\
 & \leq n^{1/4}\left(\frac{1}{n}\sum_{i\in [n]} |w_i(a)-\mu_{\bs w(a)}|^{4}\right)^{1/4}
  \leq A n^{1/4}\sigma_{\bs w(a)}\label{ineq:02072024_a},
 %A  \left(\frac{1}{n} \sum_{i\in [n]} \left(y_i(a)-\mu_{\bs y(a)} \right)^2\right)^{\frac{1}{2}}  \label{ineq:02072024_a},
\end{align}
where %\eqref{ineq:02072024_a}
the last inequality
is by \eqref{fourthmoment}.
\end{proof}

\begin{lemma}\label{lemma:variancelowerbound1}
For all $\theta_n \in \Theta_n $, if $\beta_n= \cace(\theta_n)$, then
\begin{equation}
   \Var\big( \hat{\tau} (\beta_n )\big) \geq (1 - \delta )\left(\frac{n_0}{nn_1}S^2_1(\beta_n) +  \frac{n_1}{nn_0} S^2_0(\beta_n) \right).
\end{equation}
%where $\epsilon_{A\ref{assumption:parameterspace}}$ is defined in \Cref{d:theta}.
\end{lemma}

\begin{proof}
We write $v_i(1,\beta_n)=v_i(1)$ and $v_i(0,\beta_n)=v_i(0)$ for simplicity. We note that $v_i(1)$ and $v_i(0)$ is a mean-centered version of $w_i(1)$ and $w_i(0)$. Moreover, $v_i(1)-v_i(0)=w_i(1)-w_i(0)$ for all $i$, and $\sum_{i\in \mathcal C} \left(v_i(1) - v_i(0)\right) = \sum_{i\in \mathcal C} \left(w_i(1) - w_i(0)\right)=0$.

For the sum over the indices in \eqref{eqn:variances3} corresponding to compliers, we find, when $|C|\geq2$,
\begin{align}
 &\frac{1}{| \mathcal C |-1 }\sum_{i \in \mathcal C}\big(v_i(1)-v_i(0)\big)^2 =\frac{1}{| \mathcal C | -1 }\sum_{i \in \mathcal C}\big(w_i(1)-w_i(0)\big)^2 \\
 &= \sigma^2_{\bs w (1),\mathcal C} +  \sigma^2_{\bs w(0),\mathcal C} - 2 \sigma_{\bs w (1), \bs w(0), \mathcal C} \\
& \leq \sigma^2_{\bs w(1), \mathcal C} +  \sigma^2_{\bs w(0), \mathcal C} +  2 \delta \sigma_{\bs w(1), \mathcal C} \cdot  \sigma_{\bs w(0), \mathcal C} \label{431} \\
&\leq \sigma^2_{\bs w(1), \mathcal C} +  \sigma^2_{\bs w(0), \mathcal C} +  \frac{\delta n_0}{n_1}\sigma^2_{\bs w(1), \mathcal C} 
+ \frac{\delta n_1}{n_0} \sigma^2_{\bs w(0), \mathcal C} \label{ineq_01302024_b}\\
 & = \left(1 + \frac{\delta n_0}{n_1}\right)\sigma^2_{\bs w(1), \mathcal C}+\left(1 +   \frac{\delta n_1}{n_0}\right)\sigma^2_{\bs w(0), \mathcal C}\label{penult}\\
 & \leq \left(1 + \frac{\delta n_0}{n_1}\right) \frac{1}{| \mathcal C | -1 }\sum_{i \in \mathcal C}v_i^2(1)  + \left(1 +   \frac{\delta n_1}{n_0}\right) \frac{1}{| \mathcal C |-1 }\sum_{i \in \mathcal C}v_i ^2(0). \label{ineq_01302024_d}
\end{align}
In \eqref{431} we used \eqref{ii}, in \eqref{ineq_01302024_b} we used the elementary inequality $2ab \le a^2 + b^2$, and in \eqref{ineq_01302024_d} we use the fact that
\begin{align}
     \sigma^2_{\bs w(1), \mathcal C} & = \frac{1}{| \mathcal C | -1 }\sum_{i \in \mathcal C} \left(w_i(1)- \mu_{w(1),C}\right)^2 \\
    & \leq \frac{1}{| \mathcal C | -1 }\sum_{i \in \mathcal C} \left(w_i(1)-  \frac{1}{n}\sum_{i\in [n]}\big(y_i(d_i(1))-\beta_n d_i(1)\big)\right)^2 =    \frac{1}{| \mathcal C | -1 }\sum_{i \in \mathcal C} v_i^2(1),
\end{align}
and similarly for the term with $\sigma^2_{\bs w(0), \mathcal C}$. For $|C|=1$, $v_i(1)-v_i(0)=0$ and the inequality trivially holds.

Furthermore, $\sum_{i\in \mathcal C} \left(v_i(1) - v_i(0)\right) = 0$ implies that
\be
\sum_{i\in \mathcal A \cup \mathcal N} \left(v_i(1) - v_i(0)\right) = 0,
\ee
because $v(1)$ and $v(0)$ have mean zero, by definition. 
Since $v_i(1) - v_i(0)$ is independent of $i$ for $i\in \mathcal A \cup \mathcal N$ (i.e., $y_i(d_i(1))-\beta_n d_i(1) - \left(y_i(d_i(0))-\beta_n d_i(0)\right)=0$), this implies that for such $i$ we have $v_i(1) = v_i(0)$.
 Therefore \eqref{ineq_01302024_d}  implies %Under \Cref{d:theta}, if $\frac{n_C}{n}\geq \frac{1}{\sqrt{n\log n}}$ and $\beta=\frac{1}{n_C}\sum_{i \in \mathcal C} \left(Y_i(1)-Y_i(0)\right)$, 
\begin{equation}\label{lem431}
        S^2_{10}(\beta) \leq \frac{\pi^2(1-\pi)^2}{n-1} \left( \left (1 +  \frac{\delta n_0}{n_1} \right) \sum_{i\in [n]} v^2_i(1) + \left(1+  \frac{\delta n_1}{n_0} \right) \sum_{i\in [n]}  v^2_i(0) \right).
\end{equation}
The claim then follows by combining \eqref{lem431} and \eqref{eqn:ARvariance1}.
\end{proof}

\begin{lemma}\label{lemma:lyapunov}
Consider a sequence $\left(\theta_n\right)_{n=1}^\infty$ such that $\theta_n\in\Theta_n$ and let $\beta_n=\cace(\theta_n)$.  Let $(\theta_{n_k})_{k=1}^\infty$ be any of its subsequences. For each $a\in \{0,1\}$, we have 
\begin{equation}
\lim_{k \rightarrow \infty}
\frac{1}{n_{ak}^2} 
\frac{\max_{i\in [n_k]} \left|w_i(a)-\mu_{w(a)}\right|^2}
{ \Var(\widehat{\tau} (\beta_{n_k})) } = 0.
\end{equation}

\end{lemma}

\begin{proof}
For simplicity, we map the indices of the subsequence back to $\{1,2,3,\dots\}$.  We write $v_i^2(1)=v_i^2(1,\beta_n)$ and $v_i^2(0)=v_i^2(0,\beta_n)$ (recalling (\ref{v11}) and (\ref{v00})).
By Lemma~\ref{lemma:variancelowerbound1}, there exists a constant $c >0$ independent of $n$ such that 
\begin{equation}\label{eqn:lowerboundvariance}
    (n-1)\Var\left( \widehat{\tau}(\beta) \right) \geq c \left( \frac{1}{n_1}\sum_{i\in [n]}v_{i}^2(1) + \frac{1}{n_0}\sum_{i\in [n]} v^2_{i}(0)  \right).
\end{equation}
Hence, for each $a\in\{0,1\}$,
\begin{align}
    \frac{1}{n_a^2} %&
    \frac{\max_{i\in [n]} \left|w_i(a)-\mu_{w(a)}\right|^2}{ \Var(\widehat{\tau}(\beta)) } %\label{ineq_02092024_a}%\\
    & \le  \frac{n-1}{c n_a^2} \frac{\max_{i\in [n]} \left|w_i(a)-\mu_{w(a)}\right|^2}{ n_1^{-1}\sum_{i\in [n]}v_{i}^2(1) + n_0^{-1} \sum_{i\in [n]} v^2_{i}(0) } \label{ineq_02012024_e} \\
    & \le\frac{1}{c n_a } 
    \frac{\max_{i\in [n]} \left|w_i(a)-\mu_{w(a)}\right|^2}
    { \frac{1}{n-1}\sum_{i\in [n]} v_{i}^2(a)  } \label{ineq_02012024_g} \\
   %& \leq C \left( 
    %\frac{\max_{i\in [n]}|y_i(a)-\mu_{\bs y(a)}|^2}{\sum_{i\in [n]}w_{i}^2(a) } +\frac{\max_{i\in [n]} \left|\beta_n \left(d_i(a)-\mu_{\bs d(a)}\right)\right|^2}{ \sum_{i\in [n]}w_{i}^2(a) }  
    %\right)  \label{ineq_02012024_h} \\
  % & \lesssim  \frac{\max_{i\in [n]}|y_i(a)-\mu_{\bs y(a)}|^2}{\sum_{i\in [n]}\left(y_i(a)-\mu_{\bs y(a)}\right)^2}  +  \frac{\beta^2\max_{i\in [n]} \left|d_i(a)-\mu_{\bs d(a)}\right|^2}{\beta^2\sum_{i\in [n]}\left(d_i(a)-\mu_{\bs d(a)} \right)^2}   \label{ineq_02012024_i}\\
   & \lesssim  n^{-1/2}  \label{ineq_02012024_j},
\end{align}
which implies the claim since \eqref{ineq_02012024_j} is $o(1)$. 
Here \eqref{ineq_02012024_e} is by \eqref{eqn:lowerboundvariance}, \eqref{ineq_02012024_g} %and \eqref{ineq_02012024_h} 
is an algebraic manipulation, and \eqref{ineq_02012024_j} is \Cref{lemma:max} and the fact $ \frac{1}{n-1}\sum_{i\in [n]} v_{i}^2(a)=\sigma^2_{w(a)}$.
\end{proof}

\begin{theorem}\label{theorem:ARstatisticsCLT} 
For every $\alpha\in (0,1)$,
    \begin{equation}
  \lim_{n\rightarrow \infty}   
   \inf_{\beta \in \R} 
   \inf_{ \{\theta \in \Theta_n : \cace(\theta)=\beta\}} \P_{\theta}
   \left( \left| \frac{\hat{\tau}(\beta)}{\sqrt{\Var\left( \hat \tau (\beta) \right)}}\right| > z_{1-\alpha/2} \right)= \alpha.\label{theorem:ARstatisticsCLT:statement1}
\end{equation}
\end{theorem}
\begin{proof}
    We proceed by contradiction. Suppose \eqref{theorem:ARstatisticsCLT:statement1} does not hold. Then there exists a subsequence of models $\{ \theta_{n_k} \}_{k=1}^\infty$ and  $\delta > 0$ 
     such that 
    \begin{equation*}
\left|\P_{\theta_{n_k}}
   \left( \left|\frac{\hat{\tau}(\beta)}{\sqrt{\Var\left( \hat \tau (\beta) \right)}}\right| > z_{1-\alpha/2} \right) - \alpha\right|>\delta,
    \end{equation*}
    infinitely often as $n_k \rightarrow \infty$. 
    However, by Lemma \ref{lemma:lyapunov} and Theorem \ref{theorem:clt}, we have 
    \begin{equation*}
\lim_{k\to\infty} \P_{\theta_{n_k}}\left( \left|\frac{\hat{\tau}(\beta)}{\sqrt{\Var\left( \hat \tau (\beta) \right)}}\right| > z_{1-\alpha/2} \right)=  \alpha.
\end{equation*}    
Hence we have a contradiction and the theorem is proved. 
\end{proof}

\subsection{Variance Estimator}

\begin{lemma}\label{lemma:varianceconsistency}
Consider a sequence $\left(\theta_n\right)_{n=1}^\infty$ such that $\theta_n\in\Theta_n$ and let $\beta_n=\cace(\theta_n)$.  Let $(\theta_{n_k})_{k=1}^\infty$ be any of its subsequences.
Then 
\begin{equation}\label{consistentvar}
\frac{\hat \sigma^2 (\beta_{n_k}) }{ n_1^{-1} S^2_1(\beta_{n_k})  +  n_0^{-1} S^2_0(\beta_{n_k})  } \overset{p}{\to} 1.
\end{equation}

\end{lemma}

\begin{proof} 
For simplicity, we map the indices of the subsequence back to $\{1,2,3,\dots\}$. We begin by showing 
 \be \label{n1lim}
  \frac{\pi^2 (1 - \pi)^2 n_1^{-1} (n_1 - 1)^{-1} \sum_{i\in [n]} Z_i\hat v^2_i(1)}{ n_1^{-1} S^2_1(\beta_n)  } \overset{p}{\to} 1,
 \ee
 where $\hat v_i(1)$ is defined using the value $\beta_n$. 
 The analogous limit holds for the control group, by similar reasoning, and together these claims imply \eqref{consistentvar} and complete the proof.
 To show \eqref{n1lim}, it suffices to show
\be \label{n1lim2}
  \frac{\pi^2 (1 - \pi)^2 n_1^{-1}  \sum_{i\in [n]} Z_i\hat v^2_i(1)}{S^2_1(\beta_n)  } \overset{p}{\to} 1,
 \ee
since $(n_1-1)^{-1}/n_1^{-1}\to 1$ by Assumption \ref{a:cr}. 

We have the identity 
\begin{equation*}
 \frac{1}{n_1} \sum_{i\in [n]} Z_i \hat v_i^2 (1) = \frac{1}{n_1}\sum_{i\in [n]} Z_i v^2_i(1) - \left(\frac{1}{n_1} \sum_{i\in [n]} Z_i\big(y_i(1)-\beta_n d_i(1) \big)- \mu_{\bs y(1)}-\beta_n \mu_{\bs d(1)} \right)^2,
\end{equation*}
which we use to expand the numerator in \eqref{n1lim2}, creating two terms. The second term in this expansion converges to zero in probability, by Markov's inequality and the estimate
\begin{align}
     \E\left[ \frac{\left(n_1^{-1}\sum_{i\in [n]} Z_i\big(y_i(1)-\beta_n d_i(1) \big) - \mu_{\bs y(1)}-\beta_n \mu_{\bs d(1)}\right)^2}{S^2_1(\beta_n)}\right]\notag &=\frac{n_0}{n_1 n(n-1)} 
     \frac{\sum_{i\in [n]}v^2_i(1)}{S^2_1(\beta_n)}\\ & = O\left(\frac{n_0}{n_1 n }\right) = o(1).
\end{align}
For the other term in the expansion of \eqref{n1lim}, note that 
 \be\label{b96}
 \E\left[\frac{1}{n_1}\sum_{i\in [n]} Z_i v^2_i(1)\right]=\frac{1}{n}\sum_{i\in [n]}v^2_i(1),
 \ee
which implies
\be 
\frac{n-1}{n_1} \frac{ \E\left[ \sum_{i\in [n]} Z_i v^2_i(1)\right] }{ \sum_{i\in [n]} v^2_i(1) }=1 - \frac{1}{n}.
\ee
Further, 
\begin{align}
\Var\left( \frac{1}{n_1 } \frac{ \sum_{i\in [n]} Z_i v^2_i(1)}{S^2_1(\beta)}  \right)
&\lesssim\frac{n_0}{n_1 n(n-1)}  \frac{\sum_{i\in [n]}v^4_{i}(1)}{ \sigma^4_{\bs v (1)}}\\
&= \frac{1}{n-1} \frac{\sum_{i\in [n]}v^2_i(1)}{\sigma^2_{\bs v(1)}} \cdot \frac{n_0}{n_1 n} \label{469}
\frac{\max_{i}|v_{i}(1)|^2}{\sigma^2_{\bs v (1)}}%\\
%& 
= o(1),
\end{align}
where the maximum in \eqref{469} is bounded using \Cref{lemma:max}.  
Combining these results implies \eqref{n1lim}.
\end{proof}

%Lemma \ref{lemma:varianceconsistency} and Lemma \ref{theorem:ARstatisticsCLT} yield \eqref{sampling:ARstatistics}).
\begin{theorem}\label{thm:clt2}
 For all $\alpha \in (0,1)$, 
\begin{equation}
    \lim_{\kappa \rightarrow 0^+}
    \lim_{n\rightarrow \infty}   
   \sup_{\beta \in \R} 
   \sup_{ \{\theta \in \Theta_n : \cace(\theta)=\beta\}} \P_{\theta}
   \left(\big|\Delta(\beta)\big| > z_{1-\alpha/2} - \kappa \right)\leq \alpha.
\end{equation}
\end{theorem}
\begin{proof}
For all $x, y >0$, we have 
\be\label{eqn:02092024_a}
   \P_\theta
   \left( \big|\Delta(\beta)\big| >  y  \right) 
   \leq  
   \P_\theta\left(
   \frac{\sqrt{\Var\big( \hat  \tau (\beta)\big) }}{\hat{\sigma}_{\beta}} \geq 1 + x 
   \right) 
   +   
   \P_\theta\left(\left| \frac{\hat{\tau}(\beta)}{\sqrt{\Var\left( \hat{\tau}(\beta)\right) }} \right| >\frac{y}{1+x}\right) .
\ee
Denote 
\be
\tilde{\sigma}^2_{\beta}=\frac{S^2_1(\beta)}{n_1} + \frac{S^2_0(\beta)}{n_0}.
\ee 

We first claim that for all $x>0$,
\be\label{varclaim}
\lim_{n\rightarrow \infty}   
   \sup_{\beta \in \R} 
   \sup_{ \{\theta \in \Theta_n : \cace(\theta)=\beta\}} \P_{\theta}
   \left(\frac{\sqrt{\Var\left( \hat{\tau}(\beta)\right) }}{\hat{\sigma}(\beta)}  \geq 1 + x\right) =0.
\ee
This claim holds because 
\be
\frac{\Var\left( \hat{\tau}(\beta)\right) }{\tilde{\sigma}^2_{\beta}}\leq 1,\ee
by \Cref{l:imbensrubin}, and
\be
  \lim_{n\rightarrow \infty}   
   \inf_{\beta \in \R} 
   \inf_{ \{\theta \in \Theta_n : \cace(\theta)=\beta\}} \P_{\theta}
   \left( \left|\frac{\tilde{\sigma}_{\beta}}{\hat{\sigma}(\beta)} -1 
   \right| > \epsilon \right)= 0
\ee
for every $\epsilon >0$, 
%\be
%\frac{\tilde{\sigma}_{\beta}}%{\hat{\sigma}_{\beta}} \overset{p}{\to} 1
%\ee 
by \Cref{lemma:varianceconsistency}, after arguing by contradiction as in \Cref{theorem:ARstatisticsCLT}.

We next consider the second term in \eqref{eqn:02092024_a}. Let $\mathcal{Z}$ be a standard normal random variable. We have
\begin{equation}\label{e:jan2}
    \lim_{x\to 0} 
    \lim_{n\rightarrow \infty}   
    \sup_{\beta \in \R} 
   \sup_{ \{\theta \in \Theta_n : \cace(\theta)=\beta\}} \P_{\theta}\left(\left| \frac{\hat{\tau}(\beta)}{\sqrt{\Var\left( \hat{\tau}(\beta)\right) }} \right|>\frac{y}{1+x} \right)  = \P\big(|\mathcal{Z}|\geq y\big)
\end{equation}
by Theorem \ref{theorem:ARstatisticsCLT} and the continuity of the quantile function of the absolute value of a normal distribution. 

Since the left-hand side of \eqref{eqn:02092024_a} is independent of $x$, combining \eqref{varclaim} and \eqref{e:jan2} gives
\begin{equation}
 \lim_{\kappa\to 0^+}
 \lim_{n\rightarrow \infty}   
    \inf_{\beta \in \R} 
   \inf_{ \{\theta \in \Theta_n : \cace(\theta)=\beta\}} \P_{\theta}\left( \big|\Delta(\beta) \big| > z_{1-\alpha/2}- \kappa \right )\leq \alpha,   
\end{equation}
where we again using the continuity of the quantile function for the normal distribution. This completes the proof. 
\end{proof}
\subsection{Randomization Critical Values}\label{subsection:randomizationcriticalvalue}
We write $\E^*$ and $\Var^*$ for expectation and variance under the probability measure $\P^*$. By \Cref{l:unbiased}, for any $\beta$ we have 
\be
\E^* \big[ \hat \tau^* (\beta) \big] = 0,
\ee 
and by \Cref{l:imbensrubin},
\begin{equation}\label{varrep2}
   \Var^*(\hat{\tau}^*(\beta)) = \frac{S^{*2}_1(\beta)}{n_1} + \frac{S^{*2}_0(\beta)}{n_0},
\end{equation}
where
\begin{align}\label{eqn:ARvariance2}
    &S^{*2}_1(\beta) = S^{*2}_0(\beta)=\frac{\pi^2(1-\pi)^2}{n-1}\sum_{i\in [n]} \big(Y_i-\beta D_i- \frac{1}{n}\sum_{i=1}^n \left(Y_i-\beta D_i\right)  \big)^2.
\end{align}

\begin{comment}
$S^{*2}_{1}(\beta)$ and $S^{*2}_{0}(\beta)$ can be rewritten as:
\begin{align}
   &\frac{\pi^2(1-\pi)^2}{n-1}\left(\sum_{i\in [n]} Z_i \left(Y_\mathcal{I}(\beta)-\left(\mu_{\bs Y}-\beta\mu_{\bs D}\right)\right)^2 +\sum_{i\in [n]} \left(1-Z_i\right) \left(Y_\mathcal{I}(\beta)-\left(\mu_{\bs Y}-\beta\mu_{\bs D}\right)\right)^2\right) \\      
    =& \frac{\pi^2(1-\pi)^2}{n-1}\sum_{i\in [n]} Z_i \widehat{W}_{i,\beta}(1)^2 + \frac{\pi^2(1-\pi)^2}{n-1}\sum_{i\in [n]} \left(1-Z_i\right) \widehat{W}_{i,\beta}(0)^2 \\
    & + \frac{\pi^2(1-\pi)^2n_1}{n-1}\left(  \frac{1}{n_1}\sum_{i\in [n]}Z_i\left(Y_i(1)-\beta d_i(1)\right)- \left(\mu_{\bs Y}-\beta\mu_{\bs D}\right) \right)^2 \\
    & + \frac{\pi^2(1-\pi)^2n_0}{n-1}\left(  \frac{1}{n_0}\sum_{i\in [n]}
    \left(1-Z_i\right)\left(Y_i(0)-\beta d_i(0)\right)- \left(\mu_{\bs Y}-\beta\mu_{\bs D}\right) \right)^2,
\end{align}
where $\widehat{W}_{i,\beta}(1)$ and $\widehat{W}_{i,\beta}(0) $ are defined before Lemma \ref{lemma:varianceconsistency}.
\end{comment}

The following lemma is adapted from  \cite[Lemma A.3]{wu2021randomization}. It describes the almost sure behavior of realized data moments. Recall for each $\theta_n\in\Theta_n$, we let $\beta_n$ denote its LATE and recall that $w_i(a)=y_i(d_i(a))-\beta_n d_i(a)$.

\begin{lemma}\label{lemma:tail_inequality}
Consider a sequence $\left(\theta_n\right)_{n=1}^\infty$ such that $\theta_n\in\Theta_n$ and let $\beta_n=\cace(\theta_n)$.  Let $(\theta_{n_k})_{k=1}^\infty$ be any of its subsequences. Then we have the following almost sure convergence statements for each $a \in \{0,1\}$:
\begin{equation}\label{almostsureconvergence:20240206_a}
\frac{1}{\sigma_{\bs w(a)}}\left|\frac{1}{n_a} \sum_{i \in \mathcal T(a)}  w_i(a) - \frac{1}{n}\sum_{i\in [n]} w_i(a) \right| \overset{a.s.}{\to}  0,
\end{equation}
\begin{equation}\label{almostsureconvergence:20240206_b}
\frac{1}{\sigma^2_{\bs w(a)}} \left| \frac{1}{n_a}\sum_{i\in \mathcal T(a)}  \left (w_i(a)-\hat \mu_{\bs w(a)}\right)^2 - \frac{1}{n}\sum_{i\in [n]} \left (w_i(a)-\mu_{\bs w(a)}\right)^2    \right| \overset{a.s.}{\to}  0,
\end{equation}
\end{lemma}
\begin{proof} 
The proof of \eqref{almostsureconvergence:20240206_a} is a straightforward adaptation of the proof of \cite[Lemma A3(i)]{wu2021randomization}, so we omit it. The inequality \eqref{almostsureconvergence:20240206_b} follows from  \cite[Lemma A3(ii)]{wu2021randomization}, after noting that \eqref{fourthmoment} implies 
\be
\frac{\frac{1}{n}\sum_{i\in [n]} (w_i(a)-\mu_{\bs w(a)})^4}{  \sigma^4_{w(a)}} \leq A^4.
\ee 
%\ber [Check me!]  \eer 
%    \eqref{almostsureconvergence:20240206_a}) can be proved using the same argument as in the proof of Lemmma A3-(i) in \cite{wu2021randomization}, after applying the tail inequality in their Lemma A.2 with the ratio (in their notation) $\frac{\widehat{\bar{Y}}_n-\bar{Y}_N}{\sqrt{S_N}}$.
    %\eqref{almostsureconvergence:20240206_b}) also follows from \cite{wu2021randomization}'s proof of Lemma A3-(ii) after noticing $\frac{\frac{1}{n}\sum_{i\in [n]} (y_i(a)-\mu_{\bs y(a)})^4}{\sigma^4_{W(a)}} \leq C_{A1}$ by \Cref{d:theta}-(iii).   

\end{proof}

We note that (\ref{almostsureconvergence:20240206_b}) implies that for $a\in\{0,1\}$,
\begin{equation}
    \frac{\frac{1}{n_a}\sum_{i\in \mathcal T(a)}  \left (w_i(a)-\hat \mu_{\bs w(a)}\right)^2  }{\sigma^2_{\bs w(a)}} = 1+o_{a.s.}(1).
\end{equation}

\Cref{lemma:somealgebra} and \Cref{lemma:somealgebra2} record some useful algebraic identifies for later use. We recall  $\mu_Y=\frac{1}{n}\sum_{i=1}^n Y_i$ and $\mu_D=\frac{1}{n}\sum_{i=1}^n D_i$, and define $\hat\mu_{\bs y(a)}=\frac{1}{n_a}\sum_{i\in \mathcal T(a)} y_i(a)$ and $\hat\mu_{\bs d(a)}=\frac{1}{n_a}\sum_{i\in \mathcal T(a)} d_i(a)$ for $a\in \{0,1\}$.
\begin{lemma}\label{lemma:somealgebra}
For any $\beta\in\mathbb{R}$,
\begin{equation}\label{eq:02062024_a}
     \hat\mu_{\bs y(1)}-\beta\hat\mu_{\bs d(1)} - (\mu_{\bs Y} -\beta \mu_{\bs D})=  \frac{n_0}{n} \big( \hat\mu_{\bs y(1)}- \hat\mu_{\bs y(0)} - \beta (  \hat\mu_{\bs d(1)}- \hat\mu_{\bs d(0)}  ) \big),
\end{equation}
and
\begin{equation}\label{eq:02062024_b}
     \hat\mu_{\bs y(0)}-\beta\hat\mu_{\bs d(0)} - \big(\mu_{\bs Y} -\beta \mu_{\bs D}\big)=  \frac{n_1}{n} \big( \hat\mu_{\bs y(0)}- \hat\mu_{\bs y(1)} - \beta (  \hat\mu_{\bs d(0)}- \hat\mu_{\bs d(1)}   ) \big).
\end{equation}
\end{lemma}
\begin{proof}
The identities are purely algebraic. 
Equation \eqref{eq:02062024_a} follows from
 \begin{align}
    & \hat\mu_{\bs y(1)}-\beta\hat\mu_{\bs d(1)} - \left(\mu_{\bs Y} -\beta \mu_{\bs D}\right)\\  &=    \hat\mu_{\bs y(1)}-\beta\hat\mu_{\bs d(1)} - \left(\frac{n_1}{n}\hat\mu_{\bs y(1)} + \frac{n_0}{n}\hat\mu_{\bs y(0)} - \beta\left(\frac{n_1}{n}\hat\mu_{\bs d(1)} + \frac{n_0}{n}\hat\mu_{\bs d(0)} \right)\right)\\
    & = \frac{n_0}{n} \left( \hat\mu_{\bs y(1)}- \hat\mu_{\bs y(0)} - \beta (  \hat\mu_{\bs d(1)}- \hat\mu_{\bs d(0)}   ) \right).
\end{align}       
Equation \eqref{eq:02062024_b} follows from a similar algebraic manipulation. 
\end{proof}
\begin{lemma}\label{lemma:somealgebra2}
For any $\beta\in\mathbb{R}$, every unit $i \in [ n ]$,  and each $a\in \{ 0 ,1\}$, the following holds. On the event $Z_i = a$, 
\begin{align}
     \left|Y_i-\mu_{\bs Y} - \beta (D_i -\mu_{\bs D})\right| \leq& \left|y_i(a)-\mu_{\bs y(a)} -\beta \left(d_i(a)-\mu_{\bs d(a)}\right)\right| \\ &+ \frac{n-n_a}{n}\left| \hat\mu_{\bs y(1)}-\hat\mu_{\bs y(0)}- \beta\left(  \hat\mu_{\bs d(1)}-\hat\mu_{\bs d(0)} \right)\right| \\
    & + \left|\hat \mu_{\bs y(a)}-\mu_{\bs y(a)} - \beta \left(\hat \mu_{\bs d(a)} - \mu_{\bs d(a)}\right) \right|    
\end{align}

\end{lemma}

\begin{proof}
We give the proof only for the case $Z_i = 1$, since the other case is similar. By \eqref{eq:02062024_a}, we have
\begin{align*}
    & \left|Y_i-\mu_{\bs Y} - \beta (D_i-\mu_{\bs D})\right| \\
    = & \left|y_i(1)-\frac{n_1}{n}\hat\mu_{\bs y(1)}-\frac{n_0}{n}\hat\mu_{\bs y(0)} -\beta \left(d_i(1)-\frac{n_1}{n}\hat\mu_{\bs d(1)}-\frac{n_0}{n}\hat\mu_{\bs d(0)}\right) \right|   \\
     = & \left|y_i(1)-\hat\mu_{\bs y(1)}+ \frac{n_0}{n}( \hat\mu_{\bs y(1)}-\hat\mu_{\bs y(0)}) -\beta \left(d_i(1)-\hat\mu_{\bs d(1)}+ \frac{n_0}{n}( \hat\mu_{\bs d(1)}-\hat\mu_{\bs d(0)})\right) \right| \\
     \leq & \left|y_i(1)-\hat\mu_{\bs y(1)} -\beta \left(d_i(1)-\hat\mu_{\bs d(1)}\right)\right| + \frac{n_0}{n}\left| \hat\mu_{\bs y(1)}-\hat\mu_{\bs y(0)}- \beta\left(  \hat\mu_{\bs d(1)}-\hat\mu_{\bs d(0)} \right)\right| \\
     \leq & \left|y_i(1)-\mu_{\bs y(1)} -\beta \left(d_i(1)-\mu_{\bs d(1)}\right)\right| + \frac{n_0}{n}\left| \hat\mu_{\bs y(1)}-\hat\mu_{\bs y(0)}- \beta\left(  \hat\mu_{\bs d(1)}-\hat\mu_{\bs d(0)} \right)\right| \\
     & + \left|\hat\mu_{\bs y(1)}-\mu_{\bs y(1)} - \beta \left(\hat\mu_{\bs d(1)} - \mu_{\bs d(1)}\right)\right|.
\end{align*}
This completes the proof.
\end{proof}

\begin{lemma}\label{lemma:lyapunov_randomization}
Consider a sequence $\left(\theta_n\right)_{n=1}^\infty$ such that $\theta_n\in\Theta_n$. Let $\beta_n=\cace(\theta_n)$. Let $(\theta_{n_k})_{k=1}^\infty$ be any of its subsequences.
Then for each $a\in \{0,1\}$, we have the almost sure limit
\begin{equation}
\frac{1}{n_{ak}^2} \frac{\max_{i\in [n_k]} |Y_i-\mu_{\bs Y}-\beta_{n_k} \left(D_i-\mu_{\bs D}\right)|^2}{  \Var^*(\hat{\tau}^*(\beta_{n_k})) } \overset{a.s.}{\to}  0.
\end{equation}
\end{lemma}
\begin{proof}
For simplicity, we map the indices of the subsequence back to $\{1,2,3,\dots\}$.  First, observe that \Cref{lemma:somealgebra} implies 
\begin{align}
    \sum_{i\in [n]} \left(Y_i- \mu_{\bs Y} -\beta_n D_i + \beta_n \mu_{\bs D}\right)^2 
     =& \sum_{a\in\{0,1\}}\sum_{i\in \mathcal T(a)} \left(y_i(a)- \hat \mu_{\bs y(a)} -\beta_n d_i(a)  + \beta_n \hat \mu_{\bs d(a)}\right)^2\notag
     \\ &+ \sum_{a\in\{0,1\}} n_a\left(\hat \mu_{\bs y(a)} - \beta_n \hat \mu_{\bs d(a)} -\mu_{\bs Y}  + \beta_n \mu_{\bs D} \right)^2\notag
     \\
     =& \sum_{a\in\{0,1\}}\sum_{i\in \mathcal T(a)} \left(y_i(a)- \hat \mu_{\bs y(a)} -\beta_n d_i(a) + \beta_n \hat \mu_{\bs d(a)}\right)^2  \notag
     \\
     &+ 2\frac{n_0n_1}{n}\left( \hat \mu_{\bs y(1)}- \hat \mu_{\bs y(0)} - \beta_n \left(  \hat \mu_{\bs d(1)}- \hat \mu_{\bs d(0)}   \right) \right)^2. \label{intApril5}
\end{align}

Finally, using \eqref{intApril5}, \Cref{lemma:max}, \Cref{lemma:tail_inequality},  \Cref{lemma:somealgebra}, and \Cref{lemma:somealgebra2}, %we have for each $a\in\{0,1\}$ that 
\begin{align}
&\frac{1}{n_a^2} \frac{\max_{i\in [n]} \big|Y_i-\mu_{\bs Y}-\beta_n (D_i-\mu_{\bs D})\big|^2}{  \Var^*(\hat{\tau}^*(\beta_{n}))  }\\ 
& \lesssim \frac{1}{n^2} \frac{\max_{i\in [n]}\big|y_i(1)-\mu_{\bs y(1)} -\beta_n (d_i(1)-\mu_{\bs d(1)})\big|^2}{  \Var^*(\hat{\tau}^*(\beta_{n}))  }\\ & \quad + \frac{1}{n^2} \frac{\max_{i\in [n]}\left|y_i(0)-\mu_{\bs y(0)} -\beta_n\left(d_i(0)-\mu_{\bs d(0) }\right)\right|^2}{  \Var^*(\hat{\tau}^*(\beta_{n}))  }\\
& \quad  + \frac{1}{n^2} \frac{\left| \hat\mu_{\bs y(1)}-\hat\mu_{\bs y(0)}- \beta_n\left(  \hat\mu_{\bs d(1)}-\hat\mu_{\bs d(0)} \right)\right|^2}{ \Var^*(\hat{\tau}^*(\beta_{n})) } \label{he2}\\
&  \quad + \frac{1}{n^2} \sum_{a\in\{0,1\}}\frac{\left|\hat \mu_{\bs y(a)}-\mu_{\bs y(a)}-\beta_n\left(d_i(a)-\mu_{d(a)}\right)\right|^2}{  \Var^*(\hat{\tau}^*(\beta_{n}))  }  \\
& \lesssim \frac{1}{n}\sum_{a\in\{0,1\}}\frac{\max_{i\in [n]}|y_i(a)-\mu_{\bs y(a)}-\beta_n \left(d_i(a)-\mu_{d(a)}\right)|^2}{\sigma^2_{\bs w(a)}(1+o_{a.s.}(1))} \\
& \quad + \frac{1}{n}\sum_{a\in \{0,1\}}\frac{|\hat \mu_{\bs y(a)}-\mu_{\bs y(a)}-\beta_n\left(d_i(a)-\mu_{d(a)}\right)|^2}{\sigma^2_{\bs w(a)}(1+o_{a.s.}(1))} \\%by lemma 8
& 
\quad + \frac{1}{n} \frac{\left( \hat\mu_{\bs y(1)}- \hat\mu_{\bs y(0)} - \beta_n \left(  \hat\mu_{\bs d(1)}- \hat\mu_{\bs d(0)}   \right) \right)^2}{\left( \hat\mu_{\bs y(1)}- \hat\mu_{\bs y(0)} - \beta_n\left(  \hat\mu_{\bs d(1)}- \hat\mu_{\bs d(0)}   \right) \right)^2}\label{haogeemail} \\  
&\overset{a.s.}{\to} 0.
\end{align}
We remark that the term \eqref{haogeemail} comes from \eqref{he2}, after using \eqref{intApril5} to lower bound the variance in the denominator through \eqref{varrep2}.  
\end{proof}
The following lemma establishes the convergence of the variance estimator $\hat \sigma^*(\beta)$. 

\begin{lemma}\label{lemma:randomization_variance_consistency}
Consider a sequence $\left(\theta_n\right)_{n=1}^\infty$ such that $\theta_n\in\Theta_n$. Let $\beta_n=\cace(\theta_n)$. Let $(\theta_{n_k})_{k=1}^\infty$ be any of its subsequences.
Then for every $\kappa>0$, we have 
\begin{equation}
\P^*\left( \left|\frac{\hat \sigma^*(\beta_{n_k})}{\Var^*(\hat{\tau}^*(\beta_{n_k}))} - 1\right| \geq  \kappa \right)=o_{a.s.}(1).
\end{equation}
\end{lemma}
\begin{proof}
The proof is similar to that of Lemma \ref{lemma:varianceconsistency}, using the Markov and Chebyshev inequalities and the estimates in the proof of Lemma \ref{lemma:lyapunov_randomization}.
\end{proof}

\begin{lemma}\label{lemma:cv}
    For every $\alpha \in (0, 1)$ and $\kappa >0$, 
\begin{equation}\label{rndcv}
    \lim_{n\rightarrow \infty}   
   \sup_{\beta \in \R} 
   \sup_{ \{\theta \in \Theta_n : \cace(\theta)=\beta\}} \P_{\theta}\Big(\big|\eta^*_{1-\alpha}(\beta, Z)-z_{1-\alpha/2}\big|\geq \kappa \Big)=0.
\end{equation}
\end{lemma}
\begin{proof}
    We again proceed by contradiction. Suppose \eqref{rndcv} does not hold. Then there exists a subsequence of models $\{ \theta_{n_k} \}_{k=1}^\infty$ and  $\delta_1 > 0$ such that for all $k$,
\be
\P\left( 
\big|\eta^*_{1-\alpha}(\beta, Z)-z_{1-\alpha/2}\big|\geq \kappa
\right) > \delta_1.
\ee 
However, by \Cref{theorem:clt}, \Cref{lemma:lyapunov_randomization},  and   \Cref{lemma:randomization_variance_consistency} %and \Cref{lemma:bookkeeping}, 
the distribution of $\Delta^*(\beta)$ under $\P^*$ 
converges to the normal distribution almost surely. As a consequence of this convergence in distribution and \cite[Lemma 21.2]{van2000asymptotic}, which states that convergence in distribution implies weak convergence of quantile functions, the critical value $\eta^* _{1-\alpha}(\beta, Z)$ converges to $z_{1-\alpha/2}$ almost surely under $\P^*$. This is a contradiction, which proves \eqref{rndcv}. 
\end{proof}

\begin{proof}[Proof of \Cref{t:main}]

We prove 
\begin{equation}
   \lim_{n\rightarrow \infty}   
   \inf_{\beta \in \R} 
   \inf_{ \{\theta \in \Theta_n(\delta,r,A)  : \cace=\beta\}} \left(1-\P_{\theta}\big (\beta  \not\in I^*_{1-\alpha}(Z) \big) \right)\geq 1-\alpha. \end{equation}
Or, equivalently,
\begin{equation}\label{eqn:10092024}
   \lim_{n\rightarrow \infty}   
   \sup_{\beta \in \R} 
   \sup_{ \{\theta \in \Theta_n(\delta,r,A)  : \cace=\beta\}} \P_{\theta}\left (|\Delta(\beta)| > \eta^*_{1-\alpha}(\beta,Z) \right)\leq \alpha. \end{equation}
We have the algebraic identity
\begin{align}
& \P_{\theta}\left (|\Delta(\beta)| > \eta^*_{1-\alpha}(\beta,Z) \right) = \P_{\theta}\left (|\Delta(\beta)| > \eta^*_{1-\alpha}(\beta,Z) , |\eta^*_{1-\alpha}(\beta,Z)-z_{1-\alpha/2}|\geq \epsilon\right) \\
& + \P_{\theta}\left (|\Delta(\beta)| > \eta^*_{1-\alpha}(\beta,Z) , |\eta^*_{1-\alpha}(\beta,Z)-z_{1-\alpha/2}|< \epsilon\right)\\
& \leq \P_{\theta}\left ( |\eta^*_{1-\alpha}(\beta,Z)-z_{1-\alpha/2}|\geq \epsilon\right) + \P_{\theta}\left (|\Delta(\beta)| > z_{1-\alpha/2}-\epsilon \right).
\end{align}

We have
\begin{align}
      & \lim_{n\rightarrow \infty}   
   \sup_{\beta \in \R} 
   \sup_{ \{\theta \in \Theta_n(\delta,r,A)  : \cace=\beta\}} \P_{\theta}\left (|\Delta(\beta)| > \eta^*_{1-\alpha}(\beta,Z) \right)\\
   &\leq      \lim_{n\rightarrow \infty}   
   \sup_{\beta \in \R} 
   \sup_{ \{\theta \in \Theta_n(\delta,r,A)  : \cace=\beta\}}\P_{\theta}\left ( |\eta^*_{1-\alpha}(\beta,Z)-z_{1-\alpha/2}|\geq \epsilon\right)\\
   & \quad +  \lim_{n\rightarrow \infty}   
   \sup_{\beta \in \R} 
   \sup_{ \{\theta \in \Theta_n(\delta,r,A)  : \cace=\beta\}}\P_{\theta}\left (|\Delta(\beta)| > z_{1-\alpha/2}-\epsilon \right).
\end{align}
Note that the left-hand side does not depend on $\epsilon$. For the right-hand side, we have shown that for every $\epsilon>0$,
\begin{equation}
      \lim_{n\rightarrow \infty}   
   \sup_{\beta \in \R} 
   \sup_{ \{\theta \in \Theta_n(\delta,r,A)  : \cace=\beta\}}\P_{\theta}\left ( |\eta^*_{1-\alpha}(\beta,Z)-z_{1-\alpha/2}|\geq \epsilon\right)=0
\end{equation}
by Lemma \ref{lemma:cv},
and 
\begin{equation}
    \lim_{\epsilon\to 0^+}\lim_{n\rightarrow \infty}   
   \sup_{\beta \in \R} 
   \sup_{ \{\theta \in \Theta_n(\delta,r,A)  : \cace=\beta\}}\P_{\theta}\left (|\Delta(\beta)| > z_{1-\alpha/2}-\epsilon \right)\leq \alpha,
\end{equation}
by Lemma \ref{thm:clt2}. This shows the desired statement after using that $\epsilon$ was arbitrary. 
\end{proof}

\section{Regression Adjustment}\label{a:c}
\subsection{Notation}
In this section and the next, we retain the notation and definitions set out in \Cref{s:notation}.  We  let $k,\delta,\tilde{\epsilon}, r, A,B$ be the constants fixed in the statement of \Cref{thm:coverage2}, and set $\Xi_n =  \Xi_n(k,\delta,\tilde{\epsilon},A,B,r)
$. \\

\subsection{Least-Squares Parameters} Suppose a $\beta \in \R$ is given. Then we define, for each $a\in \{0,1\}$, outcomes $\bs w_i(a,\beta)$ by 
\be\label{lastd}
\bs w_i(a,\beta) =y_i(d_i(a))-\beta d_i(a).
\ee

We define the least-squares regression parameters 
\be
\phi(a,\beta) \colon \{0,1\} \times \R \rightarrow \R, \qquad
\bs \gamma (a,\beta) \colon \{0,1\} \times \R \rightarrow \R^k,
\ee
by
\begin{equation}\label{covreg}
    \big(\phi(a,\beta),\bs \gamma(a,\beta)\big) = 
    \argmin_{ (s, \bs t ) \in \R \times \R^k }
    \frac{1}{n}\sum_{i=1}^n \left(w_i(a,\beta)-s-\bs x_i' \bs t \right)^2,
\end{equation}
%where the $t$ in the above expression is a $k\times 1$ column vector, and 
where we interpret $\bs t$ as a column vector. % and use the notation.
We define the residuals for \eqref{covreg} by
\be
\epsilon_i(a, \beta)= w_i(a, \beta) -\phi(a,\beta) -\bs x_i' \bs \gamma(a,\beta).
\ee

The parameters $ \phi^* (a,\beta)$ and $ \gamma^*(a,\beta)$ are defined similarly, by replacing $w_i(a,\beta)$ with $Y_i-\beta D_i$ in \eqref{covreg}, and we define the associated residuals $\epsilon^*_i( \beta)$  by
\be\label{eqn:residual}
\epsilon^*_i( \beta)= Y_i-\beta D_i-\phi^*(a,\beta)-x_i'\gamma^*(a,\beta).
\ee 
We note that $\phi^*(0,\beta)=\phi^*(1,\beta)$ and $\gamma^*(0,\beta)=\gamma^*(1,\beta)$.\\

\subsection{Infeasible Estimator} To analyze the asymptotic behavior of $\hat \tau_\reg(\beta)$, we define the infeasible estimator 
\begin{align*}
    \tilde{\tau}_\reg(\beta) =& \frac{1}{n_1}\sum_{i\in [n]}Z_i \big(Y_i-\beta D_i - x_i'\gamma(1, \beta)\big) - \frac{1}{n_0}\sum_{i\in [n]}(1-Z_i)\big(Y_i-\beta D_i - x_i'\gamma(0, \beta)\big).
\end{align*}
The infeasible estimator $\tilde{\tau}^*_\reg(\beta)$ is defined similarly, replacing $Z_i$ by $Z_i^*$, and $\gamma(1,\beta)$ and $\gamma(0,\beta)$ by $\gamma^*(1,\beta)$ and $\gamma^*(0,\beta)$, respectively.

A straightforward adaptation of \Cref{l:imbensrubin} shows that 
\begin{equation}\label{adjvarrep}
    \Var\big(\tilde{\tau}_\reg(\beta)\big)= \frac{\mathcal{S}^2_{1}(\beta)}{n_1}+\frac{\mathcal{S}^2_{0}(\beta)}{n_0} - \frac{\mathcal{S}^2_{10}(\beta)}{n},
\end{equation}
where
\begin{align}
     &\mathcal{S}^2_{1}(\beta)= \frac{1}{n-1}\sum_{i\in [n]} \epsilon_i(1, \beta) ^2,\\
     &\mathcal{S}^2_{0}(\beta)= \frac{1}{n-1}\sum_{i\in [n]}\epsilon_i(0, \beta)^2,\\
     &\mathcal{S}^2_{10}(\beta)= \frac{1}{n-1}\sum_{i\in [n]}\big(\epsilon_i(1, \beta)-\epsilon_i(0, \beta)\big)^2.     
\end{align}

Similarly,
\begin{equation}
    \Var^*\big(\tilde{\tau}^*_\reg(\beta)\big)= \frac{\mathcal{S}^2_{1,*}(\beta)}{n_1}+\frac{\mathcal{S}^2_{0,*}(\beta)}{n_0} ,
\qquad
     \mathcal{S}^2_{1,*}(\beta)=\mathcal{S}^2_{0,*}(\beta)= \frac{1}{n-1}\sum_{i\in [n]} \epsilon^*_i( \beta) ^2.
\end{equation}

%\section{Proof for Regression Adjustment}

\subsection{Preliminary Results: Regression Algebra and Asymptotic Equivalence} 
This section contains two elementary lemmas for later use. Lemma \ref{lemma14} contains results on some least squares algebra, and Lemma \ref{lemma:asymptoticequi} is an auxiliary lemma for proving asymptotic equivalence.  
Denote $x_i^\circ=[0,x'_i]'\in \mathbb{R}^{k+1}$ and $\tilde{x}_i=[1,x'_i]'\in \mathbb{R}^{k+1} $. For $a \in \{ 0, 1 \}$,  define
\be\label{psidef}
 \psi ( a, \beta)  = 
\begin{bmatrix}
\phi(a,\beta) \\
\gamma(a,\beta)
\end{bmatrix},
\qquad
 \hat \psi ( a, \beta)  = 
\begin{bmatrix}
\hat \phi(a,\beta) \\
\hat \gamma(a,\beta)
\end{bmatrix}.
\ee
We denote their randomization counterparts as  $\psi^* ( a, \beta)$ and ${\hat \psi}^* ( a, \beta)$.
We first document a few simple algebraic facts.
\begin{lemma}\label{lemma14}
The following claims hold. 

\begin{enumerate}
\item We have 
\begin{equation}\label{olsrep}
    \hat \psi (1, \beta) -  \psi (1, \beta)
    = \left(\frac{1}{n_1}\sum_{i\in [n]}Z_i \tilde{x}_i\tilde{x}_i'\right)^{-1}\frac{1}{n_1}\sum_{i\in [n]}Z_i\tilde{x}_i\epsilon_i(1,\beta)
\end{equation} 
and
\begin{equation}
    \hat \psi (0, \beta) -  \psi  (0, \beta)
    = \left(\frac{1}{n_0}\sum_{i\in [n]}(1-Z_i)\tilde{x}_i\tilde{x}_i'\right)^{-1}\frac{1}{n_0}\sum_{i\in [n]}(1-Z_i)\tilde{x}_i\epsilon_i(0,\beta),
\end{equation} 
provided that the matrices $\sum_{i\in [n]}Z_i \tilde{x}_i\tilde{x}_i'$ and $\sum_{i\in [n]}(1-Z_i)\tilde{x}_i\tilde{x}_i'$ are invertible. 
\item For each $a \in \{0,1\}$, 
\begin{equation}\label{orthogonalresiduals}
\sum_{i=1}^n \tilde x_i \eps_i(a, \beta)= 0_{k+1}.
\end{equation}
\end{enumerate}
\end{lemma}
\begin{proof}
The claim \eqref{olsrep} follows from the usual OLS formula.
The equality \eqref{orthogonalresiduals} follows from the well-known fact that the residuals of the OLS solution are orthogonal to the column space of the matrix of regressors. 
\end{proof}

\begin{comment}
\subsubsection{Estimator Representations}

\ber [Check if/where these are used before writing it as a lemma] \eer 

\ber

Some algebra shows:
\begin{align}
    \phi^* = \frac{1}{n}\sum_{i\in [n]}Z_i \left( Y_i(1)-\beta D_i(1)   \right) + \frac{1}{n}\sum_{i\in [n]}\left(1-Z_i\right) \left( Y_i(0)-\beta D_i(0)   \right)= \frac{n_1}{n}\widehat{W^{\beta}(1)} +  \frac{n_0}{n}\widehat{W^{\beta}(0)}.
\end{align}
\begin{align}
    \gamma^* = \left(\frac{1}{n}\sum_{i=1}^n x_ix_i'\right)^{-1} \left(\frac{n_1}{n}\widehat{xW^{\beta}(1)}+\frac{n_0}{n}\widehat{xW^{\beta}(0)}\right) 
\end{align}
\eer 
\end{comment}

\begin{lemma}\label{lemma:asymptoticequi}
Let $(\hat{\tau}_n)_{n=1}^{\infty}$ and $(\Tilde{\tau}_n)_{n=1}^{\infty}$ be two sequences of random variables, and let $(\tau_{0n})_{n=1}^{\infty}$ be a sequence of constants such that the following two conditions hold:
\begin{enumerate}
    \item For every $c\in\mathbb{R}$,
\begin{equation}\label{e:d165}
    \lim_{n\to \infty} \sup_{\theta\in \Xi_n} \left| \P_{\theta}\left( \left|\frac{\Tilde{\tau}_n-\tau_{0n}}{\sqrt{\Var(\Tilde{\tau}_n)}}\right| \geq  c\right)- \P
    \left(|\mathcal{Z}|\geq c\right) \right| = 0,
\end{equation}
where $\mathcal Z$ is a standard normal random variable. 
    \item For every $\epsilon>0$,
\begin{equation}
    \lim_{n\to \infty} \sup_{\theta\in \Xi_n}\P_{\theta} \left( \left|\frac{\Tilde{\tau}_n-\hat{\tau}_n}{\sqrt{\Var(\Tilde{\tau}_n)}}\right| \geq  \epsilon\right) = 0.
\end{equation}
\end{enumerate}
Then
\begin{equation}
\lim_{n\to\infty}\sup_{\theta\in \Xi_n}\left|\P_{\theta}\left( \left|\frac{\hat{\tau}_n-\tau_{0n}}{\sqrt{\Var(\Tilde{\tau}_n)}}\right| \geq c\right) - \P\left(|\mathcal{Z}|\geq c\right)\right|=0.
\end{equation}    
\end{lemma}
\begin{proof}
For every $\delta>0$ and $c\in\mathbb{R}$, we have

\begin{align}
     \P_{\theta}\left( \left|\frac{\hat{\tau}_n-\tau_{0n}}{\sqrt{\Var(\Tilde{\tau}_n)}}\right| \geq c+\delta\right)  &=  \P_{\theta}\left( \left|\frac{\hat{\tau}_n-\tau_{0n}}{\sqrt{\Var(\Tilde{\tau}_n)}}\right|\geq c+\delta, \left|\frac{\hat{\tau}_n-\Tilde{\tau}_n}{\sqrt{\Var(\Tilde{\tau}_n)} }  \right|  \geq \delta\right) \\
    & \quad + \P_{\theta}\left( \left|\frac{\hat{\tau}_n-\tau_{0n}}{\sqrt{\Var(\Tilde{\tau}_n)}}\right|\geq c+\delta, \left|\frac{\hat{\tau}_n-\Tilde{\tau}_n}{\sqrt{\Var(\Tilde{\tau}_n)} }  \right|  <\delta\right)\\
   & \leq  \P_{\theta}\left(  \left|\frac{\hat{\tau}_n-\Tilde{\tau}_n}{\sqrt{\Var(\Tilde{\tau}_n)} }  \right|  \geq \delta\right)+ \P_{\theta}\left( \left|\frac{\tilde{\tau}_n-\tau_{0n}}{\sqrt{\Var(\Tilde{\tau}_n)}}\right|\geq c\right),
\end{align}
and
\begin{align}
    \P_{\theta}\left( \left|\frac{\hat{\tau}_n-\tau_{0n}}{\sqrt{\Var(\Tilde{\tau}_n)}}\right| \geq c\right) & =    \P_{\theta}\left( \left|\frac{\hat{\tau}_n-\tau_{0n}}{\sqrt{\Var(\Tilde{\tau}_n)}}\right|\geq c, \left|\frac{\hat{\tau}_n-\Tilde{\tau}_n}{\sqrt{\Var(\Tilde{\tau}_n)} }  \right|  \geq \delta\right)\\
   &\quad  + \P_{\theta}\left( \left|\frac{\hat{\tau}_n-\tau_{0n}}{\sqrt{\Var(\Tilde{\tau}_n)}}\right|\geq c, \left|\frac{\hat{\tau}_n-\Tilde{\tau}_n}{\sqrt{\Var(\Tilde{\tau}_n)} }  \right|  <\delta\right)\\
   & \leq \P_{\theta}\left( \left|\frac{\hat{\tau}_n-\Tilde{\tau}_n}{\sqrt{\Var(\Tilde{\tau}_n)} }  \right|  \geq \delta\right)+\P_{\theta}\left( \left|\frac{\tilde{\tau}_n-\tau_{0n}}{\sqrt{\Var(\Tilde{\tau}_n)}}\right|\geq c-\delta\right).
\end{align}
Hence we have
\begin{align}
\P_{\theta}\left( \left|\frac{\hat{\tau}_n-\tau_{0n}}{\sqrt{\Var(\Tilde{\tau}_n)}}\right| \geq c\right)  \geq -\P_{\theta}\left(  \left|\frac{\hat{\tau}_n-\Tilde{\tau}_n}{\sqrt{\Var(\Tilde{\tau}_n)} }  \right|  \geq \delta\right) +  \P_{\theta}\left( \left|\frac{\tilde{\tau}_n-\tau_{0n}}{\sqrt{\Var(\Tilde{\tau}_n)}}\right|\geq c+\delta\right)
\end{align}
and therefore
\begin{align}
& \P_{\theta}\left( \left|\frac{\hat{\tau}_n-\tau_{0n}}{\sqrt{\Var(\Tilde{\tau}_n)}}\right| \geq c\right) - \P\left(|\mathcal{Z}|\geq c\right)  \\ & \geq -\P_{\theta}\left(  \left|\frac{\hat{\tau}_n-\Tilde{\tau}_n}{\sqrt{\Var(\Tilde{\tau}_n)} }  \right|  \geq \delta\right) +  \P_{\theta}\left( \left|\frac{\tilde{\tau}_n-\tau_{0n}} {\sqrt{\Var(\Tilde{\tau}_n)}}\right|\geq c+\delta\right)- \P\left(|\mathcal{Z}|\geq c\right)\\    
& =  -\P_{\theta}\left(  \left|\frac{\hat{\tau}_n-\Tilde{\tau}_n}{\sqrt{\Var(\Tilde{\tau}_n)} }  \right|  \geq \delta\right) +  \P_{\theta}\left( \left|\frac{\tilde{\tau}_n-\tau_{0n}} {\sqrt{\Var(\Tilde{\tau}_n)}}\right|\geq c+\delta\right)\\
& \quad -\P\left(|\mathcal{Z}|\geq c+\delta\right)+\P\left(|\mathcal{Z}|\geq c+\delta\right)- \P\left(|\mathcal{Z}|\geq c\right),
\end{align}
and similarly,
\begin{align}
& \P_{\theta}\left( \left|\frac{\hat{\tau}_n-\tau_{0n}}{\sqrt{\Var(\Tilde{\tau}_n)}}\right| \geq c\right) - \P\left(|\mathcal{Z}|\geq c\right)  \\
& \leq \P_{\theta}\left(  \left|\frac{\hat{\tau}_n-\Tilde{\tau}_n}{\sqrt{\Var(\Tilde{\tau}_n)} }  \right|  \geq \delta\right) +  \P_{\theta}\left( \left|\frac{\tilde{\tau}_n-\tau_{0n}} {\sqrt{\Var(\Tilde{\tau}_n)}}\right|\geq c-\delta\right)\\
& \quad -\P\left(|\mathcal{Z}|\geq c-\delta\right)+\P\left(|\mathcal{Z}|\geq c-\delta\right)- \P\left(|\mathcal{Z}|\geq c\right)  .
\end{align}
Collecting both inequalities and taking the limit $n\rightarrow \infty$, we find %\lim_{n\to \infty} \sup_{\theta\in\Xi_n}$
\begin{align*}
    &    \lim_{n\to \infty} \sup_{\theta\in \Xi_n}\left|\P_{\theta}\left( \left|\frac{\hat{\tau}_n-\tau_{0n}}{\sqrt{\Var(\Tilde{\tau}_n)}}\right| \geq c\right) - \P\left(|\mathcal{Z}|\geq c\right)\right|
    \\
    & \leq     \lim_{n\to \infty} \sup_{\theta\in \Xi_n}\P_{\theta}\left(  \left|\frac{\hat{\tau}_n-\Tilde{\tau}_n}{\sqrt{\Var(\Tilde{\tau}_n)} }  \right|  \geq \delta\right)
    \\
    & \quad +   \lim_{n\to \infty} \sup_{\theta\in \Xi_n}\left|\P_{\theta}\left( \left|\frac{\tilde{\tau}_n-\tau_{0n}} {\sqrt{\Var(\Tilde{\tau}_n)}}\right|\geq c+\delta\right)-\P\left(|\mathcal{Z}|\geq c+\delta\right) \right|\\
    & \quad  +   \lim_{n\to \infty} \sup_{\theta\in \Xi_n}\left|\P_{\theta}\left( \left|\frac{\tilde{\tau}_n-\tau_{0n}} {\sqrt{\Var(\Tilde{\tau}_n)}}\right|\geq c-\delta\right) -\P(|\mathcal{Z}|\geq c-\delta) \right|\\
    & \quad + \left|\P\left(|\mathcal{Z}|\geq c+\delta\right)- \P\left(|\mathcal{Z}|\geq c\right) \right| + \left|\P\left(|\mathcal{Z}|\geq c-\delta\right)- \P\left(|\mathcal{Z}|\geq c\right) \right|  .
\end{align*}
Using \eqref{e:d165}, this implies that 
\begin{align} \lim_{n\to \infty} &\sup_{\theta\in\Xi_n}\left|\P_{\theta}\left( \left|\frac{\hat{\tau}_n-\tau_{0n}}{\sqrt{\Var(\Tilde{\tau}_n)}}\right| \geq c\right) - \P\left(|\mathcal{Z}|\geq c\right)\right|\\
    & \leq  \big|\P\left(|\mathcal{Z}|\geq c+\delta\right)- \P\left(|\mathcal{Z}|\geq c\right) \big| + \big|\P(|\mathcal{Z}|\geq c-\delta)- \P\left(|\mathcal{Z}|\geq c\right) \big|.
\end{align}
Since $\delta$ is arbitrary, we can take $\delta\to 0 $ and use the continuity of the CDF of the normal distribution to obtain
\begin{equation}
     \lim_{n\to \infty} \sup_{\mathcal{\theta}\in\Xi_n}\left|\P_{\theta}\left( \left|\frac{\hat{\tau}_n-\tau_{0n}}{\sqrt{\Var(\Tilde{\tau}_n)}}\right| \geq c\right) - \P(|\mathcal{Z}|\geq c)\right|=0.
\end{equation}

\end{proof}

\subsection{Covariate-Adjusted Estimator}

\begin{lemma}\label{lemma:cov_adj_prep}
Consider a sequence $\left(\theta_n\right)_{n=1}^\infty$ such that $\theta_n\in\Xi_n$ and let $\beta_n=\cace(\theta_n)$.  Let $(\theta_{n_k})_{k=1}^\infty$ be any of its subsequences.
The following claims hold.
\begin{enumerate}
    \item We have
    \be\label{ld12}
    \Var\big(\tilde{\tau}_\reg(\beta_{n_k})\big)\geq (1-\delta)\frac{n_0}{n_1n}\mathcal{S}^2_{1}(\beta_{n_k})+ (1-\delta)\frac{n_1}{n_0n}\mathcal{S}^2_{0}(\beta_{n_k}).\ee 
   
    \item Let $\tilde{x}_i=(1,x'_i)'\in \mathbb{R}^{k+1}$. Then
    \begin{equation}\label{xtilde2norm}
        \lim_{k \rightarrow \infty} \frac{1}{n_k}\max_{i\in [n_k]} \|\Tilde{x}_i\|_2^2  =  0.
    \end{equation}

 \item For each $a\in\{0,1\}$,
    \begin{equation}\label{d1ii}
       \frac{  \max_{i\in [n_k]} 
       \epsilon_i(a, \beta_{n_k})^2}{ \frac{1}{n_k}\sum_{i\in [n_k]}\epsilon_i(a, \beta_{n_k})^2} = o(n_k),
    \end{equation}
    and
    \begin{equation}\label{d1ii2}
         \lim_{k \rightarrow \infty}   \frac{1}{n^2_{ak}}   \frac{ \max_{i\in [n_k]} 
       \epsilon_i(a, \beta_{n_k})^2}{\Var(\tilde{\tau}_\reg(\beta_{n_k}))} = 0.
    \end{equation} 
\end{enumerate}
\end{lemma}
\begin{proof}
For simplicity, we map the indices of the subsequence back to $\{1,2,3,\dots\}$. %, using assumption \eqref{iv-covariate} and 
    % the first coordinate of \eqref{orthogonalresiduals}, 
    %\be
   %\sum_{i\in [n]}u_i(a)=0,
  % \ee
  % which follows from the definition of the residuals $u_i$. 
   To prove the first claim, note that assumption \eqref{ii-covariate} implies
    \begin{align}
         & \sum_{i\in [n]}\big(\epsilon_i(1, \beta_n)-\epsilon_i(0, \beta_n)\big)^2  \\
         &= \sum_{i\in [n]} \epsilon_i(1, \beta_n) ^2 + \sum_{i\in [n]}\epsilon_i(1, \beta_n)^2 - 2\sum_{i\in [n]}\epsilon_i(1, \beta_n)\epsilon_i(0, \beta_n)\\
         & \leq \left( 1 +  \frac{\delta n_0}{n_1}\right) \sum_{i\in [n]} \epsilon_i(1, \beta_n)^2 + \left( 1 +  \frac{\delta n_1}{n_0}\right) \sum_{i\in [n]}\epsilon_i(0, \beta_n)^2. 
    \end{align}
Therefore, recalling \eqref{adjvarrep},
\begin{align}
    &(n-1)\Var\big(\tilde{\tau}_\reg(\beta_n)\big)  \\  \quad & = \frac{1}{n_1}\sum_{i\in [n]}\epsilon_i(1, \beta_n)^2 + \frac{1}{n_0}\sum_{i\in [n]}\epsilon_i(0, \beta_n)^2 -\frac{1}{n} \sum_{i\in [n]}\left(\epsilon_i(0, \beta_n)-\epsilon_i(0, \beta_n)\right)^2 \label{ineq_03062024_a} \\
    & \geq \left( \frac{1}{n_1} - \frac{1}{n}   -\frac{\delta n_0}{n_1 n}  \right) \sum_{i\in [n]}\epsilon_i(1, \beta_n)^2  + \left( \frac{1}{n_0} - \frac{1}{n} - \frac{\delta n_1}{n_0 n}  \right) \sum_{i\in [n]}\epsilon_i(0, \beta_n)^2 \\ 
    & = \frac{n_0}{n_1 n}(1 -\delta) \sum_{i\in [n]}\epsilon_i(1, \beta_n)^2 + \frac{n_1}{n_0n }(1 -\delta) \sum_{i\in [n]}\epsilon_i(0, \beta_n)^2,
    \label{ineq_03062024_b}
\end{align}
which is \eqref{ld12}.
 
For \eqref{xtilde2norm}, note that for all $s\in [k]$,  assumption \eqref{vii-covariate} implies that 
\begin{equation}\label{d179}
\max_{i\in[n]} |x_{si}|\leq \left(An \right)^{1/4}, \qquad 
\max_{i\in [n]}\|\tilde{x}_i\|_2^2 = 1 + \sum_{s=1}^k \max_{i\in [n]} |x_{si}|^2  \lesssim n^{1/2},
\end{equation}
since $k$ does not depend on $n$. 
%Using these estimates, the remainder of the proof of this claim is similar to that of \Cref{lemma:max}. 
Claim 
(\ref{d1ii}) can be proved as in Lemma \ref{lemma:max}, and claim (\ref{d1ii2}) can be proved as in Lemma \ref{lemma:lyapunov}.
\end{proof}

The following lemma collects auxiliary results about the estimators $\hat \tau_\reg$ and $\tilde \tau_\reg$.
\begin{lemma}\label{lemma:cov_adj_clt}
    Let $(\theta_n)_{n=1}^\infty$ be a sequence of models with covariates such that $\theta_n \in \Xi_n$ for all $n \in \mathbb{N}$. We denote $\beta_n = \cace(\theta_n)$. Let $(\theta_{n_k})_{k=1}^\infty$ be any of its subsequences.
The following claims hold. 
\begin{enumerate}
    \item For every $\kappa> 0$
    \begin{equation}\label{lemma:cov_asymptotic_equivalence}
        \lim_{k\to \infty}  \P_{\theta_{n_k}}\left( \big| \tilde \tau_\reg(\beta_{n_k})-\hat \tau_\reg(\beta_{n_k}) \big| \geq  \kappa  \sqrt{\Var\big(\tilde \tau_\reg(\beta_{n_k}) \big)} \right) = 0.
    \end{equation}
    \item  For every $x\in\mathbb{R}$,  
    \begin{equation}\label{lemma:cov_adj_clt_ae}
      \lim_{k\to \infty}  \left| \P_{\theta_{n_k}}\left( \left|\frac{\tilde \tau_\reg(\beta_{n_k})}{\sqrt{\Var(\tilde \tau_\reg(\beta_{n_k})})}\right| \geq  x \right)- \P\big( | \mathcal{Z}| \ge x \big) \right| = 0 .
    \end{equation}
    \item For every $x\in\mathbb{R}$,  
    \begin{equation}\label{lemma:cov_adj_clt_hat}
      \lim_{k\to \infty}  \left| \P_{\theta_{n_k}}\left( \left|\frac{ \hat \tau_\reg(\beta_{n_k})  }{\sqrt{\Var(\tilde \tau_\reg(\beta_{n_k})})}\right| \geq  x\right)- \P\big( | \mathcal{Z}| \ge x \big) \right| = 0 .
    \end{equation}
    \item We have
    \begin{equation}\label{lastclaim}
    \lim_{k\rightarrow \infty} 
    \frac{ n_{1k}^{-2} \sum_{i\in \mathcal T(1) }   \hat{\epsilon}_i(1, \beta_{n_k})^2  +  n_{0k}^{-2} \sum_{i\in \mathcal T(0) }  \hat{\epsilon}_i(0 , \beta_{n_k})^2 }{n_{1k}^{-1} \mathcal{S}^2_1(\beta_{n_k})  + n_{0k}^{-1}  \mathcal{S}^2_0(\beta_{n_k}) }   \overset{p}{\to} 1.
    \end{equation}
\end{enumerate}
\end{lemma}

\begin{proof}
For simplicity, we map the indices of the subsequence back to $\{1,2,3,\dots\}$. We begin with \eqref{lemma:cov_asymptotic_equivalence}. By the definitions of $ \hat \tau_\reg(\beta)$ and $\tilde \tau_\reg(\beta_n)$, 
\begin{align*}
     \hat \tau_\reg(\beta_n) - \tilde \tau_\reg(\beta_n) = \frac{1}{n_1}\sum_{i=1}^n Z_i x_i'\big(\hat \gamma(1,\beta_n)- \gamma(1,\beta_n)\big) - \frac{1}{n_0}\sum_{i=1}^n (1-Z_i) x_i'\big(\hat \gamma(0,\beta_n)-\gamma(0,\beta_n)\big).
\end{align*}
Recall the notation $x_i^\circ=[0,x'_i]'\in \mathbb{R}^{k+1}$ and $\tilde{x}_i=[1,x'_i]'\in \mathbb{R}^{k+1} $. Using Chebyshev's inequality and \eqref{olsrep}, the previous expression becomes, with probability approaching one,
\begin{align}
     & \hat \tau_\reg(\beta_n) - \tilde \tau_\reg(\beta_n) \\ & \quad = \frac{1}{n_1}\sum_{i=1}^n Z_i x_i^{\circ'}\left(\frac{1}{n_1}\sum_{i\in [n]}Z_i \tilde{x}_i\tilde{x}_i'\right)^{-1}\frac{1}{n_1}\sum_{i\in [n]}Z_i\tilde{x}_i\epsilon_i(1,\beta_n)
     \\
     & \qquad  -\frac{1}{n_0}\sum_{i=1}^n (1-Z_i) x_i^{\circ '}\left(\frac{1}{n_0}\sum_{i\in [n]}(1-Z_i) \tilde{x}_i\tilde{x}_i'\right)^{-1}\frac{1}{n_0}\sum_{i\in [n]}(1-Z_i)\tilde{x}_i\epsilon_i(0,\beta_n).
\end{align}
By assumptions \eqref{vii-covariate} and \eqref{viii-covariate} and the fact that the number of covariates $k$ does not depend on $n$, we have 
\be \label{apple1}
\frac{1}{n_1}\sum_{i=1}^n Z_i x_i^{\circ '}=O_p\left(\frac{1}{\sqrt{n}}\right), \qquad \frac{1}{n_0}\sum_{i=1}^n (1-Z_i) x_i^{\circ '}=O_p\left(\frac{1}{\sqrt{n}}\right),\ee
\be \label{apple3} \left(\frac{1}{n_1}\sum_{i\in [n]}Z_i \tilde{x}_i\tilde{x}_i'\right)^{-1}=O_p(1),\qquad \left(\frac{1}{n_0}\sum_{i\in [n]}(1-Z_i) \tilde{x}_i\tilde{x}_i'\right)^{-1}=O_p(1).\ee 
These bounds are justified by observing that the random vectors in \eqref{apple1}, and the random matrices being inverted in \eqref{apple3}, concentrate about their mean by the Chebyshev's inequality.

Further, we observe that the random vector 
\be\label{ab12}
\frac{1}{n_1}\sum_{i\in [n]}Z_i\tilde{x}_i\epsilon_i (1,\beta_n)
\ee 
has mean zero, and a straightforward computation using \eqref{orthogonalresiduals}
shows that its covariance matrix is 
\be\label{covmat}
\frac{n_0}{n_1 n }\frac{1}{n-1}\sum_{i\in[n]}\tilde{x}_i\tilde{x}_i'\epsilon^2_i(1, \beta_n).
\ee

Let $||\cdot||_2$ denote the spectral norm. Next, using \eqref{ld12},  then \eqref{d1ii} and \eqref{d179}, we find that
\begin{align}
        &\frac{n_0}{n_1 n } \frac{1}{n-1} \frac{\| \sum_{i\in[n]}\tilde{x}_i\tilde{x}_i'\epsilon_i^2(1, \beta_n)\|_2}{\Var(\tilde{\tau}_{\beta_n})}  \\
        & \quad \leq \frac{n_0}{n_1 n } \frac{1}{(1-\delta)(n-1)}  \left(\frac{ \sum_{i\in[n]}\|\tilde{x}_i\|_2^2\times \max_i\epsilon_i(1, \beta_n)^2}{n_0 (n_1 n )^{-1}\mathcal{S}^2_{1}(\beta_n)+ n_1 (n_0 n )^{-1} \mathcal{S}^2_{0}(\beta_n)} \right)\label{ineq_03062024_c} \\
        %& \quad \leq  \frac{n_0}{n_1 n }\frac{1}{n-1} \frac{ \sum_{i\in[n]}\|\tilde{x}_i\|_2^2\times \max_i\epsilon_i(1, \beta)^2}{(1-\delta)\frac{n_0}{n_1 n }\frac{1}{n-1}\sum_{i\in [n]}\epsilon_i(1, \beta)^2 }\\
        &\quad  \leq  \frac{ \sum_{i\in[n]}\|\tilde{x}_i\|_2^2}{(1-\delta)(n-1)}\frac{ \max_i\epsilon_i(1, \beta_n)^2}{\frac{1}{n-1}\sum_{i\in [n]}\epsilon_i(1, \beta_n)^2 }=o(n).   \label{ineq_03062024_d}
\end{align}
In particular, the diagonal (variance) elements of the covariance matrix \eqref{covmat} can be controlled using the previous display. After using Chebyshev's inequality to show that they concentrate about their respective means, we get
\be\label{banana}
\frac{1}{n_1}\sum_{i\in [n]}Z_i\tilde{x}_i\epsilon_i (1,\beta_n) = o_p \left( 
\sqrt{ n\Var(\tilde{\tau}_{\reg}(\beta_n))} \right).
\ee

%\textcolor{red}{double check variance formula}

Then recalling \eqref{apple1} and \eqref{apple3}, we find,
\begin{equation}
     \frac{1}{n_1\sqrt{\Var(\tilde{\tau}_{\reg}(\beta_n))}}\sum_{i=1}^n Z_i x_i^{\circ '}\left(\frac{1}{n_1}\sum_{i\in [n]}Z_i \tilde{x}_i\tilde{x}_i'\right)^{-1}\frac{1}{n_1}\sum_{i\in [n]}Z_i\tilde{x}_i\epsilon(1,\beta_n)=o_p(1).
\end{equation}
A similar argument shows that
\begin{equation}\label{220}
     \frac{1}{n_0\sqrt{\Var(\tilde{\tau}_{\reg}(\beta_n))}}\sum_{i=1}^n Z_i x_i^{\circ '}\left(\frac{1}{n_0}\sum_{i\in [n]}(1-Z_i) \tilde{x}_i\tilde{x}_i'\right)^{-1}\frac{1}{n_0}\sum_{i\in [n]}(1-Z_i)\tilde{x}_i\epsilon_i(0,\beta_n)=o_p(1),
\end{equation}
This proves the first claim.

For the second claim, in \eqref{lemma:cov_adj_clt_ae}, 
first notice that %\ber typos? also not consistent with definition above \eer 
\begin{align}
     \tilde{\tau}_\reg(\beta_n)&=\frac{1}{n_1}\sum_{i\in [n]}\left(w_i(a,\beta_n)-x'_i\gamma(1,\beta_n)\right)-\frac{1}{n_0}\sum_{i\in [n]}\left(w_i(a,\beta_n)-x_i'\gamma(0,\beta_n)\right)\notag \\
    & =\frac{1}{n_1}\sum_{i\in [n]}\left(w_i(a,\beta_n)-\phi(1, \beta_n)-x_i'\gamma(1,\beta_n)\right)-\frac{1}{n_0}\sum_{i\in [n]}\left(w_i(a,\beta_n)-\phi(0,\beta_n)-x_i'\gamma(0,\beta_n)\right)\notag \\
    &= \frac{1}{n_1}\sum_{i\in[n]}\epsilon_i(1, \beta_n)- \frac{1}{n_0}\sum_{i\in[n]}\epsilon_i(0, \beta_n),
\end{align}
where we use the fact that 
\be
\phi(1, \beta_n)-\phi(0,\beta_n)=\frac{1}{n}\sum_{i\in [n]} \big(y_i(d_i(1))-\beta_n d_i(1)\big)-\frac{1}{n}\sum_{i\in[n]}\big(y_i(d_i(0))-\beta_n d_i(0)\big)=0,
\ee 
which follows from the definition of $\phi(a,\beta_n)$ and our assumption that $\sum_{i=1}^n x_i=0$.
Then \eqref{lemma:cov_adj_clt_ae} follows from Theorem \ref{theorem:clt} and \eqref{d1ii2}. 

The third claim, in \eqref{lemma:cov_adj_clt_hat}, follows from \eqref{lemma:cov_asymptotic_equivalence}, \eqref{lemma:cov_adj_clt_ae} and Lemma \ref{lemma:asymptoticequi}.
To prove the last claim, in \eqref{lastclaim}, first notice that by definition, we have the identities
\begin{equation}
  Z_i \hat{\epsilon}_i(\beta_n)^2 = Z_i \left(Y_i-\beta_n D_i -\hat \phi(1,\beta_n)-x_i'\hat \gamma(1,\beta_n)\right)^2,
\end{equation}
\begin{equation}
    (1-Z_i) \hat{\epsilon}_i(\beta_n)^2 = (1-Z_i) \left(Y_i-\beta_n D_i-\hat \phi(0,\beta_n)-x_i'\hat \gamma(0,\beta_n)\right)^2.
\end{equation}
Then
\begin{align}
& \frac{1}{n_1}\sum_{i\in [n]}  Z_i \hat{\epsilon}_i(\beta_n)^2 - \frac{1}{n_1}\sum_{i\in [n]}  Z_i \epsilon_i(1, \beta_n)^2 \\
& =\frac{1}{n_1} \sum_{i\in [n]}  Z_i \left(Y_i-\beta_n D_i-\hat{\phi}(1,\beta_n)-x_i'\hat{\gamma}(1,\beta_n)\right)^2\\
& \quad - \frac{1}{n_1}\sum_{i\in [n]}  Z_i \left(Y_i-\beta_n D_i-{\phi}(1,\beta_n)-x_i'\gamma(1,\beta_n)\right)^2\\
& = \frac{1}{n_1}\sum_{i\in [n]}  Z_i \left({\phi}(1,\beta_n)+x_i'\gamma(1,\beta_n)-\hat{\phi}(1,\beta_n)
 -x_i'\hat{\gamma}(1,\beta_n)\right)\\ & \quad \times \left(2 \left(Y_i-\beta_n D_i\right) -\hat{\phi}(1,\beta_n)-x_i'\hat \gamma(1,\beta_n)-{\phi}(1,\beta_n)-x_i'\gamma(1,\beta_n)\right)\\
& = 2\left(\frac{1}{n_1}\sum_{i\in [n]} Z_i\epsilon_i(1, \beta_n)\tilde{x_i}'\begin{bmatrix}
    {\phi}(1,\beta_n) - \hat{\phi}(1,\beta_n)\\
    \gamma(1,\beta_n) - \hat{\gamma}(1,\beta_n)
\end{bmatrix}\right) \label{231a}
\\ & \quad +  \frac{1}{n_1} \left(\begin{bmatrix}
    {\phi}(1,\beta_n) - \hat{\phi}(1,\beta_n)\\
    \gamma(1,\beta_n) - \hat{\gamma}(1,\beta_n)
\end{bmatrix}'\sum_{i\in [n]}Z_i \tilde{x_i}\tilde{x_i}'\begin{bmatrix}
    {\phi}(1,\beta_n) - \hat{\phi}(1,\beta_n)\\
    \gamma(1,\beta_n) - \hat{\gamma}(1,\beta_n) 
\end{bmatrix}\right).\label{231}
\end{align}

To bound \eqref{231a} and \eqref{231}, we use \eqref{banana} along with \eqref{olsrep} and \eqref{apple1}. This gives, with probability approaching one,  
\begin{equation}\label{eqn:asymptotic_equivalence_variance1}
\frac{1}{n_1}   \frac{  \left(\sum_{i\in [n]}  Z_i \hat{\epsilon}_i(\beta_n)^2 - \sum_{i\in [n]}  Z_i \epsilon_i(1, \beta_n)^2 \right)}{n_1^{-1}{\mathcal{S}^2_1(\beta_{n})} +n_0^{-1} {\mathcal{S}^2_0(\beta_{n})} }=o_p(n),
\end{equation}
An analogous identity holds for 
\be
\frac{1}{n_0}\sum_{i\in [n]}  \left(1-Z_i\right) \hat{\epsilon}_i(\beta_n)^2 - \frac{1}{n_0}\sum_{i\in [n]}  \left(1-Z_i\right)\epsilon_i(0, \beta_n)^2,
\ee
and we similarly conclude that 
\begin{equation}\label{eqn:asymptotic_equivalence_variance2}
\frac{1}{n_0}  \frac{\left(  \sum_{i\in [n]}  (1-Z_i) \hat{\epsilon}_i(\beta_n)^2 - \sum_{i\in [n]}  (1-Z_i) \epsilon_i(0, \beta_n)^2 \right)}{n_1^{-1}{\mathcal{S}^2_1(\beta_{n})} +n_0^{-1} {\mathcal{S}^2_0(\beta_{n})} }=o_p(n),
\end{equation}
We write
\begin{align}
    & \left(\frac{\mathcal{S}^2_1(\beta_n)}{n_1} +\frac{\mathcal{S}^2_0(\beta_n)}{n_0} \right)^{-1} \left( \frac{1}{n_1^2} \sum_{i\in [n]} Z_i \hat{\epsilon}_i(\beta_n)^2  +  \frac{1}{n_0^2}\sum_{i\in [n]} \left(1-Z_i\right) \hat{\epsilon}_i(\beta_n)^2 \right)   \\
    &=  \left(\frac{\mathcal{S}^2_1(\beta_n)}{n_1} +\frac{\mathcal{S}^2_0(\beta_n)}{n_0} \right)^{-1}  \left(  \frac{1}{n_1^2} \sum_{i\in [n]} Z_i \epsilon_i(1, \beta_n)^2  +  \frac{1}{n_0^2} \sum_{i\in [n]} \left(1-Z_i\right)\epsilon_i(0, \beta_n)^2 \right) \label{eq_03062024_a}\\
     & \quad +  \frac{1}{n_1^2} \left(\frac{\mathcal{S}^2_1(\beta_n)}{n_1} +\frac{\mathcal{S}^2_0(\beta_n)}{n_0} \right)^{-1} \sum_{i\in [n]} Z_i \left(\hat{\epsilon}_i(\beta_n)^2 -  \epsilon_i(1, \beta_n)^2 \right) \label{eq_03062024_b} \\
      &\quad  +  \frac{1}{n_0^2} \left(\frac{\mathcal{S}^2_1(\beta_n)}{n_1} +\frac{\mathcal{S}^2_0(\beta_n)}{n_0} \right)^{-1} \sum_{i\in [n]} (1- Z_i) \left(\hat{\epsilon}_i(\beta_n)^2 -  \epsilon_i(0, \beta_n)^2 \right)  \label{eq_03062024_b2}
\end{align}
Using a similar argument as in the proof of Lemma \ref{lemma:varianceconsistency}, specifically \eqref{b96} and the following lines, it is routine to show that \eqref{eq_03062024_a} converges to 1 in probability. We note that this argument uses \eqref{d1ii}.  The terms in \eqref{eq_03062024_b} and \eqref{eq_03062024_b2} converge to 0 in probability by \eqref{eqn:asymptotic_equivalence_variance1} and \eqref{eqn:asymptotic_equivalence_variance2}. This completes the proof.
\end{proof}

\begin{theorem}
    For every $\alpha\in (0,1)$,
\begin{comment}
    \begin{equation}
  \lim_{n\rightarrow \infty}   
   \inf_{\beta \in \R} 
   \inf_{ \{\theta \in \Xi_n : \cace(\theta)=\beta\}} \P\left(  \hat \tau_\reg(\beta)  > z_{1-\alpha/2} \sqrt{\Var\big( \hat \tau_\reg(\beta)\big)} \right)= \alpha, %\label{theorem:ARstatisticsCLT:statement}
  \end{equation}
and
\end{comment}
\begin{equation}
    \lim_{\kappa \rightarrow 0^+}
    \lim_{n\rightarrow \infty}   
   \sup_{\beta \in \R} 
   \sup_{ \{\theta \in \Theta_n : \cace(\theta)=\beta\}} \P_{\theta}
   \left(\big|\Delta_{\reg}(\beta)\big| > z_{1-\alpha/2} - \kappa \right)\leq \alpha.\label{arstat2}
\end{equation}
\end{theorem}
%\textcolor{red}{Haoge: recall why i proved (188) this way.}
\begin{proof}
Given \Cref{lemma:cov_adj_clt}, the proof of \eqref{arstat2} is nearly identical to the proof of \Cref{thm:clt2}, so we omit it. 
%\Cref{theorem:ARstatisticsCLT}, and the proof of \eqref{arstat2} is similar to \Cref{thm:clt2}, so we omit them both.
%\ber[check with Haoge about this.] \eer
 %   The proof follows similar steps as the proofs of \Cref{theorem:ARstatisticsCLT} and \Cref{thm:clt2}. 
\end{proof}
   
%\end{enumerate}

\subsection{Randomization Distribution}

\begin{lemma}\label{l:d5}
    Let $(\theta_n)_{n=1}^\infty$ be a sequence of models with covariates such that $\theta_n \in \Xi_n$ for all $n \in \mathbb{N}$. Denote $\beta_n = \cace(\theta_n)$. Let $(\theta_{n_k})_{k=1}^\infty$ be any of its subsequences.  
    For each $a\in \{0,1\}$, we have 
    \begin{equation}\label{lemma:asconv_0311a}
       \lim_{k \rightarrow \infty}  \frac{1}{\sigma^2_{\epsilon(a,\beta_{n_k})}}\left| \frac{1}{n_{ak}}\sum_{i\in{\mathcal T(a)}}\hat{\epsilon}^2_i(a,\beta_{n_k}) - \frac{1}{n_k}\sum_{i\in [n_k]}\epsilon_i^2(a,\beta_{n_k})  \right| =  0,
    \end{equation}
\end{lemma}

\begin{proof}
For simplicity, we map the indices of the subsequence back to $\{1,2,3,\dots\}$. We begin with the identity  
\begin{align}
   &  \frac{1}{n_a}\sum_{i\in{\mathcal T(a)}}\hat{\epsilon}^2_i(a,\beta_n) - \frac{1}{n}\sum_{i\in [n]}\epsilon^2_i(a,\beta_n)\\
    & = \frac{1}{n_a}\sum_{i\in{\mathcal T}(a)} \epsilon^2_i(a,\beta_n)  - \frac{1}{n}\sum_{i\in [n]}\epsilon^2_i(a,\beta_n)\\ 
    & \quad +\frac{1}{n_a} \sum_{i\in{\mathcal T(a)}}\big(x_i'(\hat{\gamma}(a,\beta_n)-\gamma(a,\beta_n))+\hat{\phi}(a,\beta_n)-\phi(a,\beta_n) \big)^2\label{236}\\
     &\quad +\frac{2 }{n_a} \sum_{i\in{\mathcal T(a)}} \epsilon_i(a,\beta_n)\big(x_i'(\hat{\gamma}(a,\beta_n)-\gamma(a,\beta_n))+\hat{\phi}(a,\beta_n)-\phi(a,\beta_n) \big) . \label{2555}
\end{align}

By \eqref{vii-covariate} and  \cite[Lemma A3(i)]{wu2021randomization},
\be\label{orange2}
\frac{1}{n_a}\sum_{i\in{\mathcal T(a)}}x_ix_i' = \frac{1}{n}\sum_{i\in [n]} x_ix_i' + o_{a.s.}(1).% = O_{a.s.}(1).
\ee
Hence $n_a^{-1}\sum_{i\in{\mathcal T(a)}}x_ix_i'$ is invertible almost surely, by \eqref{viii-covariate}. 

Recalling the definition of $\tilde{x}_i$, % from the statement of Lemma \ref{lemma:cov_adj_clt},  
we have by a representation analogous to \eqref{olsrep} that, almost surely,  
\begin{align}\label{orange}
\begin{bmatrix}
    \phi(a,\beta_n) - \hat{\phi}(a,\beta_n)\\
    \gamma(a,\beta_n) - \hat{\gamma}(a,\beta_n)
\end{bmatrix} = \left(\frac{1}{n_a}\sum_{i\in{\mathcal T(a)}}\tilde{x}_i\tilde{x}_i'\right)^{-1}\left( \frac{1}{n_a} \sum_{i\in{\mathcal T(a)}}\tilde{x}_i\epsilon(a,\beta_n)\right).   
\end{align}
As a preliminary step, we bound \eqref{orange}. 
Note that for $s\in [k+1]$, we have by the analogue of \eqref{orthogonalresiduals} that 
\begin{equation}\label{carrot}
 \sum_{i\in[n] }\tilde{x}_{si}\epsilon_i(a,\beta_n) =0.
\end{equation}
Further, we note that the sample variance of the collection $(\tilde x_{si} \epsilon_i(a,\beta_n) )_{i=1}^n$ is bounded by the uncentered sample second moment, which is 
\be
\frac{1}{n^2} \sum_{i\in[n]} \big(\tilde{x}_{si}\epsilon_i(a,\beta_n)\big)^2 
\le  \frac{1}{n} \sqrt{\frac{1}{n}\sum_{i\in[n]}\tilde{x}_{si}^4} \times \sqrt{
       \frac{1}{n}\sum_{i\in[n]}\epsilon_i^4(a,\beta_n) }.
\ee
Then from our assumptions \eqref{iv-covariate} and \eqref{vii-covariate}, we get 
\be\label{D.242}
\frac{1}{\sigma^2_{\epsilon(a,\beta_n)} n^2} \sum_{i\in[n]} \big(\tilde{x}_{si}\epsilon_i(a,\beta_n)\big)^2 = O\left( \frac{1}{n} \right).
\ee

By \eqref{orange2}, \eqref{carrot}, \eqref{D.242}, and \cite[Lemma A3(i)]{wu2021randomization}, we conclude that 
\be\label{pineapple}
 \frac{1}{\sigma_{\epsilon(a,\beta_n)} }
\begin{bmatrix}
  \phi(a,\beta_n) - \hat{\phi}(a,\beta_n)\\
    \gamma(a,\beta_n) - \hat{\gamma}(a,\beta_n)
\end{bmatrix} 
= o_{a.s.}(1).
\ee 
We can now bound \eqref{236}. Using \eqref{orange2} and \eqref{pineapple}, we get
\begin{align}
   &\frac{1}{\sigma^{2}_{\epsilon(a,\beta_n)} n_a}\sum_{i\in{\mathcal T(a)}}\left(x_i'\left(\hat{\gamma}(a,\beta_n)-\gamma(a,\beta_n)\right)+\hat{\phi}(a,\beta_n)-\phi(a,\beta_n)\right)^2\\
   &=  \frac{1}{\sigma^{2}_{\epsilon(a,\beta_n)}}\begin{bmatrix}
    \phi(a,\beta_n) - \hat{\phi}(a,\beta_n)\\
    \gamma(a,\beta_n) - \hat{\gamma}(a,\beta_n)
\end{bmatrix}' \left(\frac{1}{n_a}\sum_{i\in{\mathcal T(a)}}\tilde{x_i}\tilde{x_i}'\right)\begin{bmatrix}
    \phi(a,\beta_n) - \hat{\phi}(a,\beta_n)\\
    \gamma(a,\beta_n) - \hat{\gamma}(a,\beta_n)
\end{bmatrix}= o_{a.s.}(1).\label{corn1}
\end{align}
Next, \eqref{2555} is, after applying \eqref{orange},
\[
2 \left( \frac{1}{n_a} \sum_{i\in{\mathcal T(a)}} \epsilon_i(a,\beta_n) \tilde{x}_i'\right)
\left(\frac{1}{n_a}\sum_{i\in{\mathcal T(a)}}\tilde{x}_i\tilde{x}_i'\right)^{-1}\left( \frac{1}{n_a} \sum_{i\in{\mathcal T(a)}}\tilde{x}_i\epsilon_i(a,\beta_n)\right).
\]
Following the  argument used to bound \eqref{orange} to estimate the previous line, we find that
\begin{equation}\label{corn2}
\frac{2 }{\sigma^{2}_{\epsilon(a,\beta_n)} n_a} \sum_{i\in{\mathcal T(a)}} \epsilon_i(a,\beta_n)\big(x_i'(\hat{\gamma}(a,\beta_n)-\gamma(a,\beta_n))+\hat{\phi}(a,\beta_n)-\phi(a,\beta_n) \big) = o_{a.s.}(1).
\end{equation}

Using \cite[Lemma A3(i)]{wu2021randomization} and \eqref{iv-covariate}, we have
\begin{equation}\label{corn3}
    \frac{1}{\sigma^{2}_{\epsilon(a,\beta_n)}n_a}\sum_{i\in{\mathcal T}(a)} \epsilon_i(a,\beta_n)^2  - \frac{1}{\sigma^{2}_{\epsilon(a,\beta_n)}n}\sum_{i\in [n]}\epsilon_i(a,\beta_n)^2 =o_{a.s.}(1).
\end{equation}
Inserting \eqref{corn1}, \eqref{corn2}, and \eqref{corn3} into \eqref{2555}, we find.
\be
\frac{1}{\sigma^{2}_{\epsilon(a,\beta_n)}n_a}\sum_{i\in{\mathcal T(a)}}\hat{\epsilon}_i^2(a,\beta_n) - \frac{1}{n\sigma^{2}_{\epsilon(a,\beta_n)}}\sum_{i\in [n]}\epsilon_i^2(a,\beta_n)=o_{a.s.}(1),%\eqref{lemma:asconv_0311b}) 
\ee
 completes the proof of \eqref{lemma:asconv_0311a}.
\end{proof}
Recall the definition of $\epsilon^*(\beta_n)$ from (\ref{eqn:residual}). Recall the definitions of $\psi$, $\hat \psi$, and $\psi^*$ from \eqref{psidef}. We note the following algebraic identity:
\begin{align}
\begin{split}\label{shake}
       \sum_{i\in [n]}\epsilon^*_i(\beta_n)^2   &= \sum_{a\in \{0,1\}}\sum_{i\in{\mathcal T(a)}} \big( Y_i-\beta_n D_i-\tilde x_i'\psi^*(a,\beta_n)\big)^2\\
       & = \sum_{a\in \{0,1\}}\sum_{i\in{\mathcal T(a)}} \big( Y_i-\beta_n D_i-\tilde x_i'\hat{\psi}(a,\beta_n)\big)^2\\
       & +\sum_{a\in \{0,1\}}\sum_{i\in{\mathcal T(a)}} \big( \tilde x_i'\psi^*(\beta_n)-\tilde x_i'\hat{\psi}(a,\beta_n)\big)^2 \\
      % &= \sum_{a\in \{0,1\}} \sum_{i\in{\mathcal T(a)}}\left(Y_i(a,\beta)-\hat{\phi}(a,\beta)-x_i'\hat{\gamma}(a,\beta)\right)^2 \\ &\quad + \sum_{a\in \{0,1\}}\sum_{i\in{\mathcal T(a)}} \left(x_i'\hat{\gamma}(a,\beta)-x_i'\gamma^*_\mathcal{I}(\beta)+\hat{\phi}(a,\beta)-\phi^*_\mathcal{I}(\beta)\right)^2\\
\end{split}
\end{align}
%%%% In revision, write this more clearly.

%\textcolor{red}{This should be Lemma 9 and Lemma 10 combine. DOUBLE CHECK THIS ONE.}
\begin{lemma}\label{lemma:maxima_randomization_mcesidual} 
    Let $(\theta_n)_{n=1}^\infty$ be a sequence of models with covariates such that $\theta_n \in \Xi_n$ for all $n \in \mathbb{N}$. Denote $\beta_n = \cace(\theta_n)$. Let $(\theta_{n_k})_{k=1}^\infty$ be any of its subsequences. For each $a\in \{0,1\}$, we have
\begin{equation}
\frac{1}{n_{ak}^2 } \frac{\max_{i\in [n_k]} \epsilon^*_i(\beta_{n_k}) ^2}{  \Var^*\big(\tilde{\tau}^*_\reg(\beta_{n_k}) \big) } \overset{a.s.}{\to}0.
\end{equation}
\end{lemma}

\begin{proof}
For simplicity, we map the indices of the subsequence back to $\{1,2,3,\dots\}$. Fix $i \in [n]$. If $Z_i = a$, then %\ber typo in last equation?\eer 
\begin{align}
\begin{split}\label{turnip}
\epsilon_i^*(\beta_n)& = Y-\beta D_i -\phi^*_i(\beta_n) -x_i'\gamma^*(\beta_n)   \\  &=  \epsilon_i(a, \beta_n)  + \phi(a, \beta_n) + x_i'\gamma(a,\beta_n)   -\phi^*_i(a,\beta_n) - x_i'\gamma^*_i(a,\beta_n) \\
& = \epsilon_i(a, \beta_n) + \tilde{x}_i' 
\big( \psi(a,\beta_n)  -  \psi^*(a,\beta_n)
\big).
\end{split}
\end{align}

We obtain
\begin{align}
    &  \frac{1}{n_a^2}
    \frac{\max_i\epsilon_i^*(\beta_n)^2}{\Var^*\big(\tilde{\tau}^*_\reg(\beta_n)\big)}\\
    &  \lesssim \frac{1}{n^2}  \sum_{a\in\{0,1\}} \frac{\max_{i\in [n]}\epsilon^2_i(a,\beta_n)}{\Var^*(\tilde{\tau}^*_\reg(\beta_n))} +  \frac{1}{n^2}  \sum_{a\in\{0,1\}}  \frac{\max_{i\in [n]}\left(
    \tilde{x}_i' 
\big( \psi(a,\beta_n) -  \hat \psi(a,\beta_n) \big)
    \right)^2}{\Var^*(\tilde{\tau}^*_\reg(\beta_n))} \label{peach1}
    \\& \quad + \frac{1}{n^2}  \sum_{a\in\{0,1\}} \frac{\max_{i\in [n]}\left(
    \tilde{x}_i'
\big(
\hat \psi(a, \beta_n) -  \psi^*(\beta_n)
\big)
    \right)^2}{\Var^*(\tilde{\tau}^*_\reg(a,\beta_n))}.\label{peach2}
\end{align}

We begin by bounding the first sum in \eqref{peach1}. We have
\begin{align}
  & \frac{1}{n^2}  \sum_{a\in\{0,1\}} \frac{\max_{i\in [n]}\epsilon_i^2(a,\beta_n)}{\Var^*(\tilde{\tau}^*_\reg(\beta_n))}\leq \frac{1}{n} \sum_{a\in \{0,1\}}  \left(\frac{\max_{i\in [n]}\epsilon_i^2(a,\beta_n)}{\frac{1}{n_a}\sum_{i\in{\mathcal T(a)}} \hat \epsilon^2(a,\beta_n) }\right)\\
  \leq & \frac{1}{n} \sum_{a\in \{0,1\}}  \left(\frac{\max_{i\in [n]}\epsilon_i^2(a,\beta_n)}{\frac{1}{n}\sum_{i\in[n]}\epsilon^2(a,\beta_n)(1+o_{a.s.}(1)) }\right)=o_{a.s.}(1)
\end{align}
by Lemma \ref{l:d5} and \eqref{d1ii}.

Next, by the Cauchy--Schwarz inequality, we conclude that the second sum in \eqref{peach1} satisfies
\begin{align}
&\frac{1}{n}  \sum_{a\in \{0,1\}} 
\frac{\max_{i\in [n]}\|\tilde{x_i}\|^2_2 \times 
\| 
\psi(a,\beta_n)- \hat \psi(a,\beta_n)
\|_2^2
}
%\|  \begin{bmatrix}
%    \phi_y(a)-\hat{\phi}_y(a) \\
%    \gamma_y(a)-\hat{\gamma}_y(a)
%\end{bmatrix}\|_2^2} 
{\sigma^2_{\epsilon(a,\beta_n)}(1+o_{a.s.}(1))}=o_{a.s.}(1)
\end{align}
by \eqref{pineapple} and (\ref{xtilde2norm}).

Finally for \eqref{peach2}, we note that \eqref{shake} implies that
\be\label{cheese2}
\sum_{i\in [n]} \epsilon_i^*(\beta_n)^2  \ge 
\sum_{a \in \{0,1\}}
\sum_{i \in \mathcal T(a) } 
\big[\hat \psi(a,\beta_n) - \psi^*(\beta_n) \big]' \tilde{x}_i\tilde{x}_i'\big[\hat \psi(a,\beta_n) - \psi^*(\beta_n) \big].
\ee
Using the Cauchy--Schwarz inequality in the numerator and \eqref{cheese2} in the denominator, we find that it is bounded by 
\begin{align}
%%%%244 starts
%%%%%%%%%%%%
 \frac{n_1 n_0 (n-1)}{n^3} \sum_{a \in \{0,1\}}   \frac{\max_{i\in [n]}\|\tilde{x_i}\|^2_2 \times  \| \hat \psi(a,\beta_n) -  \psi^*(\beta_n)  \|_2^2 } 
{ \sum_{i \in \mathcal T(a) } 
\big[\hat \psi(a,\beta_n) - \psi^*(\beta_n) \big]' \tilde{x}_i\tilde{x}_i'\big[\hat \psi(a,\beta_n) - \psi^*(\beta_n) \big]}.\label{lastequation}
\end{align}
Recall that \eqref{orange2} states
\be
\frac{1}{n_a}\sum_{i\in{\mathcal T(a)}}\Tilde{x}_i\Tilde{x}_i'=\frac{1}{n}\sum_{i\in [n]}\Tilde{x_i}x_i' + o_{a.s.}(1).
\ee
By assumption \eqref{viii-covariate}, 
\begin{align}
    \big[\hat \psi(a,\beta_n) - \psi^*(\beta_n) \big]'\left( \frac{1}{n_a} \sum_{i\in{\mathcal T(a)}}\tilde{x}_i\tilde{x}_i' \right) &\big[\hat \psi(a,\beta_n) - \psi^*(\beta_n) \big]\\ &\geq \big(\tilde \epsilon + o_{a.s.}(1) \big) \| \hat \psi(a) - \psi^*(\beta_n) \|_2^2 .
\end{align}
Using \eqref{xtilde2norm}, we find that \eqref{lastequation} is bounded by 
\begin{equation}
\frac{1}{n} \max_{i\in [n]}\|\tilde{x_i}\|_2^2 \times O_{a.s.}(1)=o_{a.s.}(1).
\end{equation}
Then all terms in \eqref{peach1} and \eqref{peach2} are $o_{a.s.}(1)$, which completes the proof.
\end{proof}

%\textcolor{red}{This should be similar to Lemma 15. }
\begin{lemma}\label{lemma:clt_cov}
    Let $(\theta_n)_{n=1}^\infty$ be a sequence of models with covariates such that $\theta_n \in \Xi_n$ for all $n \in \mathbb{N}$. Denote $\beta_n = \cace(\theta_n)$. Let $(\theta_{n_k})_{k=1}^\infty$ be any of its subsequences.   The following claims hold almost surely with respect to the randomness in $Z$. 
\begin{enumerate}
    \item For every $\kappa >0$,
    \begin{equation}\label{d61}
    \lim_{k\rightarrow \infty}
    \P^*\left( \big| \hat \tau_\reg^*(\beta_{n_k})-\tilde \tau_\reg^*(\beta_{n_k}) \big|   \ge \kappa \sqrt{\Var^*\big(\tilde{\tau}^*_\reg(\beta_{n_k})\big)} \right) = 0.
    \end{equation}
    \item For every $x > 0 $,
    \be\label{d62}
    \lim_{k\rightarrow \infty}  
    \left|\P^*\left(
    \big| \tilde \tau_\reg^*(\beta_{n_k}) \big|
    \ge 
    x  \sqrt{\Var^*\big(\tilde{\tau}^*_\reg(\beta_{n_k})\big)}
    \right)
    -\P\big( |\mathcal{Z} | \geq x\big) \right|\to 0.
    \ee
       \item For every $x > 0 $,
    \be\label{d63}
    \lim_{k\rightarrow \infty}  \left|
    \P^*\left(
    \big| \hat\tau_\reg^*(\beta_{n_k}) \big|
    \ge 
    x  \sqrt{\Var^*\big(\tilde{\tau}^*_\reg(\beta_{n_k})\big)}
    \right)
    -\P\big( |\mathcal{Z} | \geq x\big) \right|.
    \ee
    \item For every $\kappa >0$,
    \begin{equation}\label{d64} 
        \lim_{k\rightarrow 0} \P^*\left(  \left| \frac{ \hat{\sigma}^2_{\reg, *}(\beta_{n_k}) }{\Var^*\big(\widehat{\tau}_{\beta_{n_k}}^*\big)}-1\right|\geq \kappa \right) = 0 , 
    \end{equation}
\end{enumerate}

\end{lemma}
\begin{proof}
The proof of \eqref{d61} uses \Cref{lemma:maxima_randomization_mcesidual} is similar to the proof of \eqref{lemma:cov_asymptotic_equivalence}, so we omit it. Claim \eqref{d62} is implied by  \Cref{lemma:maxima_randomization_mcesidual} and \Cref{theorem:clt}. Claim \eqref{d63} follows from \eqref{d61} and \eqref{d62}. The proof of \eqref{d64} uses \Cref{lemma:maxima_randomization_mcesidual} and is similar to the proof of \eqref{lastclaim}, so we omit it.
%\textcolor{red}{[haoge, can you fix these references? they don't seem right] Lemma  \ref{lemma:cov_adj_clt}-(i) using Lemma \ref{lemma:maxima_randomizatio\nmcesidual}}. 
\end{proof}

%\textcolor{red}{This should be similar to Theorem 5.}

\begin{comment}
\begin{theorem}\label{theorem:randomizationCV_adj} 
For every $\alpha\in (0,1)$,
\begin{equation}
   \lim_{n\to\infty}  \sup_{\beta_0\in \mathbb{R}} \sup_{ \{\mathcal{ P}^n\in \Theta_N : CATE=\beta_0\}} P(\left|\widehat{c}^{\textnormal{adj}}_{1-\alpha}(\beta_0)-\Chi\right|\geq \epsilon)=0
\end{equation}
\begin{proof}
The proof is similar to that of Theorem \ref{theorem:randomizationCV}.
\end{proof}
\end{comment}

\begin{proof}[Proof of \Cref{thm:coverage2}]
Given \Cref{lemma:clt_cov}, this proof is similar to that of \Cref{t:main}, so we omit the details. 
\end{proof}

\section{Expressions for the regression-adjusted AR statistics}\label{app:reg_AR_expression}
We derive the expressions for the regression-adjusted AR statistics that are similar to (\ref{eqn:AR_sq_ratio}) and (\ref{eqn:AR_sq_ratio_rand}) for the unadjusted AR statistic. 

Denote a column $k$-vector of all one entries by $1_{k}\in\mathbb{R}^{k}$, and denote the $k$-by-$k$ identity matrix as $I_{k}\in\mathbb{R}^{k\times k}$. 
Let $\tilde{X}_1 \in \mathbb{R}^{n_1\times (k+1)}$ and $\tilde{X}_0\in \mathbb{R}^{n_0\times (k+1)}$ be the design matrices (with an intercept) for the treated-group and control-group regressions, respectively. Let $X_1 \in \mathbb{R}^{n_1\times (k+1)}$ and $X_0\in \mathbb{R}^{n_0\times (k+1)}$ be the design matrices for the treated-group and control-group regressions where the intercepts are replaced with zero. Let $Y_1\in\mathbb{R}^{n_1}$ be the column vector of outcomes of the treated group and $Y_0\in\mathbb{R}^{n_0}$ be the column vector of outcomes of the control group, and let $D_1\in\mathbb{R}^{n_1}$ be the column vector of treatment takeups of the treated group and $D_0\in\mathbb{R}^{n_0}$ be the column vector of the treatment takeups of the control group. Note that these quantities are implicit functions of random assignments $\{Z_i\}_{i=1}^n$. 

Define the projection matrices
\begin{align*}
P_1&=X_1\left(\tilde{X}_1'\tilde{X}_1\right)^{-1}\tilde{X}_1'\in \mathbb{R}^{n_1\times n_1},\\
P_0&=X_0\left(\tilde{X}_0'\tilde{X}_0\right)^{-1}\tilde{X}_0'\in \mathbb{R}^{n_0\times n_0},\\
\tilde{P}_1&=\tilde{X}_1\left(\tilde{X}_1'\tilde{X}_1\right)^{-1}\tilde{X}_1'\in \mathbb{R}^{n_1\times n_1},\\
\tilde{P}_0&=\tilde{X}_0\left(\tilde{X}_0'\tilde{X}_0\right)^{-1}\tilde{X}_0'\in \mathbb{R}^{n_0\times n_0}.
\end{align*}

Recall the definition of $\Delta_\reg (\beta)$ from (\ref{adjdelta}). Define
\begin{align}
    t_y &= \frac{1}{n_1}1_{n_1}'(I_{n_1}-P_1)Y_1 - \frac{1}{n_0}1_{n_0}'(I_{n_0}-P_0)Y_0,\\
    t_d &= \frac{1}{n_1}1_{n_1}'(I_{n_1}-P_1)D_1 - \frac{1}{n_0}1_{n_0}'(I_{n_0}-P_0)D_0,\\
    r_y &= \frac{1}{n_1^2}Y_1'(I_{n_1}-\tilde{P}_1)Y_1 +\frac{1}{n_0^2}Y_0'(I_{n_0}-\tilde{P}_0)Y_0,\\
    r_d &= \frac{1}{n_1^2}D_1'(I_{n_1}-\tilde{P}_1)D_1 +\frac{1}{n_0^2}D_0'(I_{n_0}-\tilde{P}_0)D_0,\\
    r_{yd}&= \frac{1}{n_1^2}Y_1'(I_{n_1}-\tilde{P}_1)D_1 + \frac{1}{n_0^2}Y_0'(I_{n_0}-\tilde{P}_0)D_0.
\end{align}

Then we have
\begin{equation}\label{eqn:AR_func_covariate}
    \Delta^2_\reg (\beta) = \frac{a\beta^2+b\beta+c}{d\beta^2+e\beta + f},
\end{equation}
where
\begin{align}
    a = t_d^2, \quad b = -2t_yt_d, \quad c= t_y^2, \quad  d = r_d, \quad e = -2r_{yd}, \quad f= r_y.
\end{align}
The formula for $\Delta^{*2}_\reg (\beta)$ is similar. We have 
\begin{equation}\label{eqn:AR_func_covariate_rand}
    \Delta^{*2}_{\reg,s} (\beta) = \frac{a_s\beta^2+b_s\beta+c_s}{d_s\beta^2+e_s\beta + f_s},
\end{equation}
with the $Z_i$'s are replaced by the simulated assignments $Z_{is}^{*}$'s.

We define
\be\label{eqn:rand_quantile_adj_feasible}
\widehat{\eta}^{\reg,*}_{1-\alpha}(\beta,Z) = \min\left\{ \left|\Delta^*_{\reg,k}(\beta)\right|: 
\frac{1}{\nmc}
\sum_{s=1}^{\nmc} \one \left\{ \big|\Delta^*_{\reg,s}(\beta)\big|\leq \big|\Delta^*_{\reg,k}(\beta)\big| \right\}\geq 1-\alpha\right\},
\ee
and the confidence set of interest is defined as
\begin{equation}\label{eqn:perm_cs_adj}
    \hat I^{\reg,*}_{1-\alpha}(Z) = \big\{ \beta \in \R :  |\Delta_{\reg}(\beta)| \le \hat{\eta}^{\reg,*}_{1-\alpha} (\beta, Z)  \big\}.
\end{equation}
\begin{assumption}\label{assn:positive_variance_adj}
   Let $d\beta^2+e\beta+f$ and $d_k\beta^2+e_k\beta+f_k$ be defined as in (\ref{eqn:AR_func_covariate}) and (\ref{eqn:AR_func_covariate_rand}), respectively. As  functions of $\beta$, they are positive everywhere on $\mathbb{R}$.
\end{assumption}

\begin{algo}\label{alg:1.2}
Constructing Confidence Set (\ref{eqn:perm_cs_adj})
%\vspace*{-12pt}
\begin{tabbing}
   \qquad \enspace \textbf{Require:} Observed data $\{Y_i,D_i,Z_i,x_i\}_{i\in [n]}$, simulated assignments $\{Z_{ik}^*\}_{i\in[n],k\in [\nmc]}$\\
   \qquad \enspace \textbf{Step 1:} For every $s,t\in [\nmc]$ that define distinct AR functions (\ref{eqn:AR_func_covariate}),\\
   \qquad \qquad calculate their intersections\\
   \qquad \enspace \textbf{Step 2:} Sort the distinct intersections on the real line and denote the intervals that result by $\bigl\{[\beta_i, \beta_{i+1}]\bigr\}_{i=0}^{K}$\\
   \qquad \enspace \textbf{Step 3:} For each interval $[\beta_i, \beta_{i+1}]$:\\
   \qquad \qquad If $\beta_i\neq -\infty$ and $\beta_{i+1}\neq +\infty$:\\
   \qquad \qquad \qquad Calculate $\mathcal{I}( \left(\beta_i+\beta_{i+1}\right)/2)$, defined  in (\ref{eqn:quantile_set})\\
   \qquad \qquad \qquad Select one $t_i$ such that $t_i\in\mathcal{I}( \left(\beta_i+\beta_{i+1}\right)/2)$\\
   \qquad \qquad If $\beta_i=\beta_0=-\infty$:\\
   \qquad \qquad \qquad Select any $t_i\in \mathcal{I}(\beta_1-1)$\\
   \qquad \qquad If $\beta_{i+1}=\beta_{K+1}=+\infty$:\\
   \qquad \qquad \qquad Select any $t_i\in\mathcal{I}(\beta_{K}+1)$ \\
   \qquad \enspace \textbf{Step 4:} For each interval $\bigl\{[\beta_i, \beta_{i+1}]\bigr\}_{i=0}^{K}$:\\
   \qquad \qquad Find $\mathrm{CS}_i=\bigl\{\beta: \left|\Delta^2(\beta)\right| \leq \left|\Delta^{*2}_{t_i}(\beta)\right|,\beta\in [\beta_i,\beta_{i+1}] \bigr\}$\\
   \qquad \enspace \textbf{Step 5:} Return $\mathrm{CS}=\bigcup_{i=1}^K \mathrm{CS}_i$   
\end{tabbing}
\end{algo}

% \renewcommand{\thealgorithm}{1.2} 
% \begin{algorithm}[ht]
% \caption{Constructing Confidence Set (\ref{eqn:perm_cs_adj})}
% \label{alg:1.2}
% \begin{algorithmic}
% \Require Observed data $\{Y_i,D_i,Z_i,x_i\}_{i\in [n]}$, simulated assignments $\{Z_{ik}\}_{i\in[n],k\in [\nmc]}$. 
% \State \textbf{Step 1:} For every $s,t\in [\nmc]$ that define distinct AR functions (\ref{eqn:AR_func_covariate}), calculate their intersections.
% \State \textbf{Step 2:} Sort the intersections on the real line and denote the intervals that result by $\bigl\{[\beta_i, \beta_{i+1}]\bigr\}_{i=0}^{K}$.
% \State \textbf{Step 3:} For each interval $[\beta_i, \beta_{i+1}]$, calculate $\mathcal{I}( \left(\beta_i+\beta_{i+1}\right)/2)$, defined  in (\ref{eqn:quantile_set}), and select one $t_i$ such that $t_i\in\mathcal{I}( \left(\beta_i+\beta_{i+1}\right)/2)$. For $\beta_0=-\infty$, we take any $t_i\in \mathcal{I}(\beta_1-0.1)$. For $\beta_{K+1}=\infty$, we take any $t_i\in\mathcal{I}(\beta_{K}+0.1)$.
% \State \textbf{Step 4:} For each interval $\bigl\{[\beta_i, \beta_{i+1}]\bigr\}_{i=0}^{K}$, find 
% \begin{equation}
%    \mathrm{CS}_i=\bigl\{\beta: \left|\Delta^2(\beta)\right| \leq \left|\Delta^{*2}_{t_i}(\beta)\right|,\beta\in [\beta_i,\beta_{i+1}] \bigr\}.
% \end{equation}
% \State \textbf{Step 5:} Return $\mathrm{CS}=\bigcup_{i=1}^K \mathrm{CS}_i$.  
% \end{algorithmic}

% \end{algorithm}

\section{A Faster Algorithm for Confidence Set Construction}\label{appendix:algo}

In this section, we propose alternative algorithms, which are variants of Algorithm \ref{alg:1.1} and Algorithm \ref{alg:1.2}. We note that many intersections calculated in Algorithm \ref{alg:1.1} or Algorithm \ref{alg:1.2} are ``spurious,'' in the sense that they do not represent intersections of the AR functions that realize the $1-\alpha$ quantile at those locations. It is possible to design algorithms that skip such spurious points. This greatly improves the run time. We include the details in Algorithm \ref{alg:2.1} and Algorithm \ref{alg:2.2} below.

Figure \ref{fig12} provides a graphical illustration for the key step. % in the algorithms.
The dashed lines are the AR functions from the simulated assignments. The dot-dashed line is the AR function from the observed data. The solid curve is the quantile function. We use 4 simulated assignments and the 75th percentile for the %graphical 
illustration. The confidence set is the region where the dot-dashed curve lies below the solid curve. The vertical lines highlight the intersections of different simulated AR functions. We rescale the $y$-axis and $x$-axis for a better visualization. 
We first observe that the index that realizes the quantile function remains the same between two consecutive intersections. This fact is a consequence of Lemma \ref{lemma:cs_as_intervals} and was used to construct Algorithms \ref{alg:1.1} and \ref{alg:1.2}. 

%\textcolor{red}{
The key step in Algorithms \ref{alg:2.1} and \ref{alg:2.2} is Step 4. % and Step 4.1. 
%The key steps in Algorithms \ref{alg:2.1} and \ref{alg:2.2} are Step 4.2 and Step 4.1. 
To give an example, we pay attention to the vertical lines labeled as $e_i$, $V_{i}$, and $e_{i+1}$. {We note that $V_{i}$ is the next distinct intersection after $e_i$.} From Step 4 in the algorithms, we find the index $s_{i+1}$ that realizes the quantile function on the interval [$e_i$, $V_{i}$]. The index $s_{i+1}$ that realizes the quantile function will remain the same until the AR function indexed by $s_{i+1}$ intersects another AR function. In our figure, this next intersection is labeled as $e_{i+1}$. Therefore, the index $s_{i+1}$ realizes the quantile function on the interval $[e_{i},e_{i+1}]$. Notice how this step skips a ``spurious'' intersection (the one between $V_{i}$ and $e_{i+1}$). This contrasts with Algorithms \ref{alg:1.1} and \ref{alg:1.2}, which require calculating indices $\mathcal{I}(\left(e_{i}+e_{i+1}\right)/2)$ for each pair of $e_{i}$ and $e_{i+1}$.%}

\begin{algo}\label{alg:2.1}
Constructing Confidence Set (\ref{eqn:perm_cs})
%\vspace*{-12pt}
\begin{tabbing}
   \qquad \enspace \textbf{Require:} Observed data $\{Y_i,D_i,Z_i\}_{i\in [n]}$, simulated assignments $\{Z^*_{ik}\}_{i\in[n],k\in [\nmc]}$\\
    \qquad \enspace \textbf{Step 1:} For every $s,t\in [\nmc]$ that define distinct AR functions. \\
    \qquad \qquad calculate the intersections and collect them into a set $\mathcal J$\\
   \qquad \enspace \textbf{Step 2:} Sort the distinct intersections on the real line and denote the intervals that result by $\bigl\{[\beta_i, \beta_{i+1}]\bigr\}_{i=0}^{K}$\\
    \qquad \enspace \textbf{Step 3}: Set  $e_0=-\infty$ and let $s_i$ be any element in $\mathcal{I}(\beta_1-1)$\\
    \qquad \enspace \textbf{Step 4:} For $i=1,...,\infty$:\\
    \qquad \qquad  Find the set of intersections:\\ \qquad \qquad \qquad $\mathcal{C}_i =  \big\{\beta\in\mathbb{R}: \exists s\in[\nmc]\setminus\{s_i \},\beta> e_{i-1} \text{ such that }\Delta^{2*}_s(\beta)=\Delta_{s_i}^{2*}(\beta) \text{ and }\Delta^{2*}_s,\Delta_{s_i}^{2*}\text{ are distinct} \big\}$\\
    \qquad \qquad If $\mathcal{C}_i$ is empty:\\
    \qquad \qquad \qquad Record $\big([e_{i-1},\infty),s_i\big)$ and \textbf{break}\\
    \qquad \qquad Else:\\
    \qquad \qquad \qquad Set $e_{i}=\min\left(\mathcal{C}_i\right)$ and record $\{[e_{i-1},e_{i}],s_i\}$\\
    \qquad \qquad \qquad Set $\mathcal E_i = \{ j \in \mathcal J, j > e_i \}$\\
     \qquad \qquad \qquad  If $\mathcal E_i$ is empty:\\
     \qquad \qquad \qquad \qquad Set $s_i$ to be any element of $\mathcal I(e_i +1)$\\
    \qquad \qquad \qquad  Else:\\
    \qquad \qquad \qquad \qquad  Set $V_i = \min \left(\mathcal E_i\right)$\\
    \qquad \qquad \qquad \qquad Set $s_{i+1}$ to be any element in  $\mathcal{I}( (e_i+V_i)/2)$\\
    \qquad \enspace \textbf{Step 5:} For each interval--index pair in $([e_{k}, e_{k+1}],s_k)$ recorded in the previous steps:\\
    \qquad \qquad Find $\mathrm{CS}_k=\bigl\{\beta: \left|\Delta^2(\beta)\right| \leq \left|\Delta^{*2}_{s_k}(\beta)\right|,\beta\in [e_{k},e_{k+1}] \bigr\}$\\
    \qquad \enspace \textbf{Step 6:} Return $\mathrm{CS}=\bigcup_{k=0}^{n_I}\mathrm{CS}_k$,  where $n_I$ is the number of times the for loop was executed.
\end{tabbing}
\end{algo}

\begin{algo}\label{alg:2.2}
Constructing Confidence Set (\ref{eqn:perm_cs_adj})
%\vspace*{-12pt}
\begin{tabbing}
   \qquad \enspace \textbf{Require:} Observed data $\{Y_i,D_i,Z_i,x_i\}_{i\in [n]}$, simulated assignments $\{Z^*_{ik}\}_{i\in[n],k\in [\nmc]}$\\
   \qquad \enspace \textbf{Step 1:} For every $s,t\in [\nmc]$ that define distinct AR functions \eqref{eqn:AR_func_covariate},\\
   \qquad \qquad calculate the intersections and collect them into a set $\mathcal J$\\
   \qquad \enspace \textbf{Step 2:} Sort the distinct intersections on the real line and denote the intervals that result by $\bigl\{[\beta_i, \beta_{i+1}]\bigr\}_{i=0}^{K}$\\
   \qquad \enspace \textbf{Step 3:} Set $e_0=-\infty$ and let $s_i$ be any element in $\mathcal{I}(\beta_1-1)$\\
   \qquad \enspace \textbf{Step 4:} For $i=1,...,\infty$:\\
  \qquad \qquad  Find the set of intersections:\\ \qquad \qquad \qquad $\mathcal{C}_i =  \big\{\beta\in\mathbb{R}: \exists s\in[\nmc]\setminus\{s_i \},\beta> e_{i-1} \text{ such that }\Delta^{2*}_s(\beta)=\Delta_{s_i}^{2*}(\beta) \text{ and }\Delta^{2*}_s,\Delta_{s_i}^{2*}\text{ are distinct} \big\}$\\
   \qquad \qquad If $\mathcal{C}_i$ is empty:\\
    \qquad \qquad \qquad Record $\big([e_{i-1},\infty),s_i\big)$ and \textbf{break}\\
    \qquad \qquad Else:\\
    \qquad \qquad \qquad Set $e_{i}=\min\left(\mathcal{C}_i\right)$ and record $\{[e_{i-1},e_{i}],s_i\}$\\
    \qquad \qquad \qquad Set $\mathcal E_i = \{ j \in \mathcal J, j > e_i \}$\\
     \qquad \qquad \qquad  If $\mathcal E_i$ is empty:\\
     \qquad \qquad \qquad \qquad Set $s_i$ to be any element of $\mathcal I(e_i +1)$\\
    \qquad \qquad \qquad  Else:\\
    \qquad \qquad \qquad \qquad  Set $V_i = \min \left(\mathcal E_i\right)$\\
    \qquad \qquad \qquad \qquad Set $s_{i+1}$ to be any element in  $\mathcal{I}( (e_i+V_i)/2)$\\
     \qquad \enspace \textbf{Step 5:} For each interval--index pair in $([e_{k}, e_{k+1}],s_k)$ recorded in the previous steps:\\
     \qquad \qquad Find: $\mathrm{CS}_k=\bigl\{\beta: \left|\Delta^2(\beta)\right| \leq \left|\Delta^{*2}_{s_k}(\beta)\right|,\beta\in [e_{k},e_{k+1}] \bigr\}$\\
     \qquad \enspace \textbf{Step 6:} Return $\mathrm{CS}=\bigcup_{k=0}^{n_I}\mathrm{CS}_k$, where $n_I$ is the number of times the for loop was executed
\end{tabbing}
\end{algo}

\begin{figure}
    \centering
\includegraphics[width=4in,height=3in]{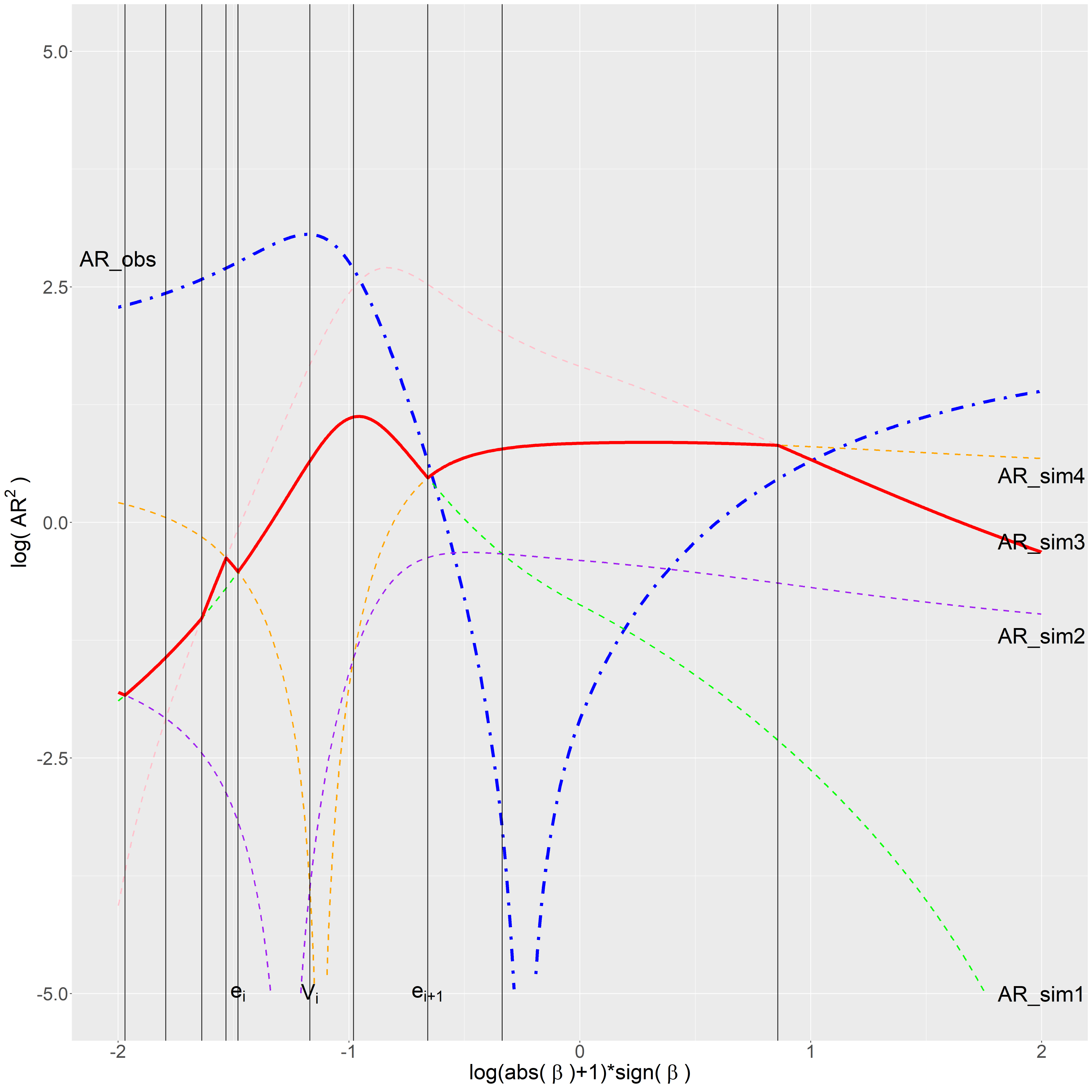}
    \caption{A figure illustrating Algorithm \ref{alg:2.1} and Algorithm \ref{alg:2.2}. The dashed lines are the AR functions from simulated assignments. The dot-dashed line is the AR function from the observed data. The solid curve is the quantile function. We use 4 simulated assignments and the 75th percentile for the illustration. The vertical lines denote the intersections of the simulated AR functions.
    }
    \label{fig12}
\end{figure}

\newpage 
\section{Proof of Theorem \ref{thm:calculation}}\label{section:alg_proof}

We give  proofs for the cases without covariates. The proof for the algorithms with covariates is similar. Let $\mathrm{CS}^*$ denote the confidence set defined in (\ref{eqn:perm_cs}).

We first consider Algorithm \ref{alg:1.1}. With a slight abuse of notation, we mean $(-\infty,\beta_1]$ when we write $[\beta_0,\beta_1]$ with $\beta_0=-\infty$, and similarly we mean $[\beta_K,\infty)$ when we write $[\beta_K,\beta_{K+1}]$ with $\beta_{K+1}=\infty$.
%For Algorithm \ref{alg:1.1}, 
We note that $\bigcup_{i=1}^K [\beta_i,\beta_{i+1}]=\mathbb{R}$ by definition.
Recall that $\mathrm{CS}_i$ is a set returned in Step 4. % of Algorithm \ref{alg:1.1}. 
We have

\begin{align}
       \mathrm{CS}_i & =\bigl\{\beta: \left|\Delta^2(\beta)\right| \leq \left|\Delta^{*2}_{t_i}(\beta)\right|,\beta\in [\beta_i,\beta_{i+1}] \bigr\} \label{eqn:0_11262024}\\
       & = \bigl\{\beta: \left|\Delta^2(\beta)\right| \leq \widehat{\eta}^*_{1-\alpha}(\beta,Z),\beta\in [\beta_i,\beta_{i+1}] \bigr\} \label{eqn:11262024} \\
       & = \bigl\{\beta: \left|\Delta^2(\beta)\right| \leq \widehat{\eta}^*_{1-\alpha}(\beta,Z) \bigr\} \cap  [\beta_i,\beta_{i+1}] \\
       & = \mathrm{CS}^* \cap  [\beta_i,\beta_{i+1}] ,\label{eqn:3_11262024}
\end{align}
where (\ref{eqn:11262024}) is justified by Lemma \ref{lemma:cs_as_intervals}. Hence we have 
\[\bigcup_{i=1}^K  \mathrm{CS}_i  =\bigcup_{i=1}^K  \mathrm{CS}^* \cap  [\beta_i,\beta_{i+1}] = \mathrm{CS}^* \cap \bigcup_{i=1}^K [\beta_i,\beta_{i+1}]=\mathrm{CS}^* \cap \mathbb{R} = \mathrm{CS}^*.\]  This proves the theorem for Algorithm \ref{alg:1.1}. 

We next consider Algorithm \ref{alg:2.1}. Recall that $n_I$ is the number of times that the for loop is executed (which is clearly finite, since $\mathcal J$ is). 
We write $[e_0,e_1]$ to denote $(\infty,e_1]$, 
and $[e_{n_I},e_{n_I+1}]$ for $[e_{n_I},\infty)$.
%and $[e_{n_I},\infty)$ when we write $[e_{n_I},e_{n_I+1}]$ with $e_{n_I+1}=\infty$.
 %For Algorithm \ref{alg:2.1}, we first notice that it will halt because $\mathcal{V}_i$ always increases in the for loop and the set of intersections is finite. Also note, $e_{n_{I}+1}=\infty$ by definition. Hence we 
 With these notations, we have  $\bigcup_{i=0}^{n_I} [e_i,e_{i+1}]=\mathbb{R}$. 

%\textcolor{red}{Denote the next distinct intersection after $e_k$ as $\mathcal{V}_{k+1}$.} 
We have by construction that $[e_{k},e_{k+1}]=[e_{k},\mathcal{V}_{k}]\cup [\mathcal{V}_{k},e_{k+1}]$. We have 
$s_k\in \mathcal{I}(\beta)$ for all $\beta \in [e_{k},\mathcal{V}_{k}]$ by construction and Lemma $\ref{lemma:cs_as_intervals}$. On $(\mathcal{V}_{k},e_{k+1})$, there is no distinct AR function that intersects the AR function that realizes the $1-\alpha$ quantile.  Hence $s_k\in \mathcal{I}(\beta)$ for all $\beta \in [\mathcal{V}_{k},e_{k+1})$ and by continuity $s_k \in I(e_{k+1})$ if $e_{k+1}<\infty$. 
By the same argument as in (\ref{eqn:0_11262024})--(\ref{eqn:3_11262024}),  the set $\mathrm{CS}_i$ returned in Step 5 of Algorithm \ref{alg:2.1} satisfies $\mathrm{CS}_i=\mathrm{CS}^*\cap [e_{i+1},e_i]$, and hence $\bigcup_{i=0}^{n_I}\mathrm{CS}_i=\mathrm{CS}^*$.

\section{Algebra for Remark \ref{remark2}}\label{appendix:remark}

We first show that (\ref{beta_wald}) satisfies $\widehat{\tau}_{\reg}(\widehat{\beta}^{\reg}_{\textnormal{Wald}})=0$ and, consequently, $\Delta_{\reg}(\widehat{\beta}^{\reg}_{\textnormal{Wald}})=0$, provided that the estimator $\widehat{\beta}^{\reg}_{\textnormal{Wald}}$ is well defined (i.e., its denominator is not zero).

Recall the definition of $\widehat{\gamma}(1,\beta)$ and $\widehat{\gamma}(0,\beta)$ introduced below equation (\ref{olsme}). 
By the linearity of the OLS estimator, they can be rewritten as:
\begin{align}
\widehat{\gamma}(1,\beta) = \widehat{\tau}_{y,1} - \beta \widehat{\tau}_{d,1}, \qquad 
    &\widehat{\gamma}(0,\beta) = \widehat{\tau}_{y,0} - \beta \widehat{\tau}_{d,0},
\end{align}
where $\widehat{\tau}_{y,1}$, $\widehat{\tau}_{y,0}$, $\widehat{\tau}_{d,1}$, and $\widehat{\tau}_{y,0}$ were introduced in Remark \ref{remark2}.

Note that $\widehat{\tau}(\beta)$ can be rewritten as:
\begin{align}
   \widehat{\tau}_{\reg}(\beta)= & \frac{1}{n_1}\sum_{i\in [n]}Z_i\big(Y_i-\beta D_i  -  x_i'\hat \gamma(1, \beta)\big)  \\  & - \frac{1}{n_0}\sum_{i\in [n]}(1-Z_i)\big(Y_i-\beta D_i  - x_i'\hat\gamma(0, \beta)\big), \\
    = & \frac{1}{n_1}\sum_{i\in [n]}Z_i\left(Y_i-x_i'\widehat{\tau}_{y,1}\right) - \frac{1}{n_0}\sum_{i\in [n]}\left(1-Z_i\right)\left(Y_i-x_i'\widehat{\tau}_{y,0}\right)\\
   & - \beta \left(\frac{1}{n_1}\sum_{i\in [n]}Z_i\left(D_i-x_i'\widehat{\tau}_{d,1}\right) - \frac{1}{n_0}\sum_{i\in [n]}\left(1-Z_i\right)\left(D_i-x_i'\widehat{\tau}_{d,0}\right)\right) \label{eqn:I.301_12222024}.
\end{align}
It is then obvious that $\tau_{\reg}(\widehat{\beta}^\reg_{\textnormal{Wald}})=0$ provided that the term \[\frac{1}{n_1}\sum_{i\in [n]}Z_i\left(D_i-x_i'\widehat{\tau}_{d,1}\right) - \frac{1}{n_0}\sum_{i\in [n]}\left(1-Z_i\right)\left(D_i-x_i'\widehat{\tau}_{d,0}\right)\] is nonzero.

We now show that the estimator $\widehat{\beta}^\reg_{\textnormal{Wald}}$ is  algebraically equivalent to the coefficient for $D_i$ in the instrumental variable (IV) regression \begin{equation}
    Y_i\sim \alpha_0 + \alpha_1 D_i + \gamma_{1} Z_ix_i + \gamma_{0}\left(1-Z_i\right)x_i,  
\end{equation}
with the instruments $1, Z_i, Z_ix_i$, and $ \left(1-Z_i\right)x_i$, when both estimators are well defined (i.e., the denominator is not zero and the design matrix is invertible).  

For notational simplicity, we define the column vectors\footnote{We changed the order of the variables for a simplified block representation in (\ref{eqn:I.307_1222024}). } 
\begin{equation}
    \tilde{X}_i= \left(D_i,1,Z_ix_i', \left(1-Z_i\right)x_i'\right)'
\end{equation}
and
\begin{equation}
    \tilde{W}_i=\left(Z_i,1,Z_ix_i', \left(1-Z_i\right)x_i'\right)'.
\end{equation}
We assume that $\sum_{i\in [n]}\tilde{W_i}\tilde{X_i}'$ and $\sum_{i\in [n]}\tilde{W_i}\tilde{W_i}'$ are invertible so that the operations below are well-defined. 

The just-identified IV estimator can be represented as:
\begin{equation}
\begin{bmatrix}
    \widehat{\alpha}_{1,\textnormal{IV}}\\
    \widehat{\alpha}_{0,\textnormal{IV}}\\ \widehat{\gamma}_{1,\textnormal{IV}}\\
    \widehat{\gamma}_{0,\textnormal{IV}} 
\end{bmatrix} =
    \left(\frac{1}{n}\sum_{i\in [n]}\tilde{W_i}\tilde{X_i}'\right)^{-1} \frac{1}{n}\sum_{i\in [n]}\tilde{W_i}Y_i .\label{eqn:I.305_20222024}
\end{equation}
We can further represent the estimator as
\begin{align}
& \left(\frac{1}{n}\sum_{i\in [n]}\tilde{W_i}\tilde{X_i}'\right)^{-1} \frac{1}{n}\sum_{i\in [n]}\tilde{W_i}\tilde{W_i}' \left(\frac{1}{n}\sum_{i\in [n]}\tilde{W_i}\tilde{W_i}'\right)^{-1}\frac{1}{n}\sum_{i\in [n]}\tilde{W_i}Y_i. 
\end{align}
We note that the first-stage regression can be written as
\begin{align}\label{eqn:I.307_1222024}
  \left(\frac{1}{n}\sum_{i\in [n]}\tilde{W_i}\tilde{W_i}'\right)^{-1}\frac{1}{n}\sum_{i\in [n]}\tilde{W_i}\tilde{X}_i' = 
\begin{pNiceArray}{c|ccc}
 \widehat{\alpha}_{d,1} & 0 & \hdots & 0 \\
 \hline 
 \widehat{\alpha}_{d,0} &  \Block{3-3}{I_{2k+1}} \\
 \widehat{\gamma}_{d,1} \\
 \widehat{\gamma}_{d,0} 
\end{pNiceArray},
\end{align}
where $I_{2k+1}$ is the $(2k+1)\times (2k+1)$ identity matrix, and $ \widehat{\alpha}_{d,1}$,  $\widehat{\alpha}_{d,0}$, $\widehat{\gamma}_{d,1}$, and $\widehat{\gamma}_{d,0}$ are the OLS estimators from the corresponding regression:
\begin{equation}
    D_i \sim \alpha_{d,1} Z_i +  \alpha_{d,0} +Z_ix_i'\gamma_{d,1} + \left(1-Z_i\right)x_i'\gamma_{d,0}.
\end{equation}
The reduced-form regression estimator can be written as
\begin{equation}
    \left(\frac{1}{n}\sum_{i\in [n]}\tilde{W_i}\tilde{W_i}'\right)^{-1}\frac{1}{n}\sum_{i\in [n]}\tilde{W_i}Y_i=\begin{bmatrix}
        \widehat{\alpha}_{y,1}\\ \widehat{\alpha}_{y,0} \\
         \widehat{\gamma}_{y,1}\\
          \widehat{\gamma}_{y,0}
    \end{bmatrix},
\end{equation}
where $ \widehat{\alpha}_{y,1}$,  $\widehat{\alpha}_{y,0}$, $\widehat{\gamma}_{y,1}$, and $\widehat{\gamma}_{y,0}$ are the OLS estimators from the corresponding regression:
\begin{equation}
    Y_i \sim \alpha_{y,1} Z_i +  \alpha_{y,0} +Z_ix_i'\gamma_{y,1} + \left(1-Z_i\right)x_i'\gamma_{y,0}.
\end{equation}
Equation (\ref{eqn:I.305_20222024}) implies that the IV estimator satisfies the identity
\begin{equation}
    \begin{pNiceArray}{c|ccc}
 \widehat{\alpha}_{d,1} & 0 &\hdots & 0 \\
 \hline 
 \widehat{\alpha}_{d,0} &  \Block{3-3}{I_{2k+1}} \\
 \widehat{\gamma}_{d,1} \\
 \widehat{\gamma}_{d,0}, 
\end{pNiceArray}\begin{bmatrix}
    \widehat{\alpha}_{1,\textnormal{IV}}\\
    \widehat{\alpha}_{0,\textnormal{IV}}\\ \widehat{\gamma}_{1,\textnormal{IV}}\\
    \widehat{\gamma}_{0,\textnormal{IV}} 
\end{bmatrix} = \begin{bmatrix}
        \widehat{\alpha}_{y,1}\\ \widehat{\alpha}_{y,0} \\
         \widehat{\gamma}_{y,1}\\
          \widehat{\gamma}_{y,0},
    \end{bmatrix}.
\end{equation}
The first row of the matrix identity implies
\begin{equation}
    \widehat{\alpha}_{d,1}  \widehat{\alpha}_{1,\textnormal{IV}} =  \widehat{\alpha}_{y,1}.
\end{equation}
By the OLS representation in \cite{lin2013agnostic}, we have:
\begin{equation}
     \widehat{\alpha}_{y,1} = \frac{1}{n_1}\sum_{i\in [n]}Z_i\left(Y_i-x_i'\widehat{\gamma}_{y,1}\right) - \frac{1}{n_0}\sum_{i\in [n]}\left(1-Z_i\right)\left(Y_i-x_i'\widehat{\gamma}_{y,0}\right)
\end{equation}
and
\begin{equation}
     \widehat{\alpha}_{d,1} = \frac{1}{n_1}\sum_{i\in [n]}Z_i\left(D_i-x_i'\widehat{\gamma}_{d,1}\right) - \frac{1}{n_0}\sum_{i\in [n]}\left(1-Z_i\right)\left(D_i-x_i'\widehat{\gamma}_{d,0}\right).
\end{equation}
Hence, provided that $ \widehat{\alpha}_{d,1}$ is nonzero,
\begin{equation}
\widehat{\alpha}_{1,\textnormal{IV}}=\frac{\frac{1}{n_1}\sum_{i\in [n]}Z_i\left(Y_i-x_i'\widehat{\gamma}_{y,1}\right) - \frac{1}{n_0}\sum_{i\in [n]}\left(1-Z_i\right)\left(Y_i-x_i'\widehat{\gamma}_{y,0}\right)}{ \frac{1}{n_1}\sum_{i\in [n]}Z_i\left(D_i-x_i'\widehat{\gamma}_{d,1}\right) - \frac{1}{n_0}\sum_{i\in [n]}\left(1-Z_i\right)\left(D_i-x_i'\widehat{\gamma}_{d,0}\right)}=\widehat{\beta}^{\reg}_{\textnormal{Wald}},
\end{equation}
as desired.

\bibliographystyle{abbrvnat}
\bibliography{IV}

\begin{thebibliography}{37}
\providecommand{\natexlab}[1]{#1}
\providecommand{\url}[1]{\texttt{#1}}
\expandafter\ifx\csname urlstyle\endcsname\relax
  \providecommand{\doi}[1]{doi: #1}\else
  \providecommand{\doi}{doi: \begingroup \urlstyle{rm}\Url}\fi

\bibitem[Anderson and Rubin(1949)]{anderson1949estimation}
T.~W. Anderson and H.~Rubin.
\newblock Estimation of the parameters of a single equation in a complete system of stochastic equations.
\newblock \emph{The Annals of Mathematical Statistics}, 20\penalty0 (1):\penalty0 46--63, 1949.

\bibitem[Andrews et~al.(2011)Andrews, Cheng, and Guggenberger]{andrews2011generic}
D.~W.~K. Andrews, X.~Cheng, and P.~Guggenberger.
\newblock Generic results for establishing the asymptotic size of confidence sets and tests, 2011.
\newblock Cowles Foundation Discussion Paper.

\bibitem[Andrews et~al.(2019)Andrews, Stock, and Sun]{andrews2019weak}
I.~Andrews, J.~Stock, and L.~Sun.
\newblock Weak instruments in instrumental variables regression: Theory and practice.
\newblock \emph{Annual Review of Economics}, 11:\penalty0 727--753, 2019.

\bibitem[Angrist and Koles{\'a}r(2024)]{angrist2024one}
J.~Angrist and M.~Koles{\'a}r.
\newblock One instrument to rule them all: The bias and coverage of just-id iv.
\newblock \emph{Journal of Econometrics}, 240\penalty0 (2):\penalty0 105398, 2024.

\bibitem[Angrist and Pischke(2009)]{angrist2009mostly}
J.~D. Angrist and J.-S. Pischke.
\newblock \emph{Mostly Harmless Econometrics: An Empiricist's Companion}.
\newblock Princeton university press, 2009.

\bibitem[Ansel et~al.(2018)Ansel, Hong, and Li]{ansel2018ols}
J.~Ansel, H.~Hong, and J.~Li.
\newblock {OLS} and {2SLS} in randomized and conditionally randomized experiments.
\newblock \emph{Jahrb{\"u}cher f{\"u}r National{\"o}konomie und Statistik}, 238\penalty0 (3-4):\penalty0 243--293, 2018.

\bibitem[Bai et~al.(2023)Bai, Guo, Shaikh, and Tabord-Meehan]{bai2023inference}
Y.~Bai, H.~Guo, A.~M. Shaikh, and M.~Tabord-Meehan.
\newblock Inference in experiments with matched pairs and imperfect compliance.
\newblock \emph{arXiv preprint arXiv:2307.13094}, 2023.

\bibitem[Bound et~al.(1995)Bound, Jaeger, and Baker]{bound1995problems}
J.~Bound, D.~A. Jaeger, and R.~M. Baker.
\newblock Problems with instrumental variables estimation when the correlation between the instruments and the endogenous explanatory variable is weak.
\newblock \emph{Journal of the American Statistical Association}, 90\penalty0 (430):\penalty0 443--450, 1995.

\bibitem[Bugni and Gao(2023)]{bugni2023inference}
F.~A. Bugni and M.~Gao.
\newblock Inference under covariate-adaptive randomization with imperfect compliance.
\newblock \emph{Journal of Econometrics}, 237\penalty0 (1):\penalty0 105497, 2023.

\bibitem[Chung and Romano(2013)]{chung2013exact}
E.~Chung and J.~P. Romano.
\newblock Exact and asymptotically robust permutation tests.
\newblock \emph{The Annals of Statistics}, 41\penalty0 (2):\penalty0 484--507, 2013.

\bibitem[Cohen and Fogarty(2022)]{cohen2022gaussian}
P.~L. Cohen and C.~B. Fogarty.
\newblock Gaussian prepivoting for finite population causal inference.
\newblock \emph{Journal of the Royal Statistical Society Series B: Statistical Methodology}, 84\penalty0 (2):\penalty0 295--320, 2022.

\bibitem[Dufour and Taamouti(2005)]{dufour2005projection}
J.-M. Dufour and M.~Taamouti.
\newblock Projection-based statistical inference in linear structural models with possibly weak instruments.
\newblock \emph{Econometrica}, 73\penalty0 (4):\penalty0 1351--1365, 2005.

\bibitem[Green et~al.(2003)Green, Gerber, and Nickerson]{green2003getting}
D.~P. Green, A.~S. Gerber, and D.~W. Nickerson.
\newblock Getting out the vote in local elections: Results from six door-to-door canvassing experiments.
\newblock \emph{The Journal of Politics}, 65\penalty0 (4):\penalty0 1083--1096, 2003.

\bibitem[Imbens and Angrist(1994)]{LATE1994}
G.~Imbens and J.~Angrist.
\newblock Identification and estimation of local average treatment effects.
\newblock \emph{Econometrica}, 62\penalty0 (2):\penalty0 467--475, 1994.
\newblock ISSN 00129682, 14680262.
\newblock URL \url{http://www.jstor.org/stable/2951620}.

\bibitem[Imbens and Rosenbaum(2005)]{imbens2005robust}
G.~Imbens and P.~R. Rosenbaum.
\newblock Robust, accurate confidence intervals with a weak instrument: quarter of birth and education.
\newblock \emph{Journal of the Royal Statistical Society Series A: Statistics in Society}, 168\penalty0 (1):\penalty0 109--126, 2005.

\bibitem[Imbens and Rubin(2015)]{imbens2015causal}
G.~W. Imbens and D.~B. Rubin.
\newblock \emph{Causal Inference for Statistics, Social, and Biomedical Sciences}.
\newblock Cambridge University Press, 2015.

\bibitem[Kang et~al.(2018)Kang, Peck, and Keele]{kang2018inference}
H.~Kang, L.~Peck, and L.~Keele.
\newblock Inference for instrumental variables: a randomization inference approach.
\newblock \emph{Journal of the Royal Statistical Society Series A: Statistics in Society}, 181\penalty0 (4):\penalty0 1231--1254, 2018.

\bibitem[Keane and Neal(2021)]{keane2021practical}
M.~P. Keane and T.~Neal.
\newblock A practical guide to weak instruments.
\newblock \emph{Annual Review of Economics}, 16, 2021.

\bibitem[Keele et~al.(2017)Keele, Small, and Grieve]{keele2017randomization}
L.~Keele, D.~Small, and R.~Grieve.
\newblock Randomization-based instrumental variables methods for binary outcomes with an application to the ‘{IMPROVE}’ trial.
\newblock \emph{Journal of the Royal Statistical Society Series A: Statistics in Society}, 180\penalty0 (2):\penalty0 569--586, 2017.

\bibitem[Li and Ding(2017)]{li2017general}
X.~Li and P.~Ding.
\newblock General forms of finite population central limit theorems with applications to causal inference.
\newblock \emph{Journal of the American Statistical Association}, 112\penalty0 (520):\penalty0 1759--1769, 2017.

\bibitem[Lin(2013)]{lin2013agnostic}
W.~Lin.
\newblock Agnostic notes on regression adjustments to experimental data: Reexamining {F}reedman’s critique.
\newblock \emph{The Annals of Applied Statistics}, 7\penalty0 (1):\penalty0 295--318, 2013.

\bibitem[Maddala and Jeong(1992)]{maddala1992exact}
G.~S. Maddala and J.~Jeong.
\newblock On the exact small sample distribution of the instrumental variable estimator.
\newblock \emph{Econometrica: Journal of the Econometric Society}, pages 181--183, 1992.

\bibitem[Mikusheva(2010)]{mikusheva2010robust}
A.~Mikusheva.
\newblock Robust confidence sets in the presence of weak instruments.
\newblock \emph{Journal of Econometrics}, 157\penalty0 (2):\penalty0 236--247, 2010.

\bibitem[Mikusheva and Poi(2006)]{mikusheva2006tests}
A.~Mikusheva and B.~P. Poi.
\newblock Tests and confidence sets with correct size when instruments are potentially weak.
\newblock \emph{The Stata Journal}, 6\penalty0 (3):\penalty0 335--347, 2006.

\bibitem[Moreira(2003)]{moreira2003conditional}
M.~J. Moreira.
\newblock A conditional likelihood ratio test for structural models.
\newblock \emph{Econometrica}, 71\penalty0 (4):\penalty0 1027--1048, 2003.

\bibitem[Moreira(2009)]{moreira2009tests}
M.~J. Moreira.
\newblock Tests with correct size when instruments can be arbitrarily weak.
\newblock \emph{Journal of Econometrics}, 152\penalty0 (2):\penalty0 131--140, 2009.

\bibitem[Nelson and Startz(1988)]{nelson1988distribution}
C.~Nelson and R.~Startz.
\newblock The distribution of the instrumental variables estimator and its t-ratiowhen the instrument is a poor one, 1988.

\bibitem[Nelson and Startz(1990)]{nelson1990}
C.~R. Nelson and R.~Startz.
\newblock Some further results on the exact small sample properties of the instrumental variable estimator.
\newblock \emph{Econometrica}, 58\penalty0 (4):\penalty0 967--976, 1990.
\newblock ISSN 00129682, 14680262.
\newblock URL \url{http://www.jstor.org/stable/2938359}.

\bibitem[Nolen and Hudgens(2011)]{nolen2011randomization}
T.~L. Nolen and M.~G. Hudgens.
\newblock Randomization-based inference within principal strata.
\newblock \emph{Journal of the American Statistical Association}, 106\penalty0 (494):\penalty0 581--593, 2011.

\bibitem[Ren(2024)]{ren2024model}
J.~Ren.
\newblock Model-assisted complier average treatment effect estimates in randomized experiments with noncompliance.
\newblock \emph{Journal of Business \& Economic Statistics}, 42\penalty0 (2):\penalty0 707--718, 2024.

\bibitem[Staiger and Stock(1997)]{staigerstock97}
D.~Staiger and J.~H. Stock.
\newblock Instrumental variables regression with weak instruments.
\newblock \emph{Econometrica}, 65\penalty0 (3):\penalty0 557--586, 1997.
\newblock ISSN 00129682, 14680262.
\newblock URL \url{http://www.jstor.org/stable/2171753}.

\bibitem[Stock et~al.(2002)Stock, Wright, and Yogo]{stock2002survey}
J.~Stock, J.~Wright, and M.~Yogo.
\newblock A survey of weak instruments and weak identification in generalized method of moments.
\newblock \emph{Journal of Business \& Economic Statistics}, 20\penalty0 (4):\penalty0 518--529, 2002.

\bibitem[Tuvaandorj(2024)]{tuvaandorj2024robust}
P.~Tuvaandorj.
\newblock Robust permutation tests in linear instrumental variables regression, 2024.
\newblock Forthcoming in \emph{Journal of the American Statistical Association}.

\bibitem[Van~der Vaart(2000)]{van2000asymptotic}
A.~Van~der Vaart.
\newblock \emph{Asymptotic Statistics}, volume~3.
\newblock Cambridge University Press, 2000.

\bibitem[Wu and Ding(2021)]{wu2021randomization}
J.~Wu and P.~Ding.
\newblock Randomization tests for weak null hypotheses in randomized experiments.
\newblock \emph{Journal of the American Statistical Association}, 116\penalty0 (536):\penalty0 1898--1913, 2021.

\bibitem[Zhao and Ding(2021)]{zhao2021covariate}
A.~Zhao and P.~Ding.
\newblock Covariate-adjusted {F}isher randomization tests for the average treatment effect.
\newblock \emph{Journal of Econometrics}, 225\penalty0 (2):\penalty0 278--294, 2021.

\bibitem[Zhong et~al.(2023)Zhong, Johansson, and Zhang]{zhong2023inference}
Z.~Zhong, P.~Johansson, and J.~L. Zhang.
\newblock Inference of sample complier average causal effects in completely randomized experiments.
\newblock \emph{arXiv preprint arXiv:2311.17476}, 2023.

\end{thebibliography}

\end{document}